\documentclass[10pt]{amsart}
\usepackage{amssymb, enumitem}
\usepackage[all]{xy}
\usepackage{hyperref, aliascnt}
\usepackage[english]{babel}
\usepackage{enumitem}

\def\today{\number\day\space\ifcase\month\or   January\or February\or
   March\or April\or May\or June\or   July\or August\or September\or
   October\or November\or December\fi\   \number\year}

\theoremstyle{definition}
\newtheorem{lma}{Lemma}[section]

\newaliascnt{thmCt}{lma}
\newtheorem{thm}[thmCt]{Theorem}
\aliascntresetthe{thmCt}

\newaliascnt{corCt}{lma}
\newtheorem{cor}[corCt]{Corollary}
\aliascntresetthe{corCt}

\newaliascnt{propCt}{lma}
\newtheorem{prop}[propCt]{Proposition}
\aliascntresetthe{propCt}

\newtheorem*{thm*}{Theorem}
\newtheorem*{cor*}{Corollary}
\newtheorem*{prop*}{Proposition}

\newcounter{theoremintro}
\newtheorem{thmintro}[theoremintro]{Theorem}
\newtheorem{corintro}[theoremintro]{Corollary}

\newaliascnt{pgrCt}{lma}

\aliascntresetthe{pgrCt}

\newaliascnt{dfCt}{lma}
\newtheorem{df}[dfCt]{Definition}
\aliascntresetthe{dfCt}

\newaliascnt{remCt}{lma}
\newtheorem{rem}[remCt]{Remark}
\aliascntresetthe{remCt}

\newaliascnt{remsCt}{lma}

\aliascntresetthe{remsCt}

\newaliascnt{egCt}{lma}
\newtheorem{eg}[egCt]{Example}
\aliascntresetthe{egCt}

\newaliascnt{egsCt}{lma}

\aliascntresetthe{egsCt}

\newaliascnt{qstCt}{lma}

\aliascntresetthe{qstCt}

\newaliascnt{pbmCt}{lma}

\aliascntresetthe{pbmCt}

\newaliascnt{notaCt}{lma}
\newtheorem{nota}[notaCt]{Notation}
\aliascntresetthe{notaCt}

\newcommand{\beq}{\begin{equation}}
\newcommand{\eeq}{\end{equation}}
\newcommand{\beqa}{\begin{eqnarray*}}
\newcommand{\eeqa}{\end{eqnarray*}}
\newcommand{\bal}{\begin{align*}}
\newcommand{\eal}{\end{align*}}
\newcommand{\bi}{\begin{itemize}}
\newcommand{\ei}{\end{itemize}}
\newcommand{\be}{\begin{enumerate}}
\newcommand{\ee}{\end{enumerate}}

\newcommand{\ep}{\varepsilon}

\newcommand{\Z}{{\mathbb{Z}}}
\newcommand{\C}{{\mathbb{C}}}
\newcommand{\N}{{\mathbb{N}}}

\newcommand{\Hi}{{\mathcal{H}}}

\newcommand{\M}{{\mathcal{M}}}
\newcommand{\MoR}{{\mathcal{M}\overline{\otimes}\mathcal{R}}}

\newcommand{\R}{{\mathcal{R}}}

\newcommand{\B}{{\mathcal{B}}}
\newcommand{\U}{{\mathcal{U}}}

\newcommand{\id}{{\mathrm{id}}}

\newcommand{\spec}{{\mathrm{sp}}}

\newcommand{\diag}{{\mathrm{diag}}}
\newcommand{\supp}{{\mathrm{supp}}}

\newcommand{\Aut}{{\mathrm{Aut}}}

\newcommand{\Ad}{{\mathrm{Ad}}}

\newcommand{\dimRok}{{\mathrm{dim}_\mathrm{Rok}}}

\newcommand{\ca}{$C^*$-algebra}
\newcommand{\cas}{$C^*$-algebras}
\newcommand{\uca}{unital $C^*$-algebra}

\newcommand{\wtRp}{weak tracial Rokhlin property}

\newcommand{\Rp}{Rokhlin property}

\newcommand{\I}{\infty}

\title[]{Strongly outer actions of amenable groups on $\mathcal{Z}$-stable $C^*$-algebras}

\date{\today}

\thanks{The first named author was partially supported
by the Deutsche Forschungsgemeinschaft (SFB 878), and by a Postdoctoral Research Fellowship
from the Humboldt Foundation. This work was supported by GIF grant 1137/2011,
 by Israel Science Foundation Grant 476/16 and by the Fields Institute.}

\author[Eusebio Gardella]{Eusebio Gardella}
\address{Eusebio Gardella
Mathematisches Institut, Fachbereich Mathematik und Informatik der
Universit\"at M\"unster, Einsteinstrasse 62, 48149 M\"unster, Germany.}
\email{gardella@uni-muenster.de}
\urladdr{www.math.uni-muenster.de/u/gardella/}
\author{Ilan Hirshberg}
\address{Ilan Hirshberg
Department of Mathematics, Ben-Gurion University of the Negev, Be’er
Sheva 8410501, Israel.}
\email{ilan@math.bgu.ac.il}
\urladdr{www.math.bgu.ac.il/~ilan/}
\begin{document}

\begin{abstract}
Let $A$ be a separable, unital, simple, $\mathcal{Z}$-stable, nuclear $C^*$-algebra,
and let $\alpha\colon G\to \Aut(A)$ be an action of a countable amenable group. If the trace space $T(A)$ is a Bauer simplex and
the action of $G$ on $\partial_eT(A)$ has finite orbits and Hausdorff orbit space, we show that the following are equivalent:
\be\item $\alpha$ is strongly outer;
\item $\alpha\otimes\id_{\mathcal{Z}}$ has the weak tracial Rokhlin property.\ee
If $G$ is moreover residually finite, the above conditions are also equivalent to
\be\item[(3)] $\alpha\otimes\id_{\mathcal{Z}}$ has finite Rokhlin dimension (in fact, at most 2).\ee

When the covering dimension of $\partial_eT(A)$ is finite, we prove that
$\alpha$ is cocycle conjugate to $\alpha\otimes\id_{\mathcal{Z}}$. In particular, the equivalences above
hold for $\alpha$ in place of $\alpha\otimes\id_{\mathcal{Z}}$.
%
\end{abstract}

\maketitle

\tableofcontents

\renewcommand*{\thetheoremintro}{\Alph{theoremintro}}
\section{Introduction}

The Rokhlin property and its various generalizations form a collection of regularity properties for group actions on \ca s,
whose roots stem from the Rokhlin lemma in Ergodic Theory.
Early
works include the studies of cyclic group actions on UHF-algebras by
Herman and Jones, and Herman and Ocneanu, and later
for automorphisms by Kishimoto.
Although the Rokhlin property is relatively common for actions of the integers, there are significant $K$-theoretic obstructions
for the Rokhlin property for finite group actions (and hence actions of groups which have torsion). This was studied in depth by Izumi \cite{Izu_finiteI_2004,
Izu_finiteII_2004} and spurred additional work.
 One obstruction is
that the Rokhlin property, at least for finite groups, implies certain divisibility properties
on $K$-theory. Attempts to circumvent impediments of this sort led Phillips to introduce the
\emph{tracial Rokhlin property} \cite{Phi_tracial_2011},
where the projections in the Rokhlin property are now assumed to have a left over which
is small in trace.
Among other applications,
the tracial Rokhlin property has been used by Echterhoff, L\"uck, Phillips, and Walters
\cite{EchLucPhiWal_structure_2010} to study fixed point
algebras of the irrational rotation algebra $A_\theta$ under certain canonical actions of finite cyclic groups.

The tracial Rokhlin property does not bypass the most obvious obstruction to
admitting Rokhlin actions: the existence of nontrivial projections. For example, the
Jiang-Su algebra has no nontrivial projections, and hence
does not admit any action of a nontrivial group with the tracial Rokhlin property. The need to study
weaker versions of these properties led to two further generalizations. The first
one, called the
\emph{weak tracial Rokhlin property}, which replaces projections with positive elements, has been considered by
the second author and Orovitz \cite{HirOro_tracially_2013}, Sato \cite{Sat_rohlin_2010}, and
Matui and Sato \cite{MatSat_stability_2012}.

A different approach was taken in a paper by the second author, Winter, and Zacharias \cite{HirWinZac_rokhlin_2015},
which introduced the notion of \emph{Rokhlin dimension}. In this formulation,
the partition of unity appearing in the Rokhlin property is replaced by a multi-tower partition
of unity consisting of positive contractions, the elements of each tower being indexed by the group elements and
permuted by the group action. Rokhlin dimension zero then corresponds to the Rokhlin property, but the extra
flexibility makes the property of having finite Rokhlin dimension a much more common one.
For example, for actions of the integers, Rokhlin dimension one turns out to be generic for actions of
 $\mathcal{Z}$-absorbing separable \ca s.
 Finite Rokhlin dimension is used as a tool to show that various structural
 properties of interest pass from an algebra to the crossed product, particularly finiteness of the nuclear dimension,
 and absorption of a strongly self-absorbing \ca.
Rokhlin dimension has been defined and studied for actions of various classes of groups:
for residually finite groups by Szab{\'o}, Wu, and Zacharias \cite{SzaWuZac_rokhlin_2014}; for compact groups by
the first author \cite{Gar_rokhlin_2017, Gar_regularity_2017}, the second author and Phillips \cite{HirPhi_rokhlin_2015},
and further by the authors and Santiago
\cite{GarHirSan_rokhlin_2017}; and for flows by the second author, Szab{\'o}, Winter, and Wu
\cite{HirSzaWinWu_rokhlin_2017}. This notion has also been explored for
quantum group actions in \cite{GarKalLup_model_2019}.

The focus of the study of the various Rokhlin-type properties naturally centered on two extreme cases: either actions on commutative \ca s
or actions on simple \ca s. For instance, the results of \cite{SzaWuZac_rokhlin_2014,HirSzaWinWu_rokhlin_2017} focused on showing that
actions on commutative \ca s of various groups have finite Rokhlin dimension provided the Gelfand spectrum is finite dimensional and the
induced action on the spectrum is free. In the simple case, in \cite{Lia_rokhlin_2016,Lia_rokhlin_2017} Liao showed that for actions of
$\Z^m$ on simple nuclear separable unital and $\mathcal{Z}$-absorbing \ca s, whose trace simplex is a Bauer simplex with finite dimensional
boundary and is fixed by the action, strong outerness\footnote{An action $\alpha\colon G\to\Aut(A)$ of a discrete group $G$ on a \ca\ $A$ is said to be \emph{strongly outer}
if for every $g\in G\setminus\{e\}$ and every $\tau\in T(A)^{\alpha_g}$, the weak extension of $\alpha_g$ in the GNS
representation of $A$ associated to $\tau$ is outer.} is equivalent to finite Rokhlin dimension. Liao's argument does not work for finite
group actions, or for non-finitely generated groups, and part of the motivation for this paper was to find a suitable generalization
for Liao's theorem to general amenable groups.

This work focuses on actions on simple \ca s. We study the relationships between strong outerness, the weak tracial Rokhlin property and
finite Rokhlin dimension by showing that they are equivalent in many cases of interest.
More specifically, we obtain the following main results.

\begin{thmintro}\label{thmintro:Equiv}
Let $G$ be a countable amenable group, let $A$ be a separable, simple, $\mathcal{Z}$-stable,
nuclear \uca, and let $\alpha\colon G\to\Aut(A)$ be an action.
Suppose that $T(A)$ is a nonempty Bauer simplex, and that the induced action of $G$ on
$\partial_eT(A)$ has finite orbits and Hausdorff orbit space.
Then the following conditions are equivalent:
\be\item $\alpha$ is strongly outer;
\item $\alpha\otimes\id_{\mathcal{Z}}$ has the weak tracial Rokhlin property.\ee
If $G$ is furthermore residually finite, the above conditions are also equivalent to
\be\item[(3)] $\alpha\otimes\id_{\mathcal{Z}}$ has finite Rokhlin dimension.
\item[(4)] $\alpha\otimes\id_{\mathcal{Z}}$ has Rokhlin dimension at most 2.\ee

\end{thmintro}

The result above generalizes and extends a number of works by several authors, and constitutes an
important step towards developing a classification theory for strongly outer actions of
amenable groups on classifiable \ca s. Observe that all the tracial assumptions are automatically
satisfied if either $\partial_eT(A)$ is finite; or if $\alpha$ fixes all traces; or if $G$ is finite.

\begin{thmintro}\label{thmintro:Zabs}
Let $G$ be a countable amenable group, let $A$ be a separable, simple, $\mathcal{Z}$-stable, nuclear \uca, and let $\alpha\colon G\to\Aut(A)$ be an action.
Suppose that $\partial_eT(A)$ is nonempty, compact and finite dimensional, and that the
induced action of $G$ on $\partial_eT(A)$ has finite orbits and Hausdorff orbit space.
Then $\alpha$ is cocycle conjugate to $\alpha\otimes\id_{\mathcal{Z}}$.
\end{thmintro}

In particular, since the weak tracial Rokhlin property and Rokhlin dimension are cocycle conjugacy invariants, we
deduce that in the context of Theorem~\ref{thmintro:Equiv}, if $\dim_{\mathrm{cov}}(\partial_eT(A))<\I$, then $\alpha\otimes\id_{\mathcal{Z}}$
can be replaced by $\alpha$. That is we have:

\begin{corintro}\label{corintro:FiniteG}
Let $G$ be an amenable countable residually finite group, let $A$ be a separable, simple, $\mathcal{Z}$-stable, nuclear \uca, and let $\alpha\colon G\to\Aut(A)$ be an action.
Suppose that $T(A)$ is a nonempty Bauer simplex, that $\dim_{\mathrm{cov}}(\partial_eT(A))<\I$  and that the
induced action of $G$ on $\partial_eT(A)$ has finite orbits and Hausdorff orbit space.
Then the following conditions are equivalent:
\be
\item $\alpha$ is strongly outer;
\item $\alpha$ has the weak tracial Rokhlin property;
\item $\alpha$ has finite Rokhlin dimension (in fact, at most 2).
\ee
\end{corintro}

\textbf{Acknowledgements:} The first named author is thankful to Nate Brown, Hung-Chang Liao, Martino Lupini,
Hannes Thiel, Qingyun Wang, and Stuart White for helpful conversations. We thank G\'abor Szab\'o and the anonymous referee for reading a previous 
version of this paper and making various helpful comments.
This work was initiated while the first named author was visiting the second in October 2016, and part of
it was completed while the authors were participating in the workshop ``Future targets in the classification program for
amenable C*-algebras'', held at the BIRS Centre, Banff, in September 2017.

\section{Absorption of McDuff actions}

In this section, we isolate and study a particular class of discrete group actions on the hyperfinite II$_1$-factor $\R$,
which we call \emph{McDuff actions}; see \autoref{df:McDuffaction}. This class contains all actions obtained as infinite
tensor products of finite-dimensional unitary representations, as well as certain Bernoulli (sub-)shifts. Moreover, just like
for the hyperfinite II$_1$-factor, absorption of a McDuff action can be neatly characterized in terms of equivariant
embeddings into central sequence algebras; see \autoref{thm:W*BdleAbsR}. Using this characterization, we show a pair of
very useful results concerning McDuff absorption of equivariant $W^*$-bundles: in \autoref{thm:McDuffFromFiberToBundle}, we
show that if an equivariant $W^*$-bundle with finite dimensional base has fibers that absorb a fixed McDuff action, then
the bundle itself absorbs this action as well. The bundles that we obtain in some of our applications do not have finite
dimensional base, so we also prove the following variant: if an equivariant $W^*$-bundle has fibers that absorb a fixed McDuff
action, then the $(\R,\id_{\R})$-stabilization of the bundle absorbs this action as well; see \autoref{thm:MoRabsorbsMcDuff}.
In other words, finite-dimensionality of the base can be dispensed of at the cost of adding an equivariant copy of $\R$.

The results in this section, which are also of independent interest, are fundamental tools that will be used in the remainder
of this work.

We begin with the main definition of this section.

\begin{df}\label{df:McDuffaction}
Let $G$ be a discrete group and let $\delta\colon G\to\Aut(\R)$ be an action. We say that $\delta$ is \emph{strongly McDuff} if
there exist an equivariant isomorphism
\[\varphi\colon (\R,\delta)\to (\R\overline{\otimes}\R,\delta\otimes\delta)\]
and unitaries $(w_n)_{n\in\N}$ in $\R\overline{\otimes}\R$ satisfying
\[\lim_{n\to \I} \|w_n\varphi(x)w_n^*-1\otimes x\|_2=0 = \lim_{n\to \I} \|(\delta_g\otimes\delta_g)(w_n)-w_n\|_2\]
for all $x\in \R$ and for all $g\in G$. In other words, the equivariant isomorphism $\varphi$ is \emph{$G$-equivariantly
approximately unitarily equivalent} to the second tensor factor embedding $\R\to \R\overline{\otimes}\R$.

We say that $\delta$ is \emph{McDuff} if it is cocycle equivalent to a strongly
McDuff action.
\end{df}

The trivial action on $\R$ is clearly McDuff, as is any Bernoulli shift $\beta\colon G\to\Aut(\overline{\otimes}_{g\in G}\R)$ of an amenable group $G$. 


Actions as in the following proposition will be relevant to our work. The proof that such product type actions
are strongly McDuff is inspired by similar results in the context of strongly self-absorbing \ca s and actions;
see \cite{TomWin_strongly_2007} and \cite{Sza_strongly_2018}. 

\begin{prop}\label{eg:deltanu}
Let $G$ be a discrete group, let $\nu\colon G\to \U(M_{d})$ be a unitary representation with $d>1$. 
Identify $\R$ with the weak closure of $\bigotimes\limits_{n\in\N}M_{d}$
in the GNS representation associated to its unique trace.
We define an action $\delta^{\nu}\colon G\to \Aut(\R)$ by setting, for $g\in G$,
\[\delta^{\nu}_g=\bigotimes\limits_{n=1}^\I \Ad(\nu_g).\]
Then $\delta^\nu$ is strongly McDuff.
\end{prop}
\begin{proof}
We begin by showing that $\delta^\nu\otimes\delta^\nu$ has what may be called approximately inner $G$-equivariant flip
(in the tracial sense).
For $g\in G$ and $m\in\N$, we set 
\[\mu^{(m)}_g=\nu_g\otimes \cdots\otimes \nu_g\in M_{d}\otimes\cdots\otimes M_{d}\subseteq \mathcal{R}.\]
Let $(w_n)_{n\in\N}$ be a sequence of unitaries in $\mathcal{R}\overline{\otimes}\mathcal{R}$ satisfying
\[\lim_{n\to\I}\|w_n(x\otimes y)w_n^*-y\otimes x\|_2=0\]
for all $x,y\in\mathcal{R}$. Given $g\in G$ and $m\in\N$, we have 
\[\lim_{n\to \I} \|\Ad(\mu^{(m)}_g\otimes \mu^{(m)}_g)(w_n)-w_n\|_2=\lim_{n\to\I}\|w_n(\mu^{(m)}_g\otimes \mu^{(m)}_g)w_n^*-\mu^{(m)}_g\otimes \mu^{(m)}_g\|_2=0,\]
and thus $\lim\limits_{n\to\I}\|(\delta^\nu_g\otimes\delta^\nu_g)(w_n)-w_n\|_2=0$. 

Fix an equivariant isomorphism
\[\varphi\colon (\R,\delta^\nu) \to (\R\overline{\otimes}\R,\delta^\nu\otimes\delta^\nu).\]
(Note that such an isomorphism exists, since it may be obtained by rearranging the matricial tensor factors
of $\R\overline{\otimes}\R$.) It suffices to show that there is a $G$-equivariant approximate unitary 
equivalence between $\varphi$ and the first factor embedding.

Fix a free ultrafilter $\omega$ over $\N$. 
Since $(\R,\delta^\nu)$ is in fact equivariantly isomorphic to its inifinite tensor
product, we may choose a unital, equivariant embedding
\[\rho\colon (\R,\delta^\nu)\to (\R^\omega\cap \R',(\delta^\nu)^\omega),\]
whose image commutes with $\varphi(\R)\cup \iota(\R)$. 

Let $w\in (\R\overline{\otimes}\R)^\omega$ denote the unitary determined by $(w_n)_{n\in\N}$, 
which is fixed by $(\delta^\nu\otimes\delta^\nu)^\omega$. We denote by 
$\kappa_\omega$ the canonical embedding
of $\R\overline{\otimes}\R$ into its ultrapower. We abbreviate $\kappa_\omega\circ\varphi$ to $\varphi_\omega$,
and similarly $\kappa_\omega\circ\iota$ to $\iota_\omega$. Thus, there are well-defined equivariant
unital homomorphisms $\varphi_\omega\otimes \rho$ and $\iota_\omega\otimes\rho$ from $\R\overline{\otimes}\R$ into
its own ultrapower. Then 
\[\varphi_\omega=(\varphi_\omega\otimes\rho)\circ\iota \ \ \mbox{ and } \ \ 
\iota_\omega=(\iota_\omega\otimes\rho)\circ\iota.\]
Moreover, since $(w_n)_{n\in\N}$ implements the flip on $\mathcal{R}\overline{\otimes}\mathcal{R}$, we also have 
\[\Ad(w)\circ (\varphi_\omega\otimes\rho)=\rho\otimes\varphi_\omega \ \  \mbox{ and } \ \  \Ad(w)\circ (\iota_\omega\otimes\rho)=\rho\otimes\iota_\omega.\]
Hence,
\[\Ad(w)\circ \varphi_\omega=\Ad(w)\circ (\varphi_\omega\otimes\rho)\circ\iota=(\rho\otimes\varphi_\omega)\circ\iota=\rho\circ\iota.
\]
Likewise, $\Ad(w)\circ \iota_\omega=\rho$. It follows that $\iota_\omega$ and $\varphi_\omega$ are 
unitarily equivalent via a $G$-invariant unitary. A standard reindexation argument then shows that 
$\iota$ and $\varphi$ are $G$-equivariantly approximately unitarily equivalent, showing that $\delta^\nu$ is strongly McDuff.
\end{proof}

\begin{eg}\label{eg:FiniteGModelStrMcDuff}
Let $G$ be a finite group, and let $\mu_G\colon G\to\Aut(\mathcal{R})$ be the
unique (up to conjugacy) outer action of $G$ on $\mathcal{R}$;
see \cite{Jon_finite_1980}. Then 
$\mu_G$ is strongly McDuff by \autoref{eg:deltanu}, 
since it can be realized as $\delta^\lambda$, for 
the left-regular representation $\lambda\colon G\to\U(\ell^2(G))$.
\end{eg}

Our next goal is to prove a characterization of absorption of a McDuff action in terms
of central sequence algebras (\autoref{thm:W*BdleAbsR}),
which resembles McDuff's characterization of absorption of $\mathcal{R}$.
We need the result for equivariant $W^*$-bundles, which are the equivariant 
version of the notion introduced by Ozawa in Section~5
of~\cite{Oza_dixmier_2013}. We define these first.

\begin{df}\label{df:EqWbundle}
Let $G$ be a discrete group, and let $K$ be a compact metrizable space. An \emph{equivariant $W^*$-bundle
over $K$} is a quadruple $(\M, K, E,\gamma)$, where:
\be\item $\M$ is a \ca;
\item there is a given unital inclusion of $C(K)$ into the center of $\M$;\footnote{The choice of inclusion is part of the definition,
but we lighten the notation by omitting it throughout.}
\item $E\colon \M\to C(K)$ is a faithful conditional expectation satisfying $E(ab)=E(ba)$ for all $a,b\in\M$; 
\item the norm-closed unit ball of $\M$ is complete in the \emph{uniform 2-norm} defined by
$\|a\|_{2,u}=\|E(a^*a)\|^{1/2}$ for all $a\in \M$; and
\item $\gamma\colon G\to \Aut(\M)$ is an action satisfying $\gamma_g\circ E=E=E\circ\gamma_g$ for all $g\in G$.
\ee

We say that the bundle is \emph{strictly separable} if $\M$ contains a countable subset which is dense
in the uniform 2-norm.

For $\lambda\in K$, define a tracial state $\tau_\lambda$ on $\M$ by $\tau_\lambda=\mathrm{ev}_\lambda\circ E$.
If $\pi_\lambda\colon \M\to \B(L^2(\M,\tau_\lambda))$ denotes the associated GNS representation, we call the
image $\M_\lambda$ of $\pi_\lambda$ the \emph{fiber} of $\M$ over $\lambda$. It is easy to check that $\gamma$ induces an
action $\gamma^{\lambda}\colon G\to\Aut(\M_\lambda)$ which makes $\pi_\lambda$ equivariant.
\end{df}

The space $K$ and the conditional expectation $E$ are often suppressed from the notation, and we will often simply say
that $\M$ is a $W^*$-bundle and that $\gamma\colon G\to\Aut(\M)$ is a fiber-wise action.

The motivation for considering $W^*$-bundles is given by the following example,
which is due to Ozawa~\cite{Oza_dixmier_2013} (page 351).
Recall that a Choquet simplex is said to be \emph{Bauer} if its extreme boundary
is compact.

\begin{eg}\label{eg:W*bundle}
Let $A$ be a unital, simple, separable \ca\ for which $T(A)$ is a nonempty Bauer simplex.
We write $A_1$ for the unit ball of $A$. Define the \emph{uniform 2-norm}
$\|\cdot\|_{2,u}$ on $A$ by $\|a\|_{2,u}=\sup_{\tau\in T(A)}\tau(a^*a)^{1/2}$ for all $a\in A$.
Set $K=\partial_eT(A)$.  Set $\overline{A}^u$ to be the completion of $A$ in the uniform 2-norm, that is, the $C^*$-algebra of norm-bounded uniform 2-norm Cauchy sequences, modulo the ideal of sequences which converge to zero in the uniform 2-norm. 

Then $\overline{A}^u$ has a natural structure of a
$W^*$-bundle over $K$ with conditional expectation $E\colon \overline{A}^u\to C(K)$ determined by
$E(a)(\tau)=\tau(a)$ for all $a\in A$ and $\tau\in K$. Moreover, the fiber of $\overline{A}^u$
over $\tau\in K$ can be canonically identified with the weak closure $\overline{A}^\tau$ of the image of $A$
under the GNS representation $\pi_\tau$ associated to $\tau$.
\end{eg}

Equivariant $W^*$-bundles can be constructed from certain actions of \cas, as follows.

\begin{eg}\label{eg:EqW*bundle}
Adopt the notation from \autoref{eg:W*bundle}, and assume that $T(A)$ is a nonempty Bauer simplex. Let $G$ be a discrete
group, and let $\alpha\colon G\to\Aut(A)$ be an action. Assume that the orbit of every $\tau\in\partial_eT(A)$
under $\alpha$ is finite, and that the orbit space $\partial_eT(A)/G$ is Hausdorff\footnote{This is automatically the case if
the action of $G$ on $K$ factors through a finite subgroup, for example if $\partial_eT(A)$ is finite or if $G$ itself is finite.}. Set $K=\partial_eT(A)$.
Then there is a canonical faithful conditional expectation $\mathcal{E}\colon C(K)\to C(K/G)$ given by
\[\mathcal{E}(f)(G\cdot \tau)=\frac{1}{|G\cdot \tau|}\sum_{\sigma\in G\cdot \tau}f(\sigma)\]
for all $f\in C(K)$ and all $\tau\in K$.
Then $\M=\overline{A}^u$ has a natural structure of a
$W^*$-bundle over $K/G$ with conditional expectation $\M \to C(K/G)$ given by $\mathcal{E}\circ E$.
Moreover, for $\tau\in K$, the fiber of $\M$
over $G\cdot \tau$ can be canonically identified with the (finite) direct sum
$\bigoplus\limits_{\sigma\in G\cdot \tau}\overline{A}^\sigma$.
\end{eg}

Spatial tensor products of $W^*$-bundles are again $W^*$-bundles in a natural way; see 
Definition~3.4 in~\cite{BBSTWW_covering_2015}. Next, we observe that
the construction works well with equivariant $W^*$-bundles. 

\begin{rem}\label{rem:TensorProdEqWBundles}
Let $G$ be a discrete group. Given $G$-equivariant 
$W^*$-bundles $(\mathcal{M},\gamma)$ and $(\mathcal{N},\delta)$ over spaces $K$ and $L$, respectively, 
its $C^*$-algebraic minimal tensor product $\mathcal{M}\otimes \mathcal{N}$ admits a faithful
conditional expectation $E=E_{\mathcal{M}}\otimes E_{\mathcal{N}}$ onto $C(K\times L)\cong C(K)\otimes C(L)$
which is tracial; see Definition~3.4 in~\cite{BBSTWW_covering_2015}. Thus, $E$ induces a uniform 2-norm.
The tensor product $\mathcal{M}\overline{\otimes}\mathcal{N}$
of the $W^*$-bundles is the strict completion of $\mathcal{M}\otimes\mathcal{N}$ in the Hilbert
$C(K\times L)$-module associated to $(\mathcal{M}\overline{\otimes}\mathcal{N},E)$. Moreover, the 
$C^*$-algebraic tensor product action $\gamma\otimes \delta$ of $G$ on 
$\mathcal{M}\otimes\mathcal{N}$ extends to an action on the $W^*$-bundle $\mathcal{M}\overline{\otimes}\mathcal{N}$,
which we will also denote by $\gamma\otimes\delta$.
\end{rem}

We will mostly use tensor products when one of the factors is the $W^*$-bundle $\mathcal{R}$, in which case
$\mathcal{M}$ and $\MoR$ are $W^*$-bundles over the same space. When both $W^*$-bundles are von Neumann algebras,
the tensor product in the sense discussed above agrees with the spatial tensor product. 

Another notion we will recurrently use is that of ultrapowers of equivariant $W^*$-bundles; 
see Definition~3.7 in~\cite{BBSTWW_covering_2015}. 

\begin{df}\label{df:UltrapowEqWBundle}
Let $G$ be a discrete group, and let $(\M, K, E, \gamma)$ be a $G$-equivariant $W^*$-bundle, with uniform 
2-norm $\|\cdot\|_{2,u}$. Let $\omega$ be a free ultrafilter on $\N$. We define the \emph{ultrapower} of this 
$W^*$-bundle to be the quadruple $(\M^\omega, K^\omega, E^\omega, \gamma^\omega)$ defined as follows:
\be\item The \ca\ $\M^\omega$ is the quotient of $\ell^\I(\N,\M)$ by the ideal
\[c_\omega(\N,\M)=\{(a_n)_{n\in\N}\in \ell^\I(\N,\M)\colon \lim_{n\to \omega}\|a_n\|_{2,u}=0\}.\]
Then $\M^\omega$ inherits a uniform 2-norm $\|\cdot\|_{2,u}^\omega$ given by
\[\|(a_n)_{n\in\N}\|_{2,u}^\omega=\lim_{n\to\omega}\|a_n\|_{2,u}.\]
\item The space $K^\omega$ is the ultracopower of $K$, that is, the Gelfand spectrum of the $C^*$-algebraic ultrapower 
$\prod\limits_{\omega} C(K)$.\footnote{The ultrapower of $K$, denoted $K_\omega=\prod_{\omega} K$,
is defined as the quotient of $\prod\limits_{n\in\N}K$ modulo the relation given by $(x_n)_{n\in\N}\sim (y_n)_{n\in\N}$ whenever
$\{n\in\N\colon x_n=y_n\}\in \omega$. This space is not of interest to us in this paper.}
\item The conditional expectation $E^\omega \colon \M^\omega\to C(K^\omega)$ is induced by the 
conditional expectation $E\colon \ell^\I(\N,\M)\to \ell^\I(\N,C(K))$ via pointwise application of $E$; see~Proposition~3.9
in~\cite{BBSTWW_covering_2015}.
\item The action $\gamma^\omega\colon G\to\Aut(\M^\omega)$ is induced by the associated action $\gamma^\I$
of $G$ on $\ell^\I(\N,\M)$ which consists in applying $\gamma$ 
pointwise, noting that the ideal $c_\omega(\N,\M)$ is invariant under $\gamma^\I$. \ee
\end{df}

We also need the following notions of equivalence of actions on $W^*$-bundles.

\begin{df}
Let $\M$ and $\mathcal{N}$ be $W^*$-bundles over the same compact metrizable space, let $G$ be a discrete group
and let $\gamma\colon G\to\Aut(\M)$ and $\beta\colon G\to\Aut(\mathcal{N})$ be fiber-wise actions.
\be\item We say that $(\M,\gamma)$ and $(\mathcal{N},\beta)$ are \emph{conjugate (as equivariant $W^*$-bundles)}
if there exists an isomorphism of \ca s $\varphi\colon \M\to\mathcal{N}$ satisfying
\[E_\mathcal{N}\circ\varphi=E_{\M} \ \ \mbox{ and } \ \ \beta_g\circ \varphi=\varphi\circ\gamma_g\]
for all $g\in G$.
\item A $\gamma$-\emph{cocycle} is a function $u\colon G\to\U(\M)$ satisfying
$u_{gh}=u_g\gamma_g(u_h)$ for all $g,h\in G$. In this case, we define the cocycle
perturbation $\gamma^u$ of $\gamma$ to be $\gamma^u_g=\Ad(u_g)\circ\gamma_g$ for
all $g\in G$.
\item We say that $(\M,\gamma)$ and $(\mathcal{N},\beta)$ are \emph{cocycle conjugate (as equivariant $W^*$-bundles)}
if there exists a $\gamma$-cocycle $u$ such that $\gamma^u$
and $\beta$ are conjugate in the sense of (1) above.
\ee
\end{df}

In the context of the above definition, note that there is a canonical, equivariant embedding
$(\M,\gamma)\to (\M^\omega,\gamma^\omega)$, via constant sequences. Identifying $\M$ with its image 
in $\M^\omega$, we write $\M^\omega\cap \M'$ for the relative commutant, and observe that $\gamma^\omega$
restricts to an action on $\M^\omega\cap \M'$, which we also denote by $\gamma^\omega$. 

The result below is stated for $W^*$-bundles, but it is new even for von Neumann algebras.
The argument is mostly standard, and is inspired by McDuff's original work on absorption of $\R$;
see for example Proposition~3.11 in~\cite{BBSTWW_covering_2015}, which is the case of the trivial group.
Additional work is needed in to prove that (2) implies (1), particularly to obtain the cocycle, and this
has been noticed several times in the $C^*$-algebraic setting \cite{GolIzu_quasifree_2011, MatSat_stability_2012, Sza_strongly_2018}.

\begin{thm}\label{thm:W*BdleAbsR}
Let $\mathcal{M}$ be a strictly separable $W^*$-bundle, let $G$ be a countable discrete group,
let $\delta\colon G\to\Aut(\R)$ be a McDuff action,
and let $\gamma\colon G\to\Aut(\mathcal{M})$ be a fiber-wise action. Then the following are equivalent:
\be
\item $(\mathcal{M},\gamma)$ is cocycle conjugate to $(\mathcal{M}\overline{\otimes} \mathcal{R},\gamma\otimes\delta)$ (as equivariant $W^*$-bundles);
\item there exists a unital equivariant homomorphism $(\mathcal{R},\delta)\to (\mathcal{M}^\omega\cap \mathcal{M}',\gamma^\omega)$.
\ee
\end{thm}
\begin{proof}
(1) implies (2). Observe that if $(\mathcal{N}_0,\gamma_0)$ is cocycle conjugate to $(\mathcal{N}_1,\gamma_1)$, then
$(\mathcal{N}_0^\omega\cap\mathcal{N}_0,\gamma^\omega_0)$ is \emph{conjugate} to $(\mathcal{N}_1^\omega\cap\mathcal{N}_1,\gamma^\omega_1)$.
Since $((\M\overline{\otimes}\R)^\omega\cap (\M\overline{\otimes}\R)',(\gamma\otimes\delta)^\omega)$ has a unital copy of $(\R,\delta)$,
so does $(\M^\omega\cap \M',\gamma^\omega)$.

(2) implies (1). Since $\delta$ is cocycle conjugate to a
strongly McDuff action, and the statement refers to cocycle conjugacy, it suffices to prove the result when $\delta$
itself is strongly McDuff. We assume this from now on.

\textbf{Claim:} There exists a unitary $u=(u_n)_{n\in\N}$ in the fixed point algebra
$((\M\overline{\otimes}\R)^\omega\cap (\M\overline{\otimes}1_{\R})')^{(\gamma\otimes\delta)^\omega}$
such that for every contraction $a\in \M\overline{\otimes}\R$ there exists a
contraction $ b\in (\M\otimes 1_\R)^\omega$ satisfying
\[\lim_{n\to\I} \|u_nau_n^*-b\|_{2,u}^\omega=0.\]

We prove the claim. Let $\theta\colon (\R,\delta)\to (\M^\omega\cap \M',\gamma^\omega)$ be a unital, equivariant embedding, which exists by assumption.
Denote by $\iota_{\R}\colon \R\to\M\overline{\otimes}\R$ the second factor embedding, regarded as an equivariant embedding
$(\R,\delta)\to (\M\overline{\otimes}\R,\gamma\otimes\delta)$. Since the images of $\theta$ and $\iota_\R$ commute, there is a unital, equivariant
homomorphism
\[\theta\otimes\iota_\R\colon (\R\overline{\otimes}\R,\delta\otimes \delta)\to \left((\M\overline{\otimes}\R)^\omega\cap (\M\overline{\otimes}1_{\R})',(\gamma\otimes\delta)^\omega\right),\]
which is determined by $(\theta\otimes\iota_\R)(x\otimes y)=\theta(x)\iota_\R(y)$ for all $x,y\in\R$. Using that $\delta$ is strongly McDuff,
fix a sequence $(u_n)_{n\in\N}$ in the image of $\theta\otimes\iota_\R$ satisfying
\[\lim\limits_{n\to\I}\|u_n(1_\M\otimes y)u_n^*-\theta(y)\|_{2,u}^\omega=0
\] 
and
\[
\lim\limits_{n\to\I}\|(\delta_g\otimes \delta_g)(u_n)-u_n\|_{2,u}^\omega=0
 \] 
for all $y\in\R$ and for all $g\in G$. Hence
$\lim\limits_{n\to\I}\|u_n(x\otimes y)u_n^*-(x\otimes 1_\R)\theta(y)\|_{2,u}^\omega=0$ 
for all $x\in \M$ and all $y\in \R$.

It follows that there is a well-defined unital embedding $\eta\colon\M\overline{\otimes}\R\to (\M\otimes 1_\R)^\omega$ given on
simple tensors by $\eta(x\otimes y)=(x\otimes 1_\R)\theta(y)$ for all $x\in\M$ and all $y\in \R$. In particular, for $a\in\M\overline{\otimes}\R$
with $\|a\|\leq 1$, the element $b=\eta(a)\in (\M\otimes 1_\R)^\omega$ is a contraction satisfying $\lim\limits_{n\to\I} \|u_nau_n^*-b\|_{2,u}^\omega=0$, and
the claim is proved.

Let $(a_n)_{n\in\N}$ and $(b_n)_{n\in\N}$ be $\|\cdot\|_{2,u}$-dense sequences of the unit balls of $\M\overline{\otimes}\R$ and of $\M$, respectively,
and let $(F_n)_{n\in\N}$ be an increasing sequence of finite subsets of $G$ whose union equals $G$. We will inductively apply the claim to the elements
$a_n$. Recall that $\iota_\M\colon \M\to \M\overline{\otimes}\R$ denotes the (equivariant) first tensor factor embedding.
Using the claim above, we choose, for each $n\in\N$, a unitary $v_n\in \M\overline{\otimes}\R$ and elements
$x^{(1)}_n,\ldots,x^{(n)}_n \in \M$ satisfying:
\be
\item[(a)] $\left\| \Ad(v_1\cdots v_n)(a_j)-\iota_\M(x_n^{(j)})\right\|_{2,u}\leq \frac{1}{n}$ for all $j=1,\ldots, n$;
\item[(b)] $\left\| v_n\iota_\M(b_j)-\iota_\M(b_j)v_n\right\|_{2,u}\leq \frac{1}{2^n}$ for all $j=1,\ldots, n$;
\item[(c)] $\left\| v_n\iota_\M(x_n^{(j)})-\iota_\M(x_n^{(j)})v_n\right\|_{2,u}\leq \frac{1}{2^n}$ for all $j=1,\ldots, n$;
\item[(d)] $\left\| (\gamma\otimes\delta)_g(v_n)-v_n\right\|_{2,u}\leq \frac{1}{2^n}$ for all $g\in F_n$.
\ee

For $n\in\N$, set $w_n=v_1\cdots v_n$, which is a unitary in $\M\overline{\otimes}\R$. For $j\in\N$, the sequence $(\Ad(w_n)(\iota_\M(b_j)))_{n\in\N}$
is Cauchy with respect to $\|\cdot\|_{2,u}$. Since the unit ball of $\M\overline{\otimes}\R$ is complete with respect to this norm, this sequence
converges to a contraction $\varphi(b_j)\in \M\overline{\otimes}\R$. Since the assignment $b_j\mapsto \varphi(b_j)$ is linear and isometric with
respect to $\|\cdot\|_{2,u}$, it extends to a unital isometric map $\varphi\colon \M\to \M\overline{\otimes}\R$, which is easily seen to be a $\ast$-homomorphism.
By condition (c) above, we have
\[\|\varphi(b_j)-\Ad(w_m)(b_j)\|\leq \sum_{k=m+1}^\infty \frac{1}{2^k}=\frac{1}{2^m}\]
for all $j\in\N$ and all $m\geq j$. Using condition (a) at the second step, we deduce that
\[\|\varphi(b_j)-a_j\|\leq \frac{1}{2^m} + \|\Ad(w_m)(b_j)-a_j\|\leq \frac{1}{2^m}+\sum_{k=m+1}^\infty \frac{1}{2^k}=\frac{1}{2^{m-1}}\]
for all $m\geq j$. By density of $(a_n)_{n\in\N}$ and $(b_n)_{n\in\N}$, it follows that $\varphi$ is surjective, and hence an isomorphism.

We claim that $\varphi$ is an isomorphism of $W^*$-bundles. Indeed, for $j\in\N$ we use that $E_\M$ is tracial and continuous with respect
to $\|\cdot\|_{2,u}$ to get
\[(E_{\M\overline{\otimes}\R}\circ\varphi)(b_j)=\lim_{n\to\I} (E_\M\otimes\tau_\R)(\Ad(w_n)(\iota_\M(b_j)))=E_\M(b_j).\]
From this it is clear that $E_{\M\overline{\otimes}\R}\circ\varphi=E_\M$, so $\varphi$ is an isomorphism of $W^*$-bundles.

It remains to construct the cocycle and prove that $\varphi$ is equivariant. Let $g\in G$ and $j\in\N$. By condition (d) above, the bounded sequence
$(w_n(\gamma\otimes\delta)_g(w_n^*))_{n\in\N}$ is Cauchy with respect to $\|\cdot\|_{2,u}$. Thus, it converges to an element $u_g\in \M\overline{\otimes}\R$,
which is readily checked to be a unitary since multiplication on bounded sets is jointly continuous with respect to $\|\cdot\|_{2,u}$.
Denote by $u\colon G\to \mathcal{U}(\M\overline{\otimes}\R)$ the resulting map. We claim that $u$ is a cocycle for $\gamma\otimes\delta$.
To see this, let $g,h\in G$. Then
\begin{align*}
u_g(\gamma\otimes\delta)_g(u_h)&=\lim_{n\to\I}  w_n(\gamma\otimes\delta)_g(w_n^*) (\gamma\otimes\delta)_g(w_n(\gamma\otimes\delta)_h(w_n^*))\\
&=\lim_{n\to\I} w_n(\gamma\otimes\delta)_g(w_n^*)(\gamma\otimes\delta)_g(w_n)(\gamma\otimes\delta)_{gh}(w_n^*)\\
&=\lim_{n\to\I} w_n(\gamma\otimes\delta)_{gh}(w_n^*)\\
&= u_{gh},
\end{align*}
as desired. 
Finally, we show that $\varphi$ is equivariant with respect to $\gamma$ and $\Ad(u)\circ (\gamma\otimes\delta)$. Let $g\in G$. 
Working in the point-$\|\cdot\|_{2,u}$ topology, we get:
\begin{align*}
\varphi\circ\gamma_g & = \lim_{n\to\I} \Ad(w_n)\circ\iota_\M\circ\gamma_g\\
& = \lim_{n\to\I} \Ad(w_n)\circ(\gamma\otimes\delta)_g\circ \iota_\M \\
& = \lim_{n\to\I} \Ad(w_n)\circ(\gamma\otimes\delta)_g\circ\Ad(w_n^*)\circ\Ad(w_n)\circ \iota_\M \\
& = \lim_{n\to\I} \Ad(w_n(\gamma\otimes\delta)_g(w_n^*))\circ(\gamma\otimes\delta)_g\circ \Ad(w_n)\circ \iota_\M\\
&= \Ad(u_g)\circ (\gamma\otimes\delta)_g\circ \varphi.\qedhere
\end{align*}
\end{proof}

\begin{df}\label{df:EqMcDuff}
Let $\mathcal{M}$ be a strictly separable $W^*$-bundle, let $G$ be a discrete group, let $\delta\colon G\to\Aut(\R)$ be a McDuff action,
and let $\gamma\colon G\to\Aut(\mathcal{M})$ be a fiber-wise
action. We say that $(\mathcal{M},\gamma)$ is \emph{$\delta$-McDuff} if it satisfies the equivalent conditions in \autoref{thm:W*BdleAbsR}.
\end{df}

In the case of McDuff actions as in \autoref{eg:deltanu}, absorption can be
characterized using equivariant embeddings of matrix algebras, as we show next.

\begin{thm}\label{thm:Absdeltanu}
Let $G$ be a discrete group and let $(\nu_n)_{n\in\N}$ be a sequence of unitary representations as in \autoref{eg:deltanu}.
Abbreviate $\delta^{\nu}$ to $\delta$.
Let $\M$ be a strictly separable $W^*$-bundle, and let $\gamma\colon G\to\Aut(\M)$ be a fiber-wise action. Then the following
are equivalent:
\be
\item $(\M,\gamma)$ is $\delta$-McDuff;
\item there exists a unital homomorphism
$(M_{d},\Ad(\nu)) \to (\mathcal{M}^\omega\cap \mathcal{M}',\gamma^\omega)$.
\ee
\end{thm}
\begin{proof}
(1) implies (2): Since condition (1) is equivalent to the existence of a unital equivariant embedding of $(\R,\delta)$
into $(\M^\omega\cap\M',\gamma^\omega)$ by \autoref{thm:W*BdleAbsR}, and $(\R,\delta)$ has a unital and equivariant
copy of $(M_{d},\Ad(\nu))$, the result follows.

(2) implies (1): Assume that there exists a unital homomorphism
\[\varphi\colon (M_{d},\Ad(\nu))\to (\M^\omega\cap \M',\gamma^\omega).\]

\textbf{Claim:}\emph{  For any strictly separable, $\gamma^\omega$-invariant subalgebra
$\mathcal{N}\subseteq \M^\omega$, there exists a unital equivariant homomorphism
$(M_{d},\Ad(\nu))\to (\M^\omega\cap \mathcal{N}',\gamma^\omega)$.}

The proof uses a standard ``speed-up'' trick.  Let $(x_m)_{m\in\N}$
be a strictly dense subset of $\mathcal{N}$. For each $m\in\N$, find a bounded sequence $(x_m^{(n)})_{n\in\N}$ such that
$[(x_m^{(n)})_{n\in\N}]=x_m$. Using the Choi-Effros lifting theorem, find a sequence $(\varphi_n)_{n\in\N}$ of unital completely positive
maps $\varphi_n\colon M_d\to\M$ satisfying:
\bi
\item $\lim\limits_{n\to \omega}\|\varphi_n(ab)-\varphi_n(a)\varphi_n(b)\|_{2,u}=0$ for all $a,b\in M_d$;
\item $\lim\limits_{n\to \omega}\|\varphi_n(a)x-x\varphi_n(a)\|_{2,u}=0$ for all $a\in M_d$ and for all $x\in \M$;
\item $\lim\limits_{n\to \omega}\|\gamma_g(\varphi_n(a))-\varphi_n(\nu_ga\nu_g^*)\|_{2,u}=0$ for all $a\in M_d$ and for all $g\in G$.
\ei

Let $\{e_{j,k}\}_{1\leq j,k\leq d}$ be a system of matrix units for $M_d$, and let $(F_n)_{n\in\N}$ be an increasing sequence
of finite subsets of $G$ whose union equals $G$. For $n\in\N$, choose $r_n\in\N$ such that
\bi
\item $\left\|\varphi_{r_n}(e_{j,k})x_m^{(n)}-x_m^{(n)}\varphi_{r_n}(e_{j,k})\right\|_{2,u}<\frac{1}{n}$ for all $j,k=1,\ldots,d$ and for all $m\leq n$;
\item $\|\varphi_{r_n}(e_{j,k})\varphi_{r_n}(e_{i,l})-\delta_{k,i}\varphi_{r_n}(e_{j,l})\|_{2,u}<\frac{1}{n}$ for all $i,j,k,l=1,\ldots,d$;
\item $\|\gamma_g(\varphi_{r_n}(e_{j,k}))-\varphi_{r_n}(\nu_ge_{j,k}\nu_g^*)\|_{2,u}<\frac{1}{n}$ for all $g\in F_n$, and for all $j,k=1,\ldots,d$.
\ei
Denote by $\psi\colon M_d\to \M^\omega$ the map induced by the subsequence $(\varphi_{r_n})_{n\in\N}$.
It is then easy to check that $\psi$ is equivariant, and that its image
is contained in $\M^\omega\cap \mathcal{N}'$, as desired. This proves the claim.

We construct a unital equivariant homomorphism $(\R,\delta)\to (\M^\omega\cap \mathcal{M}',\gamma^\omega)$ as follows.
Let $\psi_1\colon (M_{d},\Ad(\nu))\to (\M^\omega\cap \mathcal{M}',\gamma^\omega)$
be any unital equivariant homomorphism.
Use the claim to find a unital equivariant homomorphism $\psi_2\colon (M_{d},\Ad(\nu))\to (\M^\omega\cap \mathcal{M}',\gamma^\omega)$
whose image commmutes with $\psi_1(M_{d_1})$. Proceeding inductively, we find unital equivariant homomorphisms
\[\psi\colon (M_d,\Ad(\nu))\to (\M^\omega\cap \mathcal{M}',\gamma^\omega),\]
for $k\in\N$, such that $\psi_k(M_{d})$ commutes with $\psi_j(M_{d})$ for $j\leq k-1$. These maps induce a unital equivariant homomorphism
\[\psi\colon \left(\bigotimes\limits_{n\in\N}M_{d},\bigotimes\limits_{n=1}^\I\Ad(\nu)\right) \to (\M^\omega\cap \mathcal{M}',\gamma^\omega),\]
which extends to a unital equivariant homomorphism
$(\R,\delta)\to (\M^\omega\cap \mathcal{M}',\gamma^\omega)$ by uniqueness of the trace on $\bigotimes\limits_{n\in\N}M_{d}$.
This completes the proof.
\end{proof}

Next, we need a result allowing us to conclude that a $W^*$-bundle is equivariantly McDuff whenever its fibers are.
The following strengthening of Lemma~3.17 in~\cite{BBSTWW_covering_2015} will be used.
We state the lemma for ultrapowers of $W^*$-bundles (and not ultraproducts, which we have not defined in this work), 
because this is all we will need here. With the natural definitions, the proof for ultraproducts is identical. 

\begin{lma}\label{lma:FibersToBundleCopyR}
Let $(\M,K,E)$ be a strictly separable bundle, let $G$ be a countable
group and let $\gamma\colon G\to\Aut(\M)$ be a fiber-wise action. Let $\omega$ be a free ultrafilter, and let
$S\subseteq (\M\overline{\otimes} \R)^\omega$ be a $\|\cdot\|_{2,u}^\omega$-separable, selfadjoint subset containing the unit.

Given a partition of unity $f_1,\ldots,f_m\in C(K^\omega)$, there exist projections
\[p_1,\ldots,p_m\in \left((\M\overline{\otimes} \R)^\omega\cap S'\right)^{\gamma\otimes\id_{\R}}\]
such that $\sum\limits_{j=1}^m p_j=1$ and $\tau_\lambda(p_jx)=f_j(\lambda)\tau_\lambda(x)$ for all $x\in S$ and for all $\lambda\in K^\omega$.
\end{lma}
\begin{proof}
We describe the modifications needed in the proofs of Lemma~3.16 and Lemma~3.17 in~\cite{BBSTWW_covering_2015}.
The next claim is the necessary replacement of Lemma~3.16.

\textbf{Claim:} \emph{Fix $x_1,\ldots,x_r\in \M\overline{\otimes}\R$ and a finite set $F\subseteq G$.
Given $\ep>0$, there exist projections $q_1,\ldots,q_n\in\M\overline{\otimes}\R$ such that $\sum\limits_{j=1}^m q_j=1$ and
\be
\item $\|q_jx_k-x_kq_j\|_{2,u}<\ep$ for all $j=1,\ldots,m$ and all $k=1,\ldots,r$;
\item $\|(\gamma\otimes\id_{\R})_g(q_j)-q_j\|_{2,u}<\ep$ for all $g\in F$ and all $j=1,\ldots,m$;
\item $\tau_\lambda(q_j)=f_j(\lambda)$ for all $j=1,\ldots,m$ and all $\lambda\in K$;
\item $|\tau_\lambda(q_jx_k)-f_j(\lambda)\tau_\lambda(x_k)|<\ep$ for all $j=1,\ldots,m$, all $k=1,\ldots,r$, and all $\lambda\in K$.
\ee}

To prove the claim, we use the notation of Lemma~3.16 in~\cite{BBSTWW_covering_2015}, except that what we call here $q_j$ is called $p_j$ there.
Observe that the hyperfinite II$_1$-factor $\mathcal{S}=\R\cap M_k'$ carries the trivial action of $G$.
Hence the image of the $W^*$-bundle equivariant embedding $\theta\colon C_\sigma(K,\mathcal{S})\to (\M\overline{\otimes}\R)\cap \{y_1,\ldots, y_r\}'$
is contained in the fixed point algebra of the action $\gamma\otimes\id_{\R}$. Thus, the projections $e_t$, for $t\in [0,1]$,
can be chosen to be fixed, and hence also the projections $q_1,\ldots,q_n$. This proves the claim.

We turn to the proof of the lemma. For $k\in\N$, choose a sequence $(x_k^{(n)})_{n\in\N}$ in $\M\overline{\otimes}\R$
such that $\{[(x_k^{(n)})_{n\in\N}]\colon k\in\N\}$ is dense in $S$.
For each $n\in\N$, find a partition of unity $\{f_1^{(n)}, \ldots,f_m^{(n)}\}$ of $C(K)$ such that
$[(f^{(n)}_j)_{n\in\N}]=f_j$ in $C(K^\omega)$ for $j=1,\ldots,m$. Let $(F_n)_{n\in\N}$ be an increasing sequence of finite subsets of $G$
with $G=\bigcup\limits_{n\in\N} F_n$. Using the claim, find projections
$p_1^{(n)},\ldots,p_m^{(n)}\in \M\overline{\otimes}\R$ such that
\bi
\item $\sum\limits_{j=1}^m p_j^{(n)}=1$,
\item $\|(\gamma\otimes\id_{\R})_g(p^{(n)}_j)-p^{(n)}_j\|_{2,u}<\frac{1}{n}$,
\item $\|p^{(n)}_jx_k-x_kp^{(n)}_j\|<\frac{1}{n}$,
\item $\tau_\lambda(p^{(n)}_j)=f^{(n)}_j(\lambda)$,
\item $\left|\tau_\lambda(p^{(n)}_jx_k)-f^{(n)}_j(\lambda)\tau_\lambda(x^{(n)}_k)\right|<\frac{1}{n}$
\ei
for all $g\in F_n$, for all $j=1,\ldots,m$, for all $k=1,\ldots,r$, and for all $\lambda\in K$.
Set $p_j=[(p_j^{(n)})_{n\in\N}] \in \M^\omega$ for $j=1,\ldots,m$. It is then clear that these projections satisfy the desired conditions, and the
proof is finished.
\end{proof}

The following is one of the main results of this section.

\begin{thm}\label{thm:MoRabsorbsMcDuff}
Let $(\M,K,E)$ be a strictly separable $W^*$-bundle, let $\gamma\colon G\to\Aut(M)$ be a fiber-wise action of a discrete group $G$,
and let $\delta\colon G\to \Aut(\R)$ be a McDuff action.
If $\gamma^\lambda\colon G\to\Aut(\pi_\lambda(\M))$ is cocycle conjugate to
$\gamma^\lambda\otimes\delta\colon G\to\Aut(\pi_\lambda(\M)\overline{\otimes}\R)$ for every $\lambda\in K$,
then $\gamma\otimes\id_{\mathcal{R}}\colon G\to\Aut(\mathcal{M}\overline{\otimes}\mathcal{R})$ is cocycle conjugate to
$\gamma\otimes\id_{\R}\otimes\delta\colon G\to\Aut(\M\overline{\otimes}\R\overline{\otimes}\R)$.
\end{thm}
\begin{proof}
Since there is a strongly McDuff action which is cocycle conjugate to $\delta$, it suffices
to prove the result when $\delta$ itself is strongly McDuff.
Fix $\ep>0$ and finite subsets $F\subseteq G$ and $S\subseteq \M$. For $\lambda\in K$, use \autoref{thm:W*BdleAbsR} to find
a unital equivariant homomorphism
\[\varphi_\lambda\colon (\R,\delta)\to \left(\pi_\lambda(\M)^\omega\cap \pi_\lambda(\M)',(\gamma^\lambda)^\omega\right).\]
Find a matrix subalgebra $M_d \subseteq \mathcal{R}$ and a conditional expectation
$E \colon \R \to M_d$ with
\be
\item[(a)] $\|E(a^*b)-E(a)^*E(b)\|_2 <\ep$ for all $a,b\in S$;
\item[(b)] $\|\delta_g(E(a))-E(\delta_g(a))\|_2<\ep$ for all $g\in F$ and all $a\in S$.
\ee
Since $\varphi_\lambda \circ E$ factors through $M_d$, by the Choi-Effros lifting theorem we can lift it 
to a unital linear contractive map $\psi_\lambda\colon \R\to \M^\omega\cap\M'$, and find an open subset $U_\lambda\subseteq K^\omega$
satisfying
\be
\item[(a)] $\sup\limits_{\tau\in U_\lambda}\|\psi_\lambda(a^*b)-\psi_\lambda(a)^*\psi_\lambda(b)\|_{2,\tau}^\omega<\ep$ for all $a,b\in S$;
\item[(b)] $\sup\limits_{\tau\in U_\lambda}\|\gamma^\omega_g(\psi_\lambda(a))-\psi_\lambda(\delta_g(a))\|_{2,\tau}^\omega<\ep$ for all $a\in S$ and all $g\in F$.
\ee

Using compactness of $K^\omega$, find $m\in\N$ and $\lambda_1,\ldots,\lambda_m\in K$ such that $\bigcup\limits_{j=1}^m U_{\lambda_j}=K^\omega$.
Let $f_1,\ldots,f_m\in C(K^\omega)$ be a partition of unity subordinate to this cover.
Let $\widetilde{S}$ be the $\|\cdot\|_{2,u}^\omega$-separable, selfadjoint subset of $(\M\overline{\otimes}\R)^\omega$ given by
\[\widetilde{S}= \M\overline{\otimes} \R \cup \{\gamma^\omega_g(\psi_{\lambda_j}(a))\otimes x \colon g\in F, a\in S, j=1,\ldots,m, x\in\R\}.\]
Use \autoref{lma:FibersToBundleCopyR} to find projections $p_1,\ldots,p_m\in ((\M\overline{\otimes}\R)^\omega\cap \widetilde{S}')^{\gamma\otimes \id_{\R}}$
satisfying $\sum\limits_{j=1}^m p_j=1$ and $\tau_\lambda(p_jx)=f_j(\lambda)\tau_\lambda(x)$ for all $x\in S$ and for all $\lambda\in K^\omega$.
Define a map $\psi\colon \R\to (\M\overline{\otimes}\R)^\omega\cap (S\otimes \R)'$ by
\[\psi(x)=\sum_{j=1}^m p_j(\psi_{\lambda_j}(x)\otimes 1_{\R})\]
for all $x\in \R$.

It is routine to check that
\be
\item[(a')] $\|\psi(a^*b)-\psi(a)^*\psi(b)\|_{2,u}^\omega<\ep$ for all $a,b\in S$;
\item[(b')] $\|\gamma^\omega_g(\psi(a))-\psi(\delta_g(a))\|_{2,u}^\omega<\ep$ for all $a\in S$ and all $g\in F$.
\ee

Since $F\subseteq G$, $S\subseteq\M$ and $\ep>0$ are arbitrary, countable saturation of $\M\overline{\otimes}\R$ implies that there
exists a unital, equivariant homomorphism $(\R,\delta)\to (\M\overline{\otimes}\R\cap (\M\overline{\otimes}\R)^{\prime},\gamma^\omega)$,
which implies the result by \autoref{thm:W*BdleAbsR}.
\end{proof}

In later sections, it will be crucial to know that the extra copy of the trivial
action on $\R$ can sometimes be dispensed of, whenever $\delta$ absorbs $\id_\R$.
In \autoref{thm:McDuffFromFiberToBundle}, we show that this is the case
for $W^*$-bundles whose base space has finite covering dimension.
Despite being a purely $W^*$-algebraic statement, its proof surprisingly
factors through actions on \ca s, specifically on dimension drop algebras.
We therefore make a small digression from the theme of equivariant $W^*$-bundles
to prove some facts about actions on dimension drop algebras that will be needed
in the sequel.

\begin{nota}\label{nota:DimDrop}
Given $m,n\in\N$, we denote by $I_{m,n}$ the dimension drop algebra:
\[I_{m,n}=\{f\in C([0,1],M_m\otimes M_n) \colon f(0)\in M_m\otimes 1, f(1)\in 1\otimes M_n\}.\]
If $G$ is a discrete group and $u\colon G\to\U(M_m)$ and $v\colon G\to\U(M_n)$ are unitary
representations, we denote by $\gamma_{u,v}\colon G\to\Aut(I_{m,n})$ the restriction of
$\id_{C([0,1])}\otimes \Ad(u)\otimes \Ad(v)\colon G\to \Aut(C([0,1]\otimes M_m\otimes M_n))$
to the invariant subalgebra $I_{m,n}$.
\end{nota}

Equivariant maps of order zero are a crucial tool to deal with dimension drop algebras.
The following definition is by now standard.


\begin{df}\label{df:oz}
Let $\psi\colon A\to B$ be a completely positive map between \ca s $A$ and $B$. We say that $\psi$ is \emph{order zero}
if $\psi(a)\psi(b)=0$ whenever $a,b\in A_+$ satisfy $ab=0$.
\end{df}

In the next proposition, for $n\in\N$ we write $CM_n$ for the cone $C_0((0,1],M_n)$ over $M_n$, and we write
$\widetilde{CM_n}$ for its minimal unitization.

\begin{prop}\label{prop:UnivPropDimDropUV}
Let $G$ be a discrete group, let $m,n\in\N$, let $u\colon G\to\U(M_m)$ and $v\colon G\to \U(M_n)$ be
unitary representations, and let $\beta\colon G\to\Aut(B)$ be an action of $G$ on a unital \ca\ $B$.
Then the following are equivalent:
\be\item There is a unital equivariant homomorphism $\pi\colon (I_{m,n},\gamma_{u,v})\to (B,\beta)$.
\item There exist equivariant completely positive contractive order zero maps
\[\xi\colon (M_m,\Ad(u))\to (B,\beta), \ \ \mbox{ and } \ \ \eta\colon (M_n,\Ad(v))\to (B,\beta)\]
with commuting ranges, that satisfy $\xi(1)+\eta(1)=1$.\ee
\end{prop}
\begin{proof}
(1) implies (2): Define maps
\[\varphi\colon (M_m,\Ad(u))\to (I_{m,n},\gamma_{u,v}) \ \ \mbox{ and } \ \
\psi\colon (M_n,\Ad(v)) \to (I_{m,n},\gamma_{u,v})\]
by $\varphi(x)(t)=t(x\otimes 1)$ and $\psi(y)(t)=(1-t)(1\otimes y)$ for all $x\in M_m$, for all
$y\in M_n$, and all $t\in [0,1]$. Then $\varphi$ and $\psi$ are equivariant
completely positive contractive order zero maps with commuting ranges that satisfy $\varphi(1)+\psi(1)=1$. If a homomorphism
$\pi$ as in the statement of (1) exists, then the maps $\xi$ and $\eta$ as in (2) are obtained by setting
$\xi=\pi\circ\varphi$ and $\eta=\pi\circ\psi$.

(2) implies (1):
Denote by $f_0\in C_0((0,1])$ the inclusion of $(0,1]$ into $\C$.
Define equivariant homomorphisms
\[\sigma\colon (CM_m),\id\otimes \Ad(u))\to (B,\beta), \ \ \mbox{ and } \ \
 \theta\colon (CM_n),\id\otimes \Ad(v))\to (B,\beta)
\]
by $\sigma(f_0\otimes x)=\xi(x)$ and $\theta(f_0\otimes y)$ for all $x\in M_m$
and $\theta(f_0\otimes y)=\eta(y)$ for all $y\in M_n$.
Denote by $\widetilde{\sigma}$ and $\widetilde{\theta}$ the unital equivariant
extensions of $\sigma$ and $\theta$, respectively, to the minimal unitizations.
Then the ranges of $\widetilde{\sigma}$ and $\widetilde{\theta}$ commute,
so they define a unital homomorphism
\[\widetilde{\sigma}\otimes \widetilde{\theta} \colon \widetilde{CM_m}\otimes \widetilde{CM_n}\to B.\]
It is clear that the $G$-invariant element $1\otimes 1 - (f_0\otimes 1_m)\otimes 1 - 1\otimes (f_0\otimes 1_n)$
belongs to the kernel of $\widetilde{\sigma}\otimes \widetilde{\theta}$. Since the quotient of $\widetilde{CM_m}\otimes \widetilde{CM_n}$
by the (automatically $G$-invariant) ideal generated by $1\otimes 1 - (f_0\otimes 1_m)\otimes 1 - 1\otimes (f_0\otimes 1_n)$
is isomorphic to $(I_{m,n},\gamma_{u,v})$, there is an induced unital equivariant homomorphism $\pi\colon (I_{m,n},\gamma_{u,v})\to (B,\beta)$,
as desired.
\end{proof}

\begin{lma}\label{lma:NonPrimeDimDrop}
Let $G$ be a discrete group, let $m,n\in\N$, let $u\colon G\to\U(M_m)$ and $v\colon G\to \U(M_n)$ be
unitary representations, and assume that $u$ is unitarily equivalent to a tensor factor of $v$. Then
there is a unital equivariant homomorphism $(M_m,\Ad(u))\to (I_{m,n},\gamma_{u,v})$.
\end{lma}
\begin{proof}
Let $k=n/m$.
Upon replacing $v$ with a unitarily equivalent representation (which yields a conjugate action),
we may assume that there exists a unitary representation $w\colon G\to\U(M_k)$ such that
$v=u\otimes w$.

Denote by $\{e_{i,j}\colon 1\leq i,j\leq m\}$ the matrix units of $M_m$.
Set
\[
\nu_1=\sum_{i,j=1}^m e_{i,j}\otimes e_{j,i}\otimes 1_k\in M_m\otimes M_m\otimes M_k,
\]
and observe that $\nu_1$ is a unitary implementing the tensor flip $a\otimes b\otimes c\mapsto b\otimes a\otimes c$
on $M_m\otimes M_m\otimes M_k$. Moreover, $\nu_1$ is $\Ad(u\otimes u\otimes w)$-invariant.
Since $(M_m\otimes M_m\otimes M_k)^{\Ad(u\otimes u\otimes w)}$ is
finite dimensional, its unitary group is connected and hence there is a continuous
unitary path $\nu\colon [0,1]\to (M_m\otimes M_m\otimes M_k)^{\Ad(u\otimes u\otimes w)}$ satisfying $\nu(0)=1$ and $\nu(1)=\nu_1$.
Define $\iota\colon M_m\to I_{m,n}$ by $\iota(a)(t)=\nu(t)(a\otimes 1)\nu(t)^*$ for all $a\in M_m$ and all $t\in [0,1]$.
It is immediate to check that $\iota$ is a unital, equivariant homomorphism, proving the lemma.
\end{proof}

\begin{prop}\label{prop:EqOzCommRges}
Let $G$ be a discrete group, let $B$ be a \ca, let $\beta\colon G\to\Aut(B)$ be an action,
let $d,n\in\N$, and let $\nu\colon G\to\U(M_d)$ be a unitary representation. Let
$\varphi_1,\ldots,\varphi_n\colon (M_d,\Ad(\nu))\to (B,\beta)$ be completely positive contractive
equivariant order zero maps with commuting ranges and such that $\sum\limits_{j=1}^n\varphi_j(1)$ is a contraction.
Then there exists a completely positive contractive equivariant order zero map
$\psi\colon (M_d,\Ad(\nu))\to (B,\beta)$ with $\psi(1)=\sum\limits_{j=1}^n\varphi_j(1)$.

In particular, when $B$ is unital and $\sum\limits_{j=1}^n\varphi_j(1)=1$, then there exists a unital equivariant homomorphism
$\psi\colon (M_d,\Ad(\nu))\to (B,\beta)$.
\end{prop}
\begin{proof}
Using finite induction, it is enough to prove the statement for $n=2$.
Also, without loss of generality we can assume
that $B$ is generated as a \ca\ by the images of $\varphi_1$ and $\varphi_2$.
Set $h=\varphi_1(1)+\varphi_2(1)$. Then $h$ is a strictly positive central element in $B$.
Fix $a \in M_d$. For $j=1,2$ and for $n \in \N$, set
\[
z_{n,j}(a) = \varphi_j(a) \left ( h^2+\frac{1}{n} \right )^{-1/2} \, .
\]
Fix $j\in \{1,2\}$.

\textbf{Claim:}\emph{ for any $b \in B$, the sequences $(bz_{n,j}(a))_{n\in\N}$ and $(z_{n,j}(a)b)_{n\in\N}$ converge
in $B$.} 
The proof follows the lines of the proof of Lemma~1.4.4 from~\cite{Ped_algebras_1979}.
Since $\varphi_j(a)^*\varphi_j(a) \leq h^2$, it follows that
\[
z_{n,j}(a)^*z_{n,j}(a) \leq h^2 \left ( h^2+\frac{1}{n} \right )^{-1} \leq 1
\]
and hence $z_{n,j}(a)$ is contractive for all $n\in\N$. It thus suffices to show that the set
\[\{b\in B\colon \|b\|\leq 1 \mbox{ and } (bz_{n,j}(a))_{n\in\N} \mbox{ and } (z_{n,j}(a)b)_{n\in\N} \mbox{ are Cauchy}\}\]
is dense in the unit ball of $B$.
As $(h^{1/k})_{k\in\N}$ is an approximate unit for $B$, it suffices to check that $(h^{1/k}bz_{n,j}(a))_{n\in\N}$ and
$(z_{n,j}(a)bh^{1/k})_{n\in\N}$ are Cauchy for any $b$ in the unit ball of $B$ and any fixed $k$. For any $n,m \in \N$, set
\[
d_{n,m} =  \left ( h^2+\frac{1}{n} \right )^{-1/2} - \left ( h^2+\frac{1}{m} \right )^{-1/2} \, .
\]
Note that for any $\alpha > 1$, we have $h^{\alpha} d_{n,m} \to 0$ as $n,m \to \infty$
(which can be seen by using spectral theory and considering those as functions on the spectrum of $h$). Thus,
 we have
\begin{align*}
\|h^{1/k}bz_{n,j}(a) - h^{1/k}bz_{m,j}(a) \|^2 &=  \|h^{1/k}b\varphi_j(a)d_{n,m}\|^2 \\
	&= \|d_{n,m} \varphi_j(a)^*h^{1/k}b^*bh^{1/k}\varphi_j(a)d_{n,m}\| \\
	&\leq \|d_{n,m}h^{2(1+1/k)}d_{n,m}\| \\
	&= \|h^{1+1/k}d_{n,m}\|
\end{align*}
and hence $(h^{1/k}bz_{n,j}(a))_{n\in\N}$ is Cauchy. One shows analogously that the sequence $(z_{n,j}(a)h^{1/k}b)_{n\in\N}$ is Cauchy as well,
thus proving the claim.

It follows that the sequence $(z_{n,j}(a))_{n\in\N}$ converges strictly in $M(B)$. Set
\[\psi_j(a) = \lim_{n \to \infty}z_{n,j}(a),\]
where the limit is taken to be in the strict topology. Then the resulting map $\psi_j\colon M_d\to M(B)$ is completely
positive contractive order zero, and since $h$ is $G$-invariant, it is immediate that $\psi_j$ is also
$G$-equivariant. Moreover, $\psi_1(1)+\psi_2(1)=1$.
Extend $\beta$ to an action of $G$ on $M(B)$, which we also denote by $\beta$.
By \autoref{prop:UnivPropDimDropUV}, there is a unital, equivariant homomorphism
\[\overline{\pi}\colon (I_{d,d},\gamma_{\nu,\nu})\to (M(B),\beta).\]

Use \autoref{lma:NonPrimeDimDrop} to find a unital equivariant homomorphism $\iota\colon (M_d,\Ad(\nu))\to (I_{d,d},\gamma)$.
Then $\widetilde{\pi}\circ\iota\colon (M_d,\Ad(\nu))\to (M(B),\beta)$ is a unital equivariant homomorphism.
The proof is completed by letting $\psi\colon (M_d,\Ad(\nu))\to (B,\beta)$ be the equivariant completely
positive contractive order zero map given by $\psi(a)=\widetilde{\pi}(\iota(a))(\varphi_1(1)+\varphi_2(1))$ for
all $a\in M_d$.

The last part of the statement is immediate, since a unital order zero map is a homomorphism, 
by Theorem~3.2 in~\cite{WinZac_completely_2009}.
\end{proof}

We recall (\cite{WinZac_completely_2009}) that a completely positive contractive map
$\psi\colon A\to B$ between \uca s $A$ and $B$ is order zero if and only if
it satisfies $\psi(1)\psi(a^*a)=\psi(a)^*\psi(a)$ for all $a\in A$.

Our next result is stated for product type actions as in \autoref{eg:deltanu}, but in
\autoref{cor:DimKGamenable} we will see that the same result is valid for an arbitrary
McDuff action, as long as the acting group is amenable.

\begin{thm} \label{thm:McDuffFromFiberToBundle}
Let $\mathcal{M}$ be a strictly separable $W^*$-bundle over a compact metrizable space $K$,
let $G$ be a countable discrete group, let $\nu\colon G\to \U(M_{d})$ be a unitary representation,
and let $\gamma\colon G\to\Aut(\mathcal{M})$ be a fiber-wise
action. Suppose that $\dim(K)<\I$. Then the following are equivalent:
\be
\item $(\mathcal{M},\gamma)$ is $\delta^\nu$-McDuff;
\item for each $\lambda\in K$, the fiber $(\mathcal{M}_\lambda, \gamma^\lambda)$ is $\delta^\nu$-McDuff.
\ee
\end{thm}
\begin{proof} That (1) implies (2) follows from the existence of a canonical unital and equivariant
homomorphism $\M^\omega\cap \M'\to \M^\omega_\lambda\cap \M'_\lambda$ which is induced by the quotient
map $\M\to \M_\lambda$. Hence, if $\M^\omega\cap \M'$ admits a unital and equivariant homomorphism from $(\R,\delta^\nu)$, then
so does $\M^\omega_\lambda\cap \M'_\lambda$. (Note that this implication holds for any McDuff action, not necessarily
coming from unitary representations of $G$.)

We prove the converse. By \autoref{thm:W*BdleAbsR}, it is enough to construct, a unital equivariant homomorphism
$(M_{d},\Ad(\nu))\to (\M^\omega\cap \M',\gamma^\omega)$. We follow a strategy similar to that used in~\cite{KirRor_central_2014}; see also~\cite{Lia_rokhlin_2016}. The
proof will be divided into a pair of claims, resembling Proposition~7.4 and Lemma~7.5 in~\cite{KirRor_central_2014}.

For each $\lambda\in K$, we use \autoref{thm:W*BdleAbsR} to find a unital homomorphism
\[\theta_\lambda\colon (M_d,\Ad(\nu))\to (\M_\lambda^\omega\cap \M_\lambda',(\gamma^\lambda)^\omega).\]

\textbf{Claim 1:} Let $\ep>0$, let $m\in\N$, let $F\subseteq G$ be a finite subset, and let
$S\subseteq \M$ be a $\|\cdot\|_{2,u}$-compact subset consisting of contractions. Then there exist an open cover $\U$ of $K$,
and families $\Phi^{(l)}=\{\theta_1^{(l)},\ldots,\theta_{r_l}^{(l)}\}$, for $l=1,\ldots,m$, consisting of unital completely
positive contractive maps $\theta_j^{(l)}\colon M_d\to\M$, for $j=1,\ldots,r_l$, such that for all $U\in\U$ and all
$l=1,\ldots,m$, there exists $j\in \{1,\ldots,r_l\}$ satisfying
\be
\item[(1.a)] $\sup\limits_{\tau\in U}\left\| \theta_j^{(l)}(a)s-s\theta_j^{(l)}(a) \right\|_{2,\tau}<\ep$ for all $1=1,\ldots,m$, all $j=1,\ldots, r_l$, all $a\in M_d$
with $\|a\|\leq 1$, and all $s\in S$;
\item[(1.b)] $\sup\limits_{\tau\in U}\left\| \theta_j^{(l)}(a)\theta^{(k)}_i(b)-\theta^{(k)}_i(b)\theta_j^{(l)}(a) \right\|_{2,\tau}<\ep$ for all $l,k=1,\ldots,m$ with $l\neq k$,
all $j=1,\ldots, r_l$, for all $i=1,\ldots, r_k$, and for all $a,b\in M_d$ with $\|a\|, \|b\| \leq 1$;
\item[(1.c)] $\sup\limits_{\tau\in U}\left\| \theta_j^{(l)}(a^*a)-\theta^{(l)}_j(a)^*\theta_j^{(l)}(a) \right\|_{2,\tau}<\ep$ for all $l=1,\ldots,m$, for all $j=1,\ldots, r_l$,
and for all $a\in M_d$ with $\|a\|\leq 1$;
\item[(1.d)] $\sup\limits_{\tau\in U}\left\| \gamma_g(\theta_j^{(l)}(a))-\theta_j^{(l)}(\nu_ga\nu_g^*) \right\|_{2,\tau}<\ep$ for all $g\in F$, for all $l=1,\ldots,m$, for all $j=1,\ldots, r_l$,
and for all $a\in M_d$ with $\|a\| \leq 1$;
\ee
(The tuple $(\U; \Phi^{(1)},\ldots,\Phi^{(m)})$ is an equivariant analog of an \emph{$(\ep,S)$-commuting
covering system}; see Definition~7.1 in~\cite{KirRor_central_2014}.)

To prove the claim, fix $\lambda\in K$, and let $\widetilde{S}\subseteq \mathcal{M}$ be a $\|\cdot\|_{2,u}$-compact subset consisting
of contractions.
Use the Choi-Effros lifting theorem for the quotient map $\ell^\I(\N, \M_\lambda)\to \M_\lambda^\omega$
to find a sequence $(\theta_\lambda^{(n)})_{n\in\N}$ of
unital completely positive contractive maps $\theta_\lambda^{(n)}\colon M_d \to \M_\lambda$
that lifts $\theta_\lambda$.
Use the Choi-Effros lifting theorem again for the quotient map $\M\to \M_\lambda$ to lift each 
map $\theta_\lambda^{(n)}$ to a unital, completely positive map 
$\varphi_\lambda^{(n)}\colon M_d \to \M$. By
choosing a map far enough in the sequence, we may find a unital completely positive
contractive map $\varphi_\lambda\colon M_d\to \M$ satisfying
\bi
\item $\|\varphi_\lambda(a)s-s\varphi_\lambda(a)\|_{2,\lambda}<\ep$ for all $a\in M_d$ with $\|a\|\leq 1$ and for all $s\in
\widetilde{S}$;
\item $\|\varphi_\lambda(a^*a)-\varphi_\lambda(a)^*\varphi_\lambda(a)\|_{2,\lambda}<\ep$ for all $a\in M_d$ with $\|a\|\leq 1$;
\item $\|\gamma_g(\varphi_\lambda(a))-\varphi_\lambda(\nu_ga\nu_g^*)\|_{2,\lambda}<\ep$ for all $g\in F$, and for all $a\in M_d$ with $\|a\|\leq 1$.
\ei

By compactness of the unit ball of $M_d$ and of $\widetilde{S}$, and by continuity of the 2-norm, we can find an
open set $U_\lambda$ of $K$ containing $\lambda$ such the estimates above hold with respect to
the $\|\cdot\|_{2,\tau}$-norm for all $\tau\in U_\lambda$.

We finish the proof of the claim by induction on $m$. When $m=1$, we cover $K$ by the open sets $U_\lambda$ obtained
in the previous paragraph for $\widetilde{S}=S$, and find
an integer $r_1\in\N$ and points $\lambda_1,\ldots,\lambda_{r_1}\in K$, such that
$\U=\{U_{\lambda_1},\ldots,U_{\lambda_{r_1}}\}$ is a cover of $K$. Then $\U$ and
$\Phi^{(1)}=\{\varphi_{\lambda_1},\ldots,\varphi_{\lambda_{r_1}}\}$ satisfy the desired properties.

Assume that we have found an open cover $\mathcal{V}$ and families $\Phi^{(j)}$ for $j=1,\ldots,m-1$,
satisfying the conditions in the statement. Denote by $B$ the unit ball of $M_d$.
We let $\varphi_\lambda\colon M_d\to \M$ and
$W_\lambda$ be as in the first part of this proof, for
\[\widetilde{S}=S\cup \bigcup\limits_{l=1}^{m-1}\bigcup\limits_{j=1}^{r_l}\varphi^{(l)}_j(B).\]
Find an integer $r_l\in\N$ and points $\lambda^{(l)}_1,\ldots,\lambda^{(l)}_{r_l}\in K$, such that
$\U=\{U_{\lambda^{(l)}_1},\ldots,U_{\lambda^{(l)}_{r_l}}\}$ is a cover of $K$. Let $\U$ be the
family of sets of the form $V\cap W_{\lambda^{(l)}_j}$, for $j=1,\ldots, r_l$, and set
$\Phi^{(m+1)}=\{\varphi_{\lambda_1^{(l)}},\ldots,\varphi_{\lambda_{r_l}^{(l)}}\}$. It is straightforward
to check that these satisfy the desired properties.
\newline

For the next claim, we set $m=\dim(K)$.

\textbf{Claim 2:}
Let $\ep>0$, let $F\subseteq G$ be a finite subset, and let
$S\subseteq \M$ be a $\|\cdot\|_{2,u}$-compact subset consisting of contractions. Then there exist
completely positive contractive maps $\psi^{(0)},\ldots,\psi^{(m)}\colon M_d\to \M$ with $\sum\limits_{j=0}^m\psi^{(j)}(1)=1$, satisfying
\be
\item[(2.a)] $\|\psi^{(j)}(a)s-s\psi^{(j)}(a)\|_{2,u}<\ep$ for all $j=0,\ldots,m$, for all $a\in M_d$ with $\|a\|\leq 1$, and for all $s\in S$;
\item[(2.b)] $\|\psi^{(j)}(a)\psi^{(k)}(b)-\psi^{(k)}(b)\psi^{(j)}(a)\|_{2,u}<\ep$ for all $j,k=0,\ldots,m$ with $j\neq k$, and for all $a,b\in M_d$ with $\|a\|,\|b\|\leq 1$;
\item[(2.c)] $\|\psi^{(j)}(a^*a)-\psi^{(j)}(a)^*\psi^{(j)}(a)\|_{2,u}<\ep$ for all $j=0,\ldots,m$, and for all $a\in M_d$ with $\|a\|\leq 1$;
\item[(2.d)] $\|\gamma_g(\psi^{(j)}(a))-\psi^{(j)}(\nu_ga\nu_g^*) \|_{2,u}<\ep$ for all $g\in F$, for all $j=0,\ldots,m$, and for all $a\in M_d$ with $\|a\|\leq 1$.
\ee

Let $\U$ be an open cover and let $\Phi^{(0)},\ldots,\Phi^{(m)}$ be as in the previous claim.
Since covering dimension and decomposition dimension agree for compact metric spaces (see Lemma~3.2 in~\cite{KirWin_covering_2004}),
there exists a refinement $\U'$ of $\U$, such that $\U'$ is the union of $m+1$ finite subsets
$\U_0,\ldots, \U_m$, with $\U_j=\{U_1^{(j)},\ldots, U_{s_j}^{(j)}\}$, consisting of pairwise disjoint
open subsets of $K$.

Let $\{f_k^{(j)}\colon j=0,\ldots,m; k=1,\ldots, s_j\}$ be a partition of unity
of $K$ subordinate to $\U'$, with $\supp(f_k^{(j)})\subseteq U_k^{(j)}$ for
$j=0,\ldots,m$ and $k=1,\ldots,s_j$. Observe
that for fixed $j$, the functions $f_1^{(j)},\ldots,f_{s_j}^{(j)}$ are pairwise
orthogonal. We regard these
functions as elements in $C(K)\subseteq \M$, and observe that they are left fixed by
$\gamma$, since $\gamma$ is a fiber-wise action (and, in particular, trivial on $C(K)$).

Fix $j\in \{0,\ldots,m\}$. For $k=1,\ldots,s_j$,
let $\varphi_k^{(j)}\colon M_d\to \M$ be a unital
completely positive map belonging to $\Phi^{(j)}$.
Define a linear map $\psi^{(j)}\colon M_d\to \M$ by
\[\psi^{(j)}(a)=\sum_{k=1}^{s_j} f_k^{(j)}\varphi^{(j)}_k(a)\]
for all $a\in M_d$. Since the functions $f_k^{(j)}$ belong to the center of $\M$,
they in particular commute with the images of the maps $\varphi^{(j)}_k$, and thus
it follows that $\psi^{(j)}$ is completely positive and contractive.

It remains to show that $\psi^{(0)},\ldots,\psi^{(m)}$ satisfy the conditions of the
claim. Since the verification of conditions (1), (2) and (3) is similar to the verification
of conditions (i), (iii) and (iv) in the proof of Lemma~7.5 of~\cite{KirRor_central_2014}, we will only check
condition (4). Fix $g\in F$, an index $j\in \{0,\ldots,m\}$, and a contraction $a\in M_d$.
Using that $\gamma_g(f_k^{(j)})=f_k^{(j)}$ for all $k=1,\ldots, s_j$ at the first step,
and that these contractions are orthogonal at the second step, we get
\begin{align*}
\|\gamma_g(\psi^{(j)}(a))-\psi^{(j)}(\nu_ga\nu_g^*)\|_{2,u} &=\left\|\sum_{k=1}^{s_j} f_k^{(j)}\left(\gamma_g(\varphi^{(j)}_k(a))-\varphi^{(j)}_k(\nu_ga\nu_g^*)\right)\right\| \\
&=\max_{k=1,\ldots,s_j} \left\|\gamma_g(\varphi^{(j)}_k(a))-\varphi^{(j)}_k(\nu_ga\nu_g^*)\right\|<\ep,
\end{align*}
as desired. This finishes the proof of the claim.
\newline

To finish the proof of the theorem, we choose a countable set $\{x_n\}_{n\in\N}$ which is $\|\cdot\|_{2,u}$-dense in $\M$, and an
increasing sequence $(F_n)_{n\in\N}$ of finite subsets of $G$ whose union equals $G$. Using the previous claim, we find
completely positive contractive maps $\psi_n^{(0)},\ldots,\psi_n^{(m)}\colon M_d\to \M$ satisfying conditions (2.a) through (2.d) for $\ep=1/n$, $F=F_n$,
and $S_n=\{x_1,\ldots,x_n\}$. For $j=0,\ldots,m$, let $\psi_j\colon M_d\to \M^\omega$ be the map determined by $(\psi^{(j)}_n)_{n\in\N}$.
It is then a routine exercise to check that these are equivariant completely positive contractive maps of order zero, with commuting ranges
that are contained in $\M^\omega\cap \M'$,
satisfying $\sum\limits_{j=0}^m\psi_j(1)=1$.
By \autoref{prop:EqOzCommRges},
there is a unital equivariant homomorphism $(M_d,\Ad(\nu))\to (\M^\omega\cap \M',\gamma^\omega)$, and the result follows from \autoref{thm:W*BdleAbsR}.
\end{proof}

Next, we deduce that the equivalence in \autoref{thm:MoRabsorbsMcDuff} holds for McDuff actions that absorb $\id_{\R}$,
even if they are not of product type. In particular, this is the case for \emph{any} McDuff action of an amenable group.

\begin{cor}\label{cor:DimKGamenable}
Let $\mathcal{M}$ be a strictly separable $W^*$-bundle over a compact metrizable space $K$,
let $G$ be a countable discrete group, let $\delta\colon G\to\Aut(\R)$ be a McDuff action,
and let $\gamma\colon G\to\Aut(\mathcal{M})$ be a fiber-wise
action. Suppose that $\dim(K)<\I$. If $\delta$ absorbs $\id_{\R}$, that is, if $\delta\otimes\id_{\R}$ is cocycle conjugate to 
$\delta$, then the following are equivalent:
\be
\item $(\mathcal{M},\gamma)$ is $\delta$-McDuff;
\item for each $\lambda\in K$, the fiber $(\mathcal{M}_\lambda, \gamma^\lambda)$ is $\delta$-McDuff.
\ee
When $G$ is amenable, $\delta$ is automatically $\id_{\R}$-absorbing.
\end{cor}
\begin{proof}
It is clear that (1) implies (2), so assume that (2) holds, and assume that $\delta$ is $\id_{\R}$-absorbing.
Then every fiber of $(\M,\gamma)$ is $\id_{\R}$-McDuff,
and since $\id_{\R}$ is a product type action as in \autoref{eg:deltanu}, \autoref{thm:McDuffFromFiberToBundle} implies
that $(\M,\gamma)$ is itself $\id_{\R}$-McDuff.

On the other hand, by \autoref{thm:MoRabsorbsMcDuff}, the $W^*$-bundle $(\M\overline{\otimes}\R,\gamma\otimes\id_{\R})$
is $\delta$-McDuff. Combining all of the above,
we obtain the following chain of cocycle conjugacies (denoted $\cong_{\mathrm{cc}}$):
\[(\M,\gamma)\cong_{\mathrm{cc}} (\M\overline{\otimes}\R,\gamma\otimes\id_{\R}) \cong_{\mathrm{cc}}
 (\M\overline{\otimes}\R\overline{\otimes}\R,\gamma\otimes\id_{\R}\otimes\delta)\cong_{\mathrm{cc}}  (\M\overline{\otimes}\R,\gamma\otimes\delta),
\]
which finishes the proof.

Ocneanu has shown in \cite{Ocn_actions_1985} that any amenable
group action on $\R$ is $\id_{\R}$-McDuff, which justifies
the last claim.
\end{proof}

The previous corollary will be used in the proof of \autoref{thm:EqJiangSuAbs} to show that
under very general conditions, any action of an amenable group on a $\mathcal{Z}$-stable \ca\ absorbs
the trivial action on $\mathcal{Z}$ tensorially.

\section{The weak tracial Rokhlin property}
\subsection{Pavings of amenable groups}
Let $G$ be a countable discrete group. Given a finite subset $K\subseteq G$ and $\ep>0$, we say that a finite set
$S\subseteq G$ is \emph{$(K,\ep)$-invariant} if
\[\left|S\cap \bigcap_{g\in K} gS\right| \geq (1-\ep) |S|.\]
Recall that, by a result of F\o lner, $G$ is amenable if and only if for every finite subset $K\subseteq G$ and every $\ep>0$, there exists
a nonempty $(K,\ep)$-invariant subset of $G$.

\begin{df}
Let $G$ be a discrete group and let $\ep>0$. A family $(S_j)_{j\in J}$ of finite subsets of $G$ is
said to be \emph{$\ep$-disjoint} if there exist subsets $T_j\subseteq S_j$, for $j\in J$, such that
\[|T_j|\geq (1-\ep)|S_j| \ \ \mbox{ and } \ \ T_j\cap T_k=\emptyset \ \mbox{ whenever } j\neq k.\]

Let $F\subseteq G$ be a finite subset.
A family $S_1,\ldots, S_N$ of finite subsets of $G$ is said to \emph{$\ep$-pave} the set $F$ if there
are finite subsets $L_1,\ldots,L_N$ of $G$ such that:
\be\item[(a)] $\bigcup_{j=1}^NS_jL_j\subseteq F$;
\item[(b)] the sets $S_jL_j$, for $j=1,\ldots,N$, are pairwise disjoint;
\item[(c)] $|F\setminus\bigcup_{j=1}^NS_jL_j|<\ep |F|$;
\item[(d)] for each $j=1,\ldots,N$, the sets $(S_j\ell)_{\ell\in L_j}$ are $\ep$-disjoint.\ee
\end{df}

\begin{df}\label{df:pavingsystem}
Let $S_1,\ldots,S_N$ be finite subsets of a discrete group $G$. Given $\ep>0$, we say that
the sets $S_1,\ldots,S_N$ are an \emph{$\ep$-paving system} if there exist $\delta>0$ and
a finite subset $K\subseteq G$ such that $S_1,\ldots,S_N$ $\ep$-pave any $(K,\delta)$-invariant set.
\end{df}

The existence of paving systems is guaranteed by the following result of Ornstein and Weiss; see Theorem~6
in~\cite{OrnWei_entropy_1987}.

\begin{thm}\label{thm:OWPavSyst}
Let $G$ be an amenable group and let $\ep>0$. Then there exist $N\in\N$ such that for any
$\delta>0$ and every finite set $K\subseteq G$, there is an $\ep$-paving system $S_1,\ldots,S_N$
of $G$, with each $S_j$ being $(K,\delta)$-invariant, and such that the unit of $G$ belongs to $S_1$.
\end{thm}

The above theorem, except for the very last condition, is proved in Section~3 of~\cite{Ocn_actions_1985}, and
the proof given there shows that one can always assume that $S_1$ contains the identity of $G$.
(See, specifically, the construction of the
sets $S_1,\ldots,S_N$ given at the bottom of page 17 in \cite{Ocn_actions_1985}.)

\subsection{The weak tracial Rokhlin property for actions of amenable groups}
If $A$ is a unital \ca\ and $\omega$ is a free ultrafilter over $\N$, then every trace $\tau$ on $A$ extends canonically to a
trace on $A_\omega$, which we denote by $\tau_\omega$. We write $J_{A}$ for the \emph{trace-kernel ideal} in $A_\omega$ (see Definition~4.3
in~\cite{KirRor_central_2014}), that is,
\[J_A=\left\{b=\left[(b_n)_{n\in\N}\right] \in A_\omega \colon \lim_{n\to\omega}\sup_{\tau\in T(A)}\tau(b_n^*b_n)=0\right\}.\]

Below is the definition of the \wtRp\ with which we will work in the present paper.
It is formally stronger than the one given by Wang in Definition~2.1 in~\cite{Wan_tracial_2013}, since we assume the positive contractions to
exist for \emph{any} paving family, and not just for some.
(We also use traces instead of Cuntz comparison, but this difference
is not as significant.) It will ultimately follow from \autoref{thm:wtRp} that, in the context of this
theorem, our definition and Wang's are in fact equivalent (and equivalent to strong outerness).

\begin{df}\label{df:wtRp}
Let $G$ be an amenable group, let $A$ be a simple, separable \uca, and let $\alpha\colon G\to\Aut(A)$ be an
action. We say that $\alpha$ has the \emph{weak tracial Rokhlin property} if
for every 
paving family $S_1,\ldots,S_N$ of subsets of $G$, there exist
positive contractions
$f_{\ell,g} \in A_\omega\cap A'$, for $\ell=1,\ldots,N$ and $g\in S_\ell$, satisfying:
\be
\item[(a)] $(\alpha_\omega)_{gh^{-1}}(f_{\ell, h})=f_{\ell,g}$ for all $g,h\in S_\ell$ and for all $\ell=1,\ldots,N$;
\item[(b)] $f_{\ell,g}f_{k,h}=0$ for all $\ell,k=1,\ldots,n$, for all $g\in S_\ell$ and for all $h\in S_k$ with $(\ell,g)\neq (k,h)$;
\item[(c)] $f_{\ell,g}(\alpha_\omega)_h(f_{k,r})=(\alpha_\omega)_h(f_{k,r})f_{\ell,g}$ for all $\ell,k=1,\ldots,n$, all $g\in S_\ell$ and all $r\in S_k$, and all $h\in G$;
\item[(d)] $1-\sum\limits_{\ell=1}^N\sum\limits_{g\in S_\ell} f_{\ell, g}$ belongs to $J_A$;
\item[(e)] For $\tau\in T(A)$, for $\ell=1,\ldots,N$ and for $g\in S_\ell$, the value of $\tau(f_{\ell,g})$ is
independent of $\tau$ and $g$, and is positive.
\ee

We say that $\alpha$ has the \emph{tracial Rokhlin property}, if the positive contractions $f_{\ell, g}$ above can be
chosen to be projections.
\end{df}

Condition (e) is a weakening of the uniformity condition considered by Matui-Sato in~\cite{MatSat_stability_2014}.
Conditions (a) through (d) are inspired by Ocneanu's characterization of outerness for amenable group actions on the hyperfinite
II$_1$-factor; see \cite{Ocn_actions_1985}. For later use, we isolate Ocneanu's condition in the following definition.

\begin{df}\label{df:W*Rp}
Let $G$ be an amenable group, let $\mathcal{M}$ be a $W^*$-bundle, and let $\gamma\colon G\to\Aut(\mathcal{M})$ be a fiber-wise
action. We say that $\gamma$ has the \emph{$W^*$-Rokhlin property} if
for every paving family $S_1,\ldots,S_N$ of subsets of $G$, there exist
projections
$p_{\ell,g} \in \mathcal{M}^\omega\cap \mathcal{M}'$, for $\ell=1,\ldots,N$ and $g\in S_\ell$, satisfying:
\begin{enumerate}[label=(\alph*)]
\item 
\label{df:W*Rp:(a)}
$\gamma^\omega_{gh^{-1}}(p_{\ell, h})=p_{\ell,g}$ for all $g,h\in S_\ell$ and for all $\ell=1,\ldots,N$;
\item 
\label{df:W*Rp:(c)}
$p_{\ell,g}\gamma^\omega_h(p_{k,r})=\gamma^\omega_h(p_{k,r})p_{\ell,g}$ for all $\ell,k=1,\ldots,N$, all $g\in S_\ell$ and all $r\in S_k$, and all $h\in G$;
\item 
\label{df:W*Rp:(d)}
$\sum\limits_{\ell=1}^N\sum\limits_{g\in S_\ell} p_{\ell, g}=1$.
\item 
\label{df:W*Rp:(e)}
$\tau_{\R^\omega}(p_{\ell,g}) = \tau_{\R^\omega}(p_{\ell,h}) > 0$ for all $g,h\in S_\ell$ and all $\ell=1,\ldots,N$.
\end{enumerate}
\end{df}
Note that condition (\ref{df:W*Rp:(d)}) in \autoref{df:W*Rp} implies that $p_{\ell,g}p_{k,h}=0$ for all $\ell,k=1,\ldots,N$, all $g\in S_\ell$ and all $h\in S_k$ with $(\ell,g)\neq (k,h)$.


\begin{eg}\label{eg:muG}
In \cite{Ocn_actions_1985}, Ocneanu showed that given an amenable group $G$, there is a unique (up to cocycle conjugacy) outer action of $G$
on $\mathcal{R}$. We fix such an action and denote it by $\mu_G\colon G\to \Aut(\R)$, or just by $\mu$ when $G$ is
understood. Then $\mu_G$ has the $W^*$-Rokhlin property in the sense of the previous definition. 
To see this, we rely on Ocneanu's construction presented in \cite[Chapter 4]{Ocn_actions_1985}, where $\mu_G$ is defined 
as a limit of conjugations by unitaries (see \cite[Section 4.4]{Ocn_actions_1985}). 
In particular, condition (\ref{df:W*Rp:(a)}) from \autoref{df:W*Rp} follows from Ocneanu's construction of ``approximate left translations'' $L_g$ 
on $G$, and specifically from the identity $L_{gh^{-1}}(h)=g$ for $g,h\in S_\ell$ and $\ell=1,\ldots,N$; see the first paragraph of 
page 21 of \cite{Ocn_actions_1985}. The same argument shows that condition (\ref{df:W*Rp:(e)}) is also satisfied, since the projections $p_{\ell,g}$
are the diagonal projections in some matrix algebra associated to the elements of the given paving. Conditions (\ref{df:W*Rp:(c)}) and (\ref{df:W*Rp:(d)})
are explicitly verified in Theorem~6.1 in~\cite{Ocn_actions_1985}. 
\end{eg}

\begin{rem}\label{rem:muG}
Adopting the notation from the previous example, it follows from the work of Ocneanu that $\mu_G$ is McDuff. Moreover, one 
immediately checks that any action that absorbs $\mu_G$ has the $W^*$-Rokhlin property, a fact that we will use repeatedly. If $H$ is a subgroup of $G$, then the restriction of $\mu_G$ to $H$ is clearly outer, and hence cocycle
conjugate to $\mu_H$. 
\end{rem}

The goal of this section is to prove the equivalence of conditions (1) and (2) in Theorem~\ref{thmintro:Equiv}. In fact, we prove
a slightly more general version in which nuclearity of $A$ is replaced by the condition that all of its weak closures with respect to
traces be hyperfinite. 

\begin{thm}\label{thm:wtRp}
Let $G$ be a countable amenable group, let $A$ be a separable, simple, \uca, and let $\alpha\colon G\to\Aut(A)$ be an action.
Suppose that $T(A)$ is a nonempty Bauer simplex,
and that the induced action of $G$ on $\partial_eT(A)$ has finite orbits and Hausdorff orbit space.
Finally, assume that $\overline{A}^\tau$ is hyperfinite for all $\tau\in \partial_eT(A)$.
Then the following are equivalent:
\be
\item $\alpha$ is strongly outer;
\item $\alpha\otimes\id_{\mathcal{Z}}$ has the \wtRp.\ee
\end{thm}

The first precursor of this result is Theorem~5.5 in~\cite{EchLucPhiWal_structure_2010}, where the above result is shown under the additional
assumptions that $A$ has tracial rank zero, $A$ has a unique tracial state, and $G$ is finite. More recently, Matui and Sato gave
a proof of \autoref{thm:wtRp} in the case that $A$ is nuclear and has finitely many extremal tracial states and the group
$G$ is elementary amenable; see Theorem~3.6 in~\cite{MatSat_stability_2014}. Here, we remove all the assumptions on $G$, and significantly
relax the conditions on $T(A)$. Our main innovation is the systematic use of $W^*$-bundles in the equivariant setting.

We briefly describe our strategy for (1) $\Rightarrow$ (2), which is the difficult part. 
Adopt the notation from \autoref{eg:EqW*bundle}.
Strong outerness of $\alpha$ implies that the induced action of $G$ on each fiber of $(\mathcal{M}, K/G)$ has the $W^*$-Rokhlin
property; see \autoref{prop:FiberW*Rp}. The first step is to show that the bundle action $\gamma\colon G\to\Aut(\mathcal{M})$ has the $W^*$-Rokhlin property,
and this is obtained as a consequence of \autoref{thm:McDuffFromFiberToBundle}, using absorption
of the canonical action $\mu_G$ of $G$ on $\R$ with the $W^*$-Rokhlin property. The projections coming from the $W^*$-Rokhlin property can be approximated by positive contractions
in $A$, and these elements will satisfy the conditions in \autoref{df:wtRp} with respect to the uniform 2-norm, instead of the given norm
on $A$. We use the fact that the trace ideal in $A_\omega$ is equivariantly a $\sigma$-ideal (see \autoref{prop:JAwSigmaId}) to obtain
corrected elements which verify \autoref{df:wtRp}.

We need the following analog of Kirchberg's notion of a $\sigma$-ideal in the equivariant setting. For $\Z$-actions, this notion
was already considered in~\cite{Lia_rokhlin_2016}.

\begin{df}\label{df:EqSigmaIdeal}
Let $G$ be a discrete group, let $B$ be a \ca, let $\beta\colon G\to\Aut(B)$ be an action, and let $J\subseteq B$ be an ideal satisfying $\beta_g(J)=J$
for all $g\in G$. We say that $J$ is an \emph{equivariant $\sigma$-ideal (with respect to $\beta$)}, if for every separable $\beta$-invariant
subalgebra $C\subseteq B$, there exists a positive contraction $x\in (J\cap C')^\beta$ satisfying $xc=c$ for all $c\in C\cap J$.
\end{df}

It is easy to see that if a finite group acts on a \ca, then any $\sigma$-ideal is automatically an equivariant $\sigma$-ideal: one just
averages the positive contraction in the definition of a $\sigma$-ideal to obtain a fixed one. When the group is amenable, an exact
averaging is not possible, but this is good enough to get equivariant $\sigma$-ideals in sequence algebras, as we show below.

\begin{prop}\label{prop:JAwSigmaId}
Let $A$ be a unital \ca, let $G$ be a countable discrete amenable group, let $\alpha\colon G\to\Aut(A)$ be any action, and let
$\omega$ be a free ultrafilter over $\N$.
Then the trace ideal $J_{A}$ is an equivariant $\sigma$-ideal in $A_\omega$.
\end{prop}
\begin{proof}
We abbreviate $J_A$ to $J$.
Let $C\subseteq A_\omega$ be a separable, $\alpha_\omega$-invariant subalgebra. Since $J$ is a $\sigma$-ideal in $A_\omega$, there exists
a positive contraction $x\in J\cap C'$ satisfying $xc=c$ for all $c\in C\cap J$. By Kirchberg's $\varepsilon$-test,
it is enough to prove that for every finite subset $K\subseteq G$ and every $\varepsilon>0$, there exists a
positive contraction $y\in J\cap C'$ such that $\|(\alpha_\omega)_k(y)-y\|<\ep$ for all $k\in K$ and $yc=c$ for all $c\in C\cap J$.

We fix a finite subset $K\subseteq G$ and $\ep>0$.
Using amenability of $G$, find a
finite subset $F$ of $G$ such that $|kF\triangle F|\leq \frac{\ep}{2} |F|$ for all $k\in K$.
Set $y=\frac{1}{|F|}\sum_{g\in F}(\alpha_\omega)_g(x)$. Then 
$yc = c$ for all $c\in C\cap J$.
For $k\in K$, we have
\begin{align*}
\|(\alpha_\omega)_k(y)-y\| &=\left\| \frac{1}{|F|} \sum_{g \in F} (\alpha_\omega)_{kg}(x)-(\alpha_\omega)_g(x)\right\|\\
&=\left\| \frac{1}{|F|} \sum_{kF\triangle F} (\alpha_\omega)_{kg}(x)-(\alpha_\omega)_g(x)\right\|\\
&= \frac{1}{|F|} \sum_{kF\triangle F} \left\|(\alpha_\omega)_{kg}(x)-(\alpha_\omega)_g(x)\right\|\leq \frac{\ep}{2}2=\ep.
\end{align*}
Since $J$ is an $\alpha_\omega$-invariant ideal, the positive contraction $y$ also belongs to $J$. Finally,
it is also easy to check that $y$ commutes with $C$, since $C$ is also invariant under $\alpha_\omega$. This
concludes the proof.
\end{proof}

Recall (see Definition~XVII.1.1 in~\cite{Tak_theory_2003} and the remark following Theorem 1.2 there) that an automorphism $\varphi$ of a von Neumann algebra
$M$ is said to be \emph{properly outer} if for
every central projection $p\in M$ satisfying $\varphi(p)=p$, the restriction of $\varphi$ to the corner $pM$ is outer.
An action $\gamma\colon G\to \Aut(M)$ is said to be \emph{properly outer} if $\gamma_g$ is properly outer for
all $g\in G\setminus\{1\}$.

\begin{prop}\label{prop:StrOutStrPropOut}
Let $G$ be a discrete group, let $A$ be a separable \ca, and let $\alpha\colon G\to\Aut(A)$ be a strongly outer action.
Then $\overline{\alpha}^\tau_g$ is properly outer for all $\tau\in T(A)^{\alpha_g}$ and for all $g\in G\setminus\{1\}$.
\end{prop}
\begin{proof}
Let $g\in G\setminus\{1\}$, let $\tau\in T(A)^{\alpha_g}$, and let $p\in \overline{A}^\tau$ be a central invariant projection.
We denote by $\overline{\tau}$ the extension of $\tau$ to $\overline{A}^\tau$.
Define a trace $\sigma\in T(A)$ by $\sigma(a)=\overline{\tau}(pa)/\overline{\tau}(p)$ for all $a\in A$. Then
$\sigma$ is $\alpha_g$-invariant.

We claim that $\overline{A}^{\sigma}$ can be naturally identified with $p\overline{A}^\tau$. This is probably
known to the experts, but we include a proof for the sake of completeness. Observe first that the inequality 
$\|\cdot\|_{2,\sigma}\leq \overline{\tau}(p)^{-1/2}\|\cdot\|_{2,\tau}$ follows directly
from the Cauchy-Schwarz inequality.
In particular, if $(a_n)_{n\in\N}$ is a sequence in $A$ which is Cauchy with respect to $\|\cdot\|_{2,\tau}$, then 
$(pa_n)_{n\in\N}$ is Cauchy with respect to $\|\cdot\|_{2,\sigma}$. 
This shows that $p\overline{A}^{\tau}\subseteq \overline{A}^{\sigma}$.

Conversely, note that centrality of $p$ implies that 
\[\|a\|_\sigma=\|pa\|_\sigma=\overline{\tau}(p)^{-1/2}\|pa\|_\tau\]  
for all $a\in A$. 
Given a sequence $(x_n)_{n\in\N}$ in $A$ which is Cauchy with respect to $\|\cdot\|_{2,\sigma}$, it follows from the 
first identity above that $(px_n)_{n\in\N}$ has the same limit in $\overline{A}^\sigma$, while the second identity shows
that $(px_n)_{n\in\N}$ is also Cauchy with respect to $\|\cdot\|_{2,\tau}$. Since its limit belongs to $p\overline{A}^\tau$,
this shows the converse inclusion and proves the claim.

To finish the proof, it suffices to observe that if $\alpha_g$ becomes inner in the corner $p\overline{A}^\tau$, then
it becomes inner in the weak extension with respect to $\sigma$, contradicting strong outerness.
\end{proof}

Next, we verify that strongly outer actions induce actions with the $W^*$-Rokhlin property when passing to the weak
closure with respect to any trace with finite orbit. Note that for tracial von Neumann algebras, the assignment
$M\mapsto M^\omega\cap M'$ commutes with finite direct sums.

\begin{rem}
Let $(\M,\gamma)$ be a $G$-equivariant $W^*$-bundle, and adopt the notation and assumptions from \autoref{eg:EqW*bundle}. 
If $S\subseteq \partial_eT(A)$ is a closed subset, we define the restricted 2-norm $\|\cdot\|_{2,S}$ on $A$
to be 
\[\|a\|_{2,S}=\sup_{\tau\in S}\tau(a^*a)^{1/2}\]
for all $a\in A$. Note that $\|\cdot\|_{2,\partial_eT(A)}=\|\cdot\|_{2,u}$. 
We write $\M_S$ for the $W^*$-bundle obtained by perfoming the procedure described in \autoref{eg:EqW*bundle}
using the norm $\|\cdot\|_{2,S}$ instead of $\|\cdot\|_{2,u}$. When $S$ is $G$-invariant, then $\M_S$ is naturally a 
$G$-equivariant $W^*$-bundle.
\end{rem}

We will need to use some facts about induced $C^*$-algebras; see for example 
Section~3.6 in~\cite{Wil_crossed_2007}, whose notation we will follow. 
We will not prove the results in the greatest possibly generality, and 
will instead restrict to the case we are interested in.

\begin{nota} 
Let $A$ be a \ca, let $G$ be a discrete group, and let $H\leq G$ be
a subgroup with finite index. If $\alpha\colon H\to\Aut(A)$ is an action, 
we set 
\[\mathrm{Ind}_H^G(A,\alpha)=\{f\in C_b(G,A)\colon f(gh)=\alpha_{h^{-1}}(f(g)) \mbox{ for all } g\in G\},\]
with pointwise operations. We endow this algebra 
with the $G$-action $\mathrm{Ind}_H^G(\alpha)$ given by
\[\mathrm{Ind}_H^G(\alpha)_g(f)(k)=f(g^{-1}k)\] 
for all $g,k\in G$ and all $f\in \mathrm{Ind}_H^G(A,\alpha)$.
\end{nota}

It is immediate to check that if $\alpha\colon H\to\Aut(A)$ and 
$\beta\colon H\to\Aut(B)$ are conjugate actions, then so are
$\mathrm{Ind}_H^G(\alpha)\colon G\to \Aut(\mathrm{Ind}_H^G(A,\alpha))$
and 
$\mathrm{Ind}_H^G(\beta)\colon G\to \Aut(\mathrm{Ind}_H^G(B,\beta))$.
The analogous statement for cocycle conjugacy is also true. Since we
were not able to find a reference, and its proof is not straightforward, 
we provide one here for the convenience of the reader. Again, we restrict 
to the case we are interested in and do not prove the most general result.

\begin{prop}\label{prop:IndAlgCocConj}
Let $A$ and $B$ be unital \ca s, let $G$ be a discrete group, and 
let $H\leq G$ be
a subgroup with finite index. Let $\alpha\colon H\to\Aut(A)$ 
and $\beta\colon H\to\Aut(B)$ be cocycle conjugate actions. Then 
$\mathrm{Ind}_H^G(\alpha)$
and 
$\mathrm{Ind}_H^G(\beta)$ are cocycle conjugate. 
\end{prop}
\begin{proof}
It suffices to assume that $A=B$ and that there exist unitaries 
$u_h\in A$, for $h\in H$, satisfying 
\[u_{h_1h_2}=u_{h_1}\alpha_{h_1}(u_{h_2}) \ \ \mbox{ and } \ \ 
 \Ad(u_h)\circ\alpha_h=\beta_h
\]
for all $h_1,h_2,h\in H$. 
Let $s\colon G/H\to G$ be a section for the canonical quotient map $G\to G/H$. 
To lighten the notation, for $g\in G$ we abbreviate $s(gH)$ to $s(g)$ throughout. Note that $s(gh)=s(g)$ for all $g\in G$ and all $h\in H$.
Let $w\in C_b(G,A)$ be the unitary given by 
$w(g)=u_{g^{-1}s(g)}$ for all $g\in G$. 

\textbf{Claim 1:} \emph{$\Ad(w)$ maps $\mathrm{Ind}_H^G(A,\alpha)$ to
$\mathrm{Ind}_H^G(A,\beta)$}. To prove this, let $f\in \mathrm{Ind}_H^G(A,\alpha)$, let $g\in G$ and let $h\in H$. Then
\begin{align*}
(wfw^*)(gh)&= w(gh) f(gh) w(gh)^*\\
&= u_{h^{-1}g^{-1}s(g)}\alpha_{h^{-1}}(f(g)) u_{h^{-1}g^{-1}s(g)}^*\\
&= u_{h^{-1}} \alpha_{h^{-1}}(u_{g^{-1}s(g)}f(g)u_{g^{-1}s(g)}^*)u_{h^{-1}}^*\\
&= \beta_{h^{-1}}((w fw^*)(g)),
\end{align*}
as desired. 

Denote by $\texttt{Lt}\colon C_b(G,A)\to C_b(G,A)$ the action of left
translation, and note that $\mathrm{Ind}_H^G(\alpha)$ is the restriction
of $\texttt{Lt}$ to the invariant subalgebra $\mathrm{Ind}_H^G(A,\alpha)$,
and similarly for $\mathrm{Ind}_H^G(\beta)$.
We define a $\texttt{Lt}$-cocycle $v\colon G\to \mathcal{U}(C_b(G,A))$ by
\[v_g=w^*\texttt{Lt}_g(w)\]
for all $g\in G$.

\textbf{Claim 2:} \emph{$v_g$ belongs to $\mathrm{Ind}_H^G(A,\alpha)$ for 
all $g\in G$.} For $g,k\in G$ and $h\in H$, we have 
\begin{align*}
v_g(kh)&= w^*(kh)w(g^{-1}kh)\\
&= u_{h^{-1}k^{-1}s(k)}^*u_{h^{-1}k^{-1}gs(g^{-1}k)}\\
&= (u_{h^{-1}}\alpha_{h^{-1}}(u_{k^{-1}s(k)}))^*u_{h^{-1}}\alpha_{h^{-1}}(u_{k^{-1}gs(g^{-1}k)})\\
&= \alpha_{h^{-1}}(u_{k^{-1}s(k)}^*u_{k^{-1}gs(g^{-1}k)})\\
&= \alpha_{h^{-1}}(v_g(k)),
\end{align*}
as desired.

Denote by $\varphi\colon \mathrm{Ind}_H^G(A,\alpha)\to \mathrm{Ind}_H^G(A,\beta)$ the isomorphism given by $\varphi(f)=\Ad(w)(f)$ for all
$f\in \mathrm{Ind}_H^G(A,\alpha)$; see Claim~1. 
By Claim~2 and the fact that $\mathrm{Ind}_H^G(\alpha)$ is the restriction
of $\texttt{Lt}$ to $\mathrm{Ind}_H^G(A,\alpha)$, it follows that
$v$ is a 1-cocycle for $\mathrm{Ind}_H^G(A,\alpha)$. 

\textbf{Claim 3:} \emph{for all $g\in G$, we have}
\[
\varphi\circ \Ad(v_g)\circ \mathrm{Ind}_H^G(\alpha)_g =\mathrm{Ind}_H^G(\beta)_g.
\]
Given $g\in G$ and $f\in \mathrm{Ind}_H^G(A,\alpha)$, we have
\begin{align*}
\varphi(\Ad(v_g)(\mathrm{Ind}_H^G(A,\alpha)_g(f))) &= 
w(w^*\texttt{Lt}_g(w)\texttt{Lt}_g(f)\texttt{Lt}_g(w)^*w)w^*\\
&=\texttt{Lt}_g(w fw^*)\\
&= \mathrm{Ind}_H^G(\beta)_g(\varphi(f)),
\end{align*}
as desired. It follows that $\mathrm{Ind}_H^G(\alpha)$ and $\mathrm{Ind}_H^G(\beta)$ are cocycle conjugate. 
\end{proof}

\begin{prop}\label{prop:FiberW*Rp}
Let $G$ be a countable amenable group, let $A$ be a separable, \uca, and let $\alpha\colon G\to\Aut(A)$ be a
strongly outer action.
Let $\tau$ be an extreme trace on $A$ satisfying $\overline{A}^\tau\cong \mathcal{R}$ and suppose that the $G$-orbit $G\cdot \tau$ of 
$\tau$ is finite. Then the weak extension $\gamma_\tau$ of $\alpha$ to $\mathcal{M}_{G\cdot \tau}$ is $\mu_G$-McDuff.
\end{prop}
\begin{proof}
Set $H=\mathrm{Stab}(\tau)$, which is a subgroup of $G$ with finite index.
Endow the finite set $G/H$ with its canonical $G$-action by left translation.
Denote by $\pi\colon \M_{G\cdot\tau}\to M_\tau$ the canonical quotient
map. 
Then there is an equivariant bijection 
$\sigma\colon \mathrm{Prim}(\M_{G\cdot \tau})\to G/H$ satisfying 
$\sigma(\ker(\pi))=H$.
We abbreviate $\gamma_\tau$ to $\gamma$, and write 
$\gamma^H \colon H\to\Aut(\M_\tau)$ for the induced action.

By Proposition~3.53 in~\cite{Wil_crossed_2007}, the map 
\[\varphi\colon (\M_{G\cdot\tau},\gamma)\to (\mathrm{Ind}_H^G(\M_\tau,\gamma^H),\mathrm{Ind}_H^G(\gamma^H))\]
given by $\varphi(a)(g)=\pi(\gamma_{g^{-1}}(a))$ for all $a\in \M_{G\cdot\tau}$
and all $g\in G$, is an equivariant isomorphism. 

Note that $\gamma$ has the $W^*$-Rokhlin property by 
Ocneanu's noncommutative
Rokhlin theorem (Theorem~6.1 in~\cite{Ocn_actions_1985}), 
thanks to \autoref{prop:StrOutStrPropOut}. In particular, $\gamma^H$ is
an outer action of $H$ on $\M_\tau\cong \mathcal{R}$. 
Let $\mu_G\colon G\to\Aut(\mathcal{R})$ be the outer action described in \autoref{eg:muG}, and note that the restriction of $\mu_G$ to $H$, which
is $\mu_H$, is cocycle conjugate to $\gamma^H$. By \autoref{prop:IndAlgCocConj},
there is a cocycle conjugacy of $G$-dynamical systems
\[(\mathrm{Ind}_H^G(\M_\tau,\gamma^H),\mathrm{Ind}_H^G(\gamma^H))\cong_{\mathrm{cc}}(\mathrm{Ind}_H^G(\R,\mu_H),\mathrm{Ind}_H^G(\mu_H)).\]

Since $\mu_H$ is the restriction of $\mu_G$ to $H$, the last part of 
Example~3.47 in~\cite{Wil_crossed_2007} shows that the map
\[\psi\colon \left(\mathrm{Ind}_H^G(\R,\mu_H),\mathrm{Ind}_H^G(\mu_H)\right)
 \to (C(G/H,\mathcal{R}),\texttt{Lt}\otimes\mu_G)
\]
given by $\psi(f)(gH)=(\mu_G)_g(f(g))$ for all 
$f\in \mathrm{Ind}_H^G(\R,\mu_H)$ and all $g\in G$, is an equivariant
isomorphism. 

It follows from the above discussion that $(\M_{G\cdot\tau},\gamma)$
is cocycle conjugate to $(C(G/H,\mathcal{R}),\texttt{Lt}\otimes\mu_G)$, 
and is therefore $\mu_G$-McDuff, as desired.
\end{proof}

We are now ready to prove the main result of this section.

\begin{proof}[Proof of \autoref{thm:wtRp}.]
Note that $\alpha$ is strongly outer if and only if $\alpha\otimes\id_{\mathcal{Z}}$ is strongly outer. Thus,
it suffices to prove the theorem assuming that $\alpha$ is conjugate to $\alpha\otimes\id_{\mathcal{Z}}$.
We adopt the notation from \autoref{eg:EqW*bundle}. In particular, we denote by $(\M,\gamma)$ the equivariant $W^*$-bundle obtained
from $(A,\alpha)$. Let $\omega$ be a free ultrafilter over $\N$, and write $\pi\colon A_\omega\cap A' \to \M^\omega\cap \M'$
for the canonical quotient map whose kernel is contained in $J_A$; 
see Lemma~3.10 in \cite{BBSTWW_covering_2015}, which is an improvement of
Theorem~3.1 in~\cite{MatSat_decomposition_2014}.

(2) implies (1). The argument is mostly standard, but we include the argument for completeness.
Let $\ep>0$, let $S_1,\ldots, S_N$ be a paving family of subsets of $G$, and let
$(f_{\ell,g})_{\ell=1,\ldots,N, g\in S_\ell}$ be a family of Rokhlin contractions as in
\autoref{df:wtRp}. Set $p_{\ell,g}=\pi(f_{\ell,g})\in \M^\omega\cap \M'$. Then
$p_{\ell,g}$ is a projection, and the family $(p_{\ell,g})_{\ell=1,\ldots,N, g\in S_\ell}$
witnesses the fact that $\gamma\colon G\to\Aut(\M)$ has the $W^*$-Rokhlin property from
\autoref{df:W*Rp}. By considering the canonical surjections 
onto the fibers of $\M$, which are finite
direct sums of copies of $\R$, we deduce that the induced actions on the fibers of $\M$ have
the $W^*$-Rokhlin property, and are therefore outer. It follows that $\alpha$ is strongly outer.

(1) implies (2).
We abbreviate $\mu_G$ to $\mu$.
Note that $(\M_\tau,\gamma_\tau)$ is $\mu$-McDuff by \autoref{prop:FiberW*Rp}.
Thus, $\gamma$ is $\mu$-McDuff by \autoref{thm:MoRabsorbsMcDuff}, 
and in particular has the $W^*$-Rokhlin property from \autoref{df:W*Rp} by
\autoref{rem:muG}.
Let $S_1,\ldots, S_N$ be a paving family of subsets of $G$.
Find projections
$p_{\ell,g} \in \R^\omega\cap \R'$ for $\ell=1,\ldots,N$ and $g\in S_\ell$, satisfying:
\bi
\item $\mu^\omega_{gh^{-1}}(p_{\ell, h})=p_{\ell,g}$ for all $g,h\in S_\ell$ and all $\ell=1,\ldots,N$;
\item $p_{\ell,g}p_{k,h}=0$ for all $\ell,k=1,\ldots,N$, all $g\in S_\ell$ and all $h\in S_k$ with $(\ell,g)\neq (k,h)$;
\item $p_{\ell,g}\mu^\omega_h(p_{k,r})=\mu^\omega_h(p_{k,r})p_{\ell,g}$ for all $\ell,k=1,\ldots,N$, all $g\in S_\ell$ and all $r\in S_k$, and all $h\in G$;
\item $\sum\limits_{\ell=1}^N\sum\limits_{g\in S_\ell} p_{\ell, g}=1$.
\ei

Write $\tau_{\R^{\omega}}$ for the unique trace on the II$_1$-factor
$\R^\omega$.
Since $\pi\colon A_\omega\cap A'\to \M^\omega\cap \M'$ is surjective and has kernel contained in $J_A$, we can
lift the projections
\[1_\M\otimes p_{\ell,g}\in 1_\M\otimes (\R^\omega\cap \R')\subseteq (\M\overline{\otimes}\R)^\omega\cap (\M\overline{\otimes}\R)'\cong \M^\omega\cap \M'\]
to obtain positive contractions $e_{\ell,g}\in A_\omega\cap A'$, for $\ell=1,\ldots,N$ and $g\in S_\ell$, satisfying
\be
\item $\alpha^\omega_{gh^{-1}}(e_{\ell, h})-e_{\ell,g}\in J_A$ for all $\ell=1,\ldots,N$ and $g,h\in S_\ell$;
\item $e_{\ell,g}e_{k,h}\in J_A$ for all $\ell,k=1,\ldots,N$, all $g\in S_\ell$ and all $h\in S_k$ with $(\ell,g)\neq (k,h)$;
\item $e_{\ell,g}\alpha^\omega_h(e_{k,r})-\alpha^\omega_h(e_{k,r})e_{\ell,g}\in J_A$ for all $\ell,k=1,\ldots,N$, all $g\in S_\ell$ and all $r\in S_k$, and all $h\in G$;
\item $1-\sum\limits_{\ell=1}^N\sum\limits_{g\in S_\ell} e_{\ell, g}\in J_A$.
\item $\tau(e_{\ell,g})=\tau_{\R^\omega}(p_{\ell,g})=\tau_{\R^\omega}(p_{\ell,h})>0$ for all $\ell=1,\ldots,N$ and all $g,h\in S_\ell$.
\ee

Let $C$ be the (separable) $\alpha_\omega$-invariant subalgebra of $A_\omega$ generated by $A$ and the countable set
\[\{(\alpha_\omega)_g(e_{\ell,h})\colon g\in G, \ell=1,\ldots,N, \mbox{ and } h\in S_\ell\}.\]
Since $J_A$ is an equivariant $\sigma$-ideal in
$A_\omega$ by \autoref{prop:JAwSigmaId}, there exists a positive contraction $x\in (J_A\cap C')^{\alpha_\omega}$ with $xc=c$ for
all $c\in J_A\cap C$. We use this element $x$ to ``correct'' the positive contractions $e_{\ell, g}$, as follows. For $\ell=1,\ldots,N$
and $g\in S_\ell$, set $f_{\ell,g}=(1-x)e_{\ell,g}(1-x)$. We claim that these elements satisfy the conditions of \autoref{df:wtRp}.

First, $f_{\ell,g}$ is a positive contraction, and it commutes with the copy of $A$ in $A_\omega$ because so do $x$ and $e_{\ell,g}$.
Therefore $f_{\ell,g}$ belongs to $A_\omega\cap A'$. Observe that condition (d) is obviously satisfied.
To check condition (a) in \autoref{df:wtRp}, let $\ell=1,\ldots,N$ and $g,h\in S_\ell$. Observe that $(1-x)c=0$ for every $c\in C\cap J_A$, and use this
with $c=\alpha^\omega_{gh^{-1}}(e_{\ell, h})-e_{\ell,g}$
at the second step to get
\begin{align*}
(\alpha_\omega)_{gh^{-1}}(f_{\ell, h})-f_{\ell,g}&=(\alpha_\omega)_{gh^{-1}}((1-x)e_{\ell, h}(1-x))-(1-x)e_{\ell,g}(1-x)\\
&=(1-x)\left((\alpha_\omega)_{gh^{-1}}(e_{\ell, h})-e_{\ell,g}\right)(1-x)\\
&=0.
\end{align*}

To check condition (b), let $\ell,k=1,\ldots,N$, let $g\in S_\ell$ and let $h\in S_k$ with $(\ell,g)\neq (k,h)$.
We use that $x$ commutes with the elements $e_{\ell,g}$ at the second step to get
\[ f_{\ell,g}f_{k,h}=(1-x)e_{\ell,g}(1-x)^2e_{k,h}(1-x) = (1-x)e_{\ell,g}e_{k,h}(1-x)^3=0.\]

To check condition (c), let $\ell,k=1,\ldots,N$, let $g\in S_\ell$ and $r\in S_k$, and let $h\in G$. Then
\begin{align*}
\left[f_{\ell,g},(\alpha_\omega)_h(f_{k,r})\right] &= \left[(1-x)e_{\ell,g}(1-x),(\alpha_\omega)_h((1-x)e_{k,r}(1-x))\right]\\
&= (1-x)^2\left(\left[e_{\ell,g},(\alpha_\omega)_h(e_{k,r})\right] \right) (1-x)^2=0.
\end{align*}

To check condition (e), observe that
\begin{align*}\tau(f_{\ell,g})&=\tau((1-x)e_{\ell,g}(1-x))=\tau(e_{\ell,g})+\tau(e_{\ell, g}x)+\tau(xe_{\ell, g})+\tau(xe_{\ell,g}x)\\
&=\tau(e_{\ell,g})\end{align*}
for all $\ell=1,\ldots,N$, all $g\in S_\ell$ and all traces $\tau\in T(A)$.
Hence $\tau(f_{\ell,g})=\tau_{\R^\omega\cap \R'}(p_{\ell,g})>0$ for all $\ell=1,\ldots,N$ and all
$g\in S_\ell$, and this value is independent of $\tau$ and of $g\in S_\ell$.
Hence condition (e) is satisfied, and the result follows.
\end{proof}



\section{Finiteness of the Rokhlin dimension}

In this section, we complete the proof of Theorem~\ref{thmintro:Equiv} by computing the Rokhlin dimension of 
a strongly outer action. More precisely, we show the following:

\begin{thm}\label{thm:dimRokatmost2}
Let $G$ be a residually finite countable amenable group, let $A$ be a separable, simple, finite, \uca\ with
property (SI) such that $T(A)$ is a nonempty Bauer simplex
and that the induced action of $G$ on $\partial_eT(A)$ has finite orbits and Hausdorff orbit space.
Assume moreover that 
$\overline{A}^\tau$ is hyperfinite for all $\tau\in\partial_eT(A)$.
Let $\alpha\colon G\to\Aut(A)$ be a strongly outer action. Then
\[\dimRok(\alpha\otimes\id_{\mathcal{Z}})\leq 2.\]
\end{thm}

As mentioned before, the assumption that all weak closures of $A$ with respect to (extreme) traces be hyperfinite
is automatic when $A$ is nuclear, and it is also satisfied in other interesting cases such as when $A$ has finite
tracial rank. 

\begin{rem}
Adopt the assumptions of the theorem above, and suppose moreover that $\partial_eT(A)$ has finite covering dimension.
It will follow from \autoref{thm:EqJiangSuAbs} that $\alpha$ is cocycle conjugate
to $\alpha\otimes\id_{\mathcal{Z}}$, and hence that $\alpha$ itself has Rokhlin dimension at most 2.
\end{rem}

Our result is inspired by analogous ones by Liao in \cite{Lia_rokhlin_2016, Lia_rokhlin_2017}, where similar facts are proved for $\Z^m$-actions.
Nonetheless, our approach differs significantly from Liao's, in that we obtain the Rokhlin towers by centrally embedding suitable actions
on dimension drop algebras. The advantage of our approach is that it does not require any restrictions on the group: in particular, we are able to treat
groups with torsion, as well as groups that are not finitely generated. (For example, the application of property (SI) in~\cite{Lia_rokhlin_2016}, particularly Theorem~6.4 there, makes essential use of the fact that $\Z$ has no torsion, and the methods used there seem to break down 
already for finite groups.)

The notion of Rokhlin dimension has been defined in~\cite{HirWinZac_rokhlin_2015} for finite group and integer actions on unital \ca s,
and extended to the non-unital case in \cite{HirPhi_rokhlin_2015}, and to actions of amenable residually
finite groups in~\cite{SzaWuZac_rokhlin_2014}. (There has also been some work on Rokhlin dimension for non-discrete groups;
see for example, \cite{Gar_rokhlin_2017}, \cite{Gar_regularity_2017}, \cite{GarHirSan_rokhlin_2017} and~\cite{HirSzaWinWu_rokhlin_2017}.) We recall the definition below. For a subgroup $H \subseteq G$ and for $g\in G$, we denote by $\overline{g}$ the left coset $gH$.

\begin{df}[See Definition~A, Remark~3.2 and Lemma~5.7 in~\cite{SzaWuZac_rokhlin_2014}]\label{df:Rdim}
Let $G$ be a countable residually finite group, let $A$ be a separable \uca, and let $\alpha\colon G\to\Aut(A)$ be an
action. Given $d\in\Z$ with $d\geq 0$, we say that $\alpha$ has \emph{Rokhlin dimension at most $d$}, and write $\dimRok(\alpha)\leq d$,
if for any normal subgroup $H\leq G$ of finite index, 
there exist positive contractions
$f^{(j)}_{\overline{g}} \in A_\omega\cap A'$, for $j=0,\ldots,d$ and for $\overline{g}\in G/H$, satisfying:
\be
\item[(a)] $(\alpha_\omega)_{g}(f^{(j)}_{\overline{h}})=f^{(j)}_{\overline{gh}}$ for all $j=0,\ldots,d$ and for all $g \in G$ and $\overline{h} \in G/H$;
\item[(b)] $f^{(j)}_{\overline{g}}f^{(j)}_{\overline{h}}=0$ for all $j=0,\ldots,d$ and for all $\overline{g},\overline{h}\in G/H$ with $\overline{g}\neq \overline{h}$;
\item[(c)] $\sum\limits_{j=0}^d\sum\limits_{\overline{g}\in G/H} f^{(j)}_{\overline{g}}=1$.
\ee

The \emph{Rokhlin dimension} of $\alpha$, denoted $\dimRok(\alpha)$, is the smallest integer $d$ such that $\dimRok(\alpha)\leq d$.
\end{df}

There is a strengthening of the Rokhlin dimension, called \emph{Rokhlin dimension with commuting towers}, where the
elements $f_{\overline{g}}^{(j)}$ are assumed to moreover commute with each other. We will not deal with this notion here.

We record here an equivalent definition of Rokhlin dimension, which uses seemingly weaker conditions. That both definitions
are equivalent is an immediate consequence of Kirchberg's $\ep$-test.

\begin{rem} \label{rmk:EquivalenceDimRok}
In the context of \autoref{df:Rdim}, we have $\dimRok(\alpha)\leq d$ if and only if for any normal subgroup $H \leq G$ of finite index, for any finite subset $G_0 \subseteq G$, for every $\ep>0$ and for every finite
subset $F\subseteq A$, there are positive contractions
$f^{(j)}_{\overline{g}} \in A_\omega$, for $j=0,\ldots,d$ and for $\overline{g}\in G/H$, satisfying:
\be
\item[(a)] $\left\|(\alpha_\omega)_{g}(f^{(j)}_{\overline{h}})-f^{(j)}_{\overline{gh}}\right\|<\ep$ for all $j=0,\ldots,d$, for all $g\in G_0$ and for all $\overline{h}\in H$;
\item[(b)] $\left\|f^{(j)}_{\overline{g}}f^{(j)}_{\overline{h}}\right\|<\ep$ for all $j=0,\ldots,d$ and for all $\overline{g},\overline{h}\in G/H$ with $\overline{g}\neq \overline{h}$;
\item[(c)] $\left\|1-\sum\limits_{j=0}^d\sum\limits_{\overline{g}\in G/H} f^{(j)}_{\overline{g}}\right\|<\ep$;
\item[(d)] $\left\|af^{(j)}_{\overline{g}}-f^{(j)}_{\overline{g}}a\right\|<\ep$ for all $a\in F$, for all $\overline{g}\in G/H$ and for all $j=0,\ldots,d$.
\ee
\end{rem}

We begin by computing the Rokhlin dimension of a natural product-type action.
When $G$ is a finite group, it suffices to take $H=\{1\}$ in 
\autoref{df:Rdim}.

\begin{prop}\label{prop:DimRok1}
Let $G$ be a finite group, and set $D= \bigotimes\limits_{n\in\N} \B(\ell^2(G)^{\otimes n}\oplus \C)$. Denote by 
$\lambda\colon G\to \U(\ell^2(G))$ the left regular representation, and define an action
$\alpha\colon G\to\Aut(D)$ by $\alpha_g=\bigotimes\limits_{n\in\N}\Ad(\lambda^{\otimes n}_g\oplus 1)$
for all $g\in G$. Then $\dimRok(\alpha)=1$.
\end{prop}
\begin{proof}
Given $m\in\N$, set $D_m= \B(\ell^2(G)^{\otimes m}\oplus \C)$ and let $\alpha^{(m)}\colon G\to\Aut(D_m)$
be the action given by
$\alpha^{(m)}_g=\Ad(\lambda^{\otimes m}_g\oplus 1)$ for all $g\in G$.

\textbf{Claim:} \emph{Let $\ep>0$ and $n_0\in\N$. Then there exist $m\in\N$ with $m\geq n_0$ and positive contractions $f_g^{(j)}\in D_m$,
for $g\in G$ and $j=0,1$, satisfying
\be\item $\alpha^{(m)}_g(f_h^{(j)})=f_{gh}^{(j)}$ for all $g,h\in G$ and for all $j=0,1$;
\item $f_g^{(j)}f_h^{(j)}=0$ for all $g,h\in G$ with $g\neq h$ and for all $j=0,1$;
\item $\left\|1-\sum\limits_{j=0}^1\sum\limits_{g\in G}f^{(j)}_{g}\right\|<\ep$.
\ee}

The claim shows that there exist Rokhlin towers in $D_\infty$ satisfying conditions (a), (b) and (c) in \autoref{rmk:EquivalenceDimRok}
for $d=1$. We now explain how to find new towers satisfying these conditions in addition to condition (d), and then prove the claim.

Let $F\subseteq D$ be a finite set. For the $\ep>0$ given above, find $n_0\in\N$ such that for every $n\geq n_0$ there exists
a unital equivariant embedding $\varphi_n\colon (D_n,\alpha^{(n)})\hookrightarrow (D_\infty,\alpha_\infty)$ satisfying
\[\|\varphi_n(x)a-a\varphi_n(a)\|\leq \frac{\ep}{2} \|x\|\]
for all $x\in D_n$ and all $a\in F$. Use the claim to find $m\in\N$ with $m\geq n_0$ and positive contractions $f_g^{(j)}\in D_m$,
for $g\in G$ and $j=0,1$, satisfying conditions (1), (2) and (3) above for the tolerance $\ep/2$.
One checks that the positive contractions $\varphi_m(f_g^{(j)})\in D$,
for $g\in G$ and $j=0,1$, satisfy conditions (a) through (d) in \autoref{rmk:EquivalenceDimRok}, as desired.

We proceed to prove the claim.
By Fell's absorption principle, if $\pi\colon G\to \U(\Hi)$ is any finite dimensional representation of $G$ on a Hilbert space
$\Hi$, then $\lambda\otimes\pi$ is unitarily equivalent to a direct sum of $\dim(\Hi)$ copies of $\lambda$.
By writing $\lambda^{\otimes m}$ as a direct sum of $|G|^{m-1}$ copies of $\lambda$, it follows that
$\lambda^{\otimes m}\oplus 1$ is unitarily equivalent to $\bigoplus_{k=1}^{|G|^{m-1}} \lambda\oplus 1$.
We fix such an identification for the remainder of the proof.
For the $\ep>0$ given, choose $m\in\N$ such that $|G|^{m-2}>2/\ep$ and also $m\geq n_0$.

Using that $\lambda$ contains a copy of the trivial representation, 
let $\widetilde{\lambda}\colon G\to \mathcal{U}(V)$ 
be any unitary representation satisfying 
$\lambda \cong 1 \oplus \widetilde{\lambda}$. Then 
$\lambda^{\otimes m}\oplus 1$ is unitarily conjugate to the diagonal
representation
$\diag(1,\widetilde{\lambda}_g,1,\widetilde{\lambda}_g,1 \ldots,\widetilde{\lambda}_g,1)$, where the trivial representation appears $|G|^{m-1}+1$ times and $\widetilde{\lambda}$ appears $|G|^{m-1}$ times.

Let $p_g\in \B(\C\oplus V)$, for $g\in G$,
be projections satisfying 
$\sum_{g \in G} p_g = 1$ and $\Ad (1 \oplus \widetilde{\lambda}_g)(p_h) = p_{gh}$ 
for all $g,h \in G$. Similarly, let $q_g\in \B(V\oplus \C)$ be 
projections satisfying $\sum_{g \in G} q_g = 1$ and $\Ad (\widetilde{\lambda}_g \oplus 1)(q_h) = q_{gh}$ for all $g,h \in G$. 
(Since $1\oplus \widetilde{\lambda}\cong \widetilde{\lambda}\oplus 1\cong\lambda$, one can construct the projections $p_g$ and $q_g$ by taking 
suitable unitary conjugates of the projection $e_g\in \B(\ell^2(G))$
onto the span of $\delta_g$.)

Let $a_0\colon [0,|G|^{m-1}]\to [0,1]$ be given by $a_0(x)=\frac{x}{|G|^{m-1}+1}$ for all $x\in [0,|G|^{m-1}]$, and set $a_1=1-a_0$.
For $g\in G$, set
\[f_g^{(0)}=\diag(0,a_0(1)q_g,\ldots,a_0(|G|^{m-1})q_g)\]
and
\[f_g^{(1)}=\diag(a_1(1)p_g,\ldots,a_1(|G|^{m-1})p_g),0) .\]
These are positive contractions satisfying conditions (1), (2) and (3) above, thus showing
that $\dimRok(\alpha)\leq 1$.

Finally, $\dimRok(\alpha)=0$ is impossible because the unit of $D$ is not divisible by $|G|$ in $K_0(D)$.
\end{proof}

Using the computation above, we will show that certain canonical actions on dimension drop algebras admit
Rokhlin towers that satisfy all the conditions in \autoref{df:Rdim} except for centrality; see \autoref{prop:GactIkdim2}.
We retain the notation from \autoref{nota:DimDrop}.

\begin{df}\label{df:ActonIk}
Let $G$ be a finite group, and let $k\in\N$. We denote by
$I^{(k)}_G$ the dimension drop algebra
\[ I_{G}^{(k)}=\left\{f\in C\left([0,1], \B(\ell^2(G)^{\otimes k}) \otimes \B(\ell^2(G)^{\otimes k}\oplus \C)\right)\colon
\begin{tabular}{@{}l@{}}
    $f(0)\in \B(\ell^2(G)^{\otimes k})\otimes 1,$\\
    $f(1)\in 1\otimes \B(\ell^2(G)^{\otimes k}\oplus \C)$
   \end{tabular}
\right\}\]
and we denote by $\mu^{(k)}_G\colon G\to \Aut(I^{(k)}_{G})$ the action $\gamma_{\lambda^{\otimes k},\lambda^{\otimes k}\oplus 1}$.
\end{df}

Next, we give a recipe for constructing unital equivariant homomorphisms from $(I^{(k)}_{G},\mu^{(k)}_G)$.
We do so in a generality greater than necessary, because the proof is not more complicated and in fact the
higher level of abstraction makes the argument conceptually clearer.

\begin{rem}
\label{rem:generators-Mn}
Let $n\in\N$. Recall that $M_n$ is the universal $C^*$-algebra generated by
elements $\{e_{1,k}\}_{k=1}^n$ satisfying $e_{1,k}e_{1,j}^* = \delta_{k,j}e_{1,1}$. Note that
$\Hi_n=\textrm{span} \{e_{1,k}\}_{k=1}^n$ is an $n$-dimensional Hilbert space with inner
product given by $\langle a,b \rangle e_{1,1} = ab^*$ for all $a,b\in\Hi_n$. Moreover, any orthonormal basis of $\Hi_n$
which contains $e_{1,1}$ is a set of generators satisfying the same relations.
In particular, any unitary operator on $\Hi_n$ which
fixes $e_{1,1}$ extends to an automorphism of $M_n$ which fixes the projection $e_{1,1}$.
\end{rem}

The following is an equivariant version of a well-known characterization of the dimension
drop algebra; see, for example, \cite{KirRor_central_2014}. We note, however, that our proof
is new even in the non-equivariant setting.

\begin{thm}\label{thm:univprop}
Let $G$ be a finite group,
let $B$ be a unital \ca, and let $\beta\colon G\to\Aut(B)$ be an action.
Let $n \in \N$, let $v \colon G \to U(n)$ be a unitary representation containing the trivial representation, and
let $e_{1,1}$ denote any rank-one $G$-invariant projection in $M_n$.
Suppose that there exist a completely positive contractive equivariant order zero map
\[\xi\colon \left(M_n,\Ad(v) \right) \to (B,\beta)\]
and a contraction $s\in B^\beta$ satisfying $\xi(e_{11})s=s$ and $\xi(1)+s^*s=1$.

Let $\gamma$ be the restriction to $I_{n,n+1}$ of the action of $G$ on $C([0,1]) \otimes M_n \otimes M_{n+1}$ given by
\[
\id_{C([0,1])} \otimes \Ad(v) \otimes \Ad(v \oplus 1_{\C}).
\]

Then $\xi$ can be extended to a unital, equivariant homomorphism
\[
\pi\colon (I_{n,n+1},\gamma)\to (B,\beta).
\]
\end{thm}

\begin{proof}
Denote by $D$ the universal \ca~generated by a set $\{s,f_{j,k}\colon j,k=1,\ldots,n\}$ of contractions satisfying:
\begin{enumerate}
\item $f_{j,k}^* = f_{k,j}$ for all $j,k=1,\ldots,n$.
\item $f_{j,k}f_{l,m} = \delta_{k,l}f_{j,j}f_{j,m}$ for all $j,k,l,m=1,\ldots,n$.
\item $f_{1,1}s = s$.
\item $\sum_{j=1}^n f_{j,j} + s^*s = 1$.
\end{enumerate}

We claim that $D$ is isomorphic to the dimension drop algebra $I_{n,n+1}$.
We begin by noting that $I_{n,n+1}\subseteq C([0,1],M_{n(n+1)})$ is isomorphic to the subalgebra of
$C([0,1],M_{n(n+1)})$ of the continuous functions $f$ such that $f(0)$ can be written as an $n \times n$
matrix in block form 
\[
\left ( \begin{matrix}
	z_{11} 1_{M_{n+1}} & z_{12}  1_{M_{n+1}} & \cdots & z_{1n}  1_{M_{n+1}} \\
	z_{21} 1_{M_{n+1}} & z_{22}  1_{M_{n+1}} & \cdots & z_{2n}  1_{M_{n+1}} \\	
	\vdots				&	\vdots			& \ddots & \vdots \\
	z_{n1} 1_{M_{n+1}} & z_{n2}  1_{M_{n+1}} & \cdots & z_{nn}  1_{M_{n+1}} \\
\end{matrix} \right )
\]
(where $z_{ij}$ are scalars for $i,j \in \{1,2,\ldots,n\}$)
 and such that $f(1)$ is in
the subalgebra of $M_{n(n+1)}$ isomorphic to $M_{n+1}$ which
is generated by the elements $F_{1,1}, \ldots, F_{1,n}$ and $F_{1,n+1}$ given as follows
(each vertical line represents $n+1$ entries, and the horizontal line appears after $n$ rows):
\[
\begin{matrix}
 F_{1,1} =	\left(\begin{array}{@{}ccccc|ccccc|c|ccccc@{}}
    1 & 0 & \cdots & 0 & 0  & 0 & 0 & \cdots & 0 & 0 			    & \cdots &  0 & 0 & \cdots & 0 & 0 \\
    0 & 1 & \cdots & 0 & 0 			  & 0 & 0 & \cdots & 0 & 0 	& \cdots &  0 & 0 & \cdots & 0 & 0 \\
    \vdots & \vdots &  \vdots &  \vdots & \vdots			  &  \vdots &  \vdots &  \vdots &  \vdots &  \vdots  &  \vdots &   \vdots &  \vdots & \vdots &  \vdots & \vdots \\
    0 & 0 & \cdots & 1 & 0			  & 0 & 0 & \cdots & 0 & 0  & \cdots &  0 & 0 & \cdots & 0 & 0 \\\hline
    0 & 0 & \cdots & 0 &0			  & 0 & 0 & \cdots & 0 & 0  & \cdots &  0 & 0 & \cdots & 0 & 0 \\
     \vdots & \vdots &  \vdots &  \vdots & \vdots			  &  \vdots &  \vdots &  \vdots &  \vdots &  \vdots  &  \vdots &   \vdots &  \vdots & \vdots &  \vdots & \vdots \\
       0 & 0 & \cdots & 0 &0			  & 0 & 0 & \cdots & 0 & 0  & \cdots &  0 & 0 & \cdots & 0 & 0 \\
  \end{array}\right)\end{matrix} \]

\[
    \vdots
\]
\[  \begin{matrix} F_{1,n}= 	\left(\begin{array}{@{}ccccc|ccccc|c|ccccc@{}}
        0 & 0 & \cdots & 0 & 0  & 0 & 0 & \cdots & 0 & 0 			    & \cdots &  1 & 0 & \cdots & 0 & 0 \\
        0 & 0 & \cdots & 0 & 0 			  & 0 & 0 & \cdots & 0 & 0 	& \cdots &  0 & 1 & \cdots & 0 & 0 \\
        \vdots & \vdots &  \vdots &  \vdots & \vdots			  &  \vdots &  \vdots &  \vdots &  \vdots &  \vdots  &  \vdots &   \vdots &  \vdots & \vdots &  \vdots & \vdots \\
        0 & 0 & \cdots & 0 & 0			  & 0 & 0 & \cdots & 0 & 0  & \cdots &  0 & 0 & \cdots & 1 & 0\\\hline
        0 & 0 & \cdots & 0 &0			  & 0 & 0 & \cdots & 0 & 0  & \cdots &  0 & 0 & \cdots & 0 & 0 \\
         \vdots & \vdots &  \vdots &  \vdots & \vdots			  &  \vdots &  \vdots &  \vdots &  \vdots &  \vdots  &  \vdots &   \vdots &  \vdots & \vdots &  \vdots & \vdots \\
           0 & 0 & \cdots & 0 &0			  & 0 & 0 & \cdots & 0 & 0  & \cdots &  0 & 0 & \cdots & 0 & 0 \\
      \end{array}\right) \end{matrix}
\]
\[
     \begin{matrix}F_{1,n+1} = 	\left(\begin{array}{@{}ccccc|ccccc|c|ccccc@{}}
          0 & 0 & \cdots & 0 & 1  & 0 & 0 & \cdots & 0 & 0 			    & \cdots &  0 & 0 & \cdots & 0 & 0 \\
          0 & 0 & \cdots & 0 & 0 			  & 0 & 0 & \cdots & 0 & 1 	& \cdots &  0 & 0 & \cdots & 0 & 0 \\
          \vdots & \vdots &  \vdots &  \vdots & \vdots			  &  \vdots &  \vdots &  \vdots &  \vdots &  \vdots  &  \vdots &   \vdots &  \vdots & \vdots &  \vdots & \vdots \\
          0 & 0 & \cdots & 0 & 0			  & 0 & 0 & \cdots & 0 & 0  & \cdots &  0 & 0 & \cdots & 0 & 1 \\\hline
          0 & 0 & \cdots & 0 &0			  & 0 & 0 & \cdots & 0 & 0  & \cdots &  0 & 0 & \cdots & 0 & 0 \\
           \vdots & \vdots &  \vdots &  \vdots & \vdots			  &  \vdots &  \vdots &  \vdots &  \vdots &  \vdots  &  \vdots &   \vdots &  \vdots & \vdots &  \vdots & \vdots \\
             0 & 0 & \cdots & 0 &0			  & 0 & 0 & \cdots & 0 & 0  & \cdots &  0 & 0 & \cdots & 0 & 0 \\
        \end{array}\right)
\end{matrix}
\]
(One checks that the elements $F_{1,j}$, for $j=1,2,\ldots,n+1$, are contractions, satisfy $F_{1,j}F_{i,i} = 0$ for all $i,j \in \{1,2,\ldots,n\}$ with $j\neq 1$,
$F_{1,j}F_{1,i}^* = \delta_{i,j}F_{1,1}$ and $\sum_{j=1}^{n+1} F_{1,j}^*F_{1,j} = 1$. Therefore they generate a unital copy of $M_{n+1}$.)

Denote by $\rho \colon [0,1] \to [0,1]$ the identity function.
For $j,k=1,\ldots,n$, let $\tilde{f}_{j,k} \in I_{n,n+1}$ be the matrix-valued function which,
written in $n \times n$ block form, has as its $(j,k)$-th block the diagonal matrix valued function
$\diag(\underbrace{1,1,\ldots,1}_{n \textrm{ times}},1-\rho)$, and $0$ elsewhere. Let $\tilde{s}$ be the matrix-valued function
\[
  \tilde{s} = \left(\begin{array}{@{}ccccc|ccccc|c|ccccc@{}}
    0 & 0 & \cdots & 0 & \sqrt{\rho}  & 0 & 0 & \cdots & 0 & 0 			    & \cdots &  0 & 0 & \cdots & 0 & 0 \\
    0 & 0 & \cdots & 0 & 0 			  & 0 & 0 & \cdots & 0 & \sqrt{\rho} 	& \cdots &  0 & 0 & \cdots & 0 & 0 \\
    \vdots & \vdots &  \vdots &  \vdots & \vdots			  &  \vdots &  \vdots &  \vdots &  \vdots &  \vdots  &  \vdots &   \vdots &  \vdots & \vdots &  \vdots & \vdots \\
    0 & 0 & \cdots & 0 & 0			  & 0 & 0 & \cdots & 0 & 0  & \cdots &  0 & 0 & \cdots & 0 & \sqrt{\rho} \\\hline
    0 & 0 & \cdots & 0 &0			  & 0 & 0 & \cdots & 0 & 0  & \cdots &  0 & 0 & \cdots & 0 & 0 \\
     \vdots & \vdots &  \vdots &  \vdots & \vdots			  &  \vdots &  \vdots &  \vdots &  \vdots &  \vdots  &  \vdots &   \vdots &  \vdots & \vdots &  \vdots & \vdots \\
       0 & 0 & \cdots & 0 &0			  & 0 & 0 & \cdots & 0 & 0  & \cdots &  0 & 0 & \cdots & 0 & 0 \\
  \end{array}\right),
\]
where each vertical dividing line represents $n+1$ entries.

One checks that the functions $\tilde{f}_{j,k}$ for $j,k=1,\ldots,n$ and $\tilde{s}$ satisfy the relations defining $D$ and generate $I_{n,n+1}$.
Fix a surjection $\kappa \colon D \to I_{n,n+1}$ satisfying $\kappa(f_{j,k})=\tilde{f}_{j,k}$ for all $j,k=1,\ldots,n$ and $\kappa(s)=\tilde{s}$.
It remains to show that $\kappa$ is injective.
\newline

\textbf{Claim 1:} \emph{the following identities hold:
\be
\item  $f_{j,j}f_{j,k} = f_{j,k}f_{k,k}$ for all $j,k =1,\ldots,n$.
\item For any $j =1,\ldots,n$, we have $s f_{j,1} s = 0$.
\ee}

The first of these is readily checked. The following computation establishes the second identity,
where we use that $f_{1,1}s = s$ repeatedly:
\begin{align*}
(s f_{j,1} s)^*(s f_{j,1} s) &=
s^*f_{1,j}s^*sf_{j,1}s  \\
&=   s^*f_{1,j}\left (1 - \sum_{i=1}^n f_{i,i} \right )f_{j,1}s \\
&= s^* f_{1,j}f_{j,1}s -  s^* f_{1,j}f_{j,j}f_{j,1}s \\
&= s^*f_{1,1}^2s - s^*f_{1,1}^3s = 0.
\end{align*}
This proves the claim.

Set
\[
b = s^*s + \sum_{j=1}^n f_{j,1}ss^*f_{1,j} = s^*s + ss^* + \sum_{j=2}^n f_{j,1}ss^*f_{1,j}.
\]
Note that $b$ is a positive contraction (as can be seen using that all of the summands in its definition are pairwise orthogonal).
One checks that $\kappa(b) = \rho \cdot 1$, and therefore $\spec(b) = [0,1]$.
\newline

\textbf{Claim 2:} \emph{$b$ belongs to the center of $D$.}
Let $j,k=1,\ldots,n$.
Then
\[
f_{j,k}b = f_{j,k}s^*s + f_{j,j}f_{j,1}ss^*f_{1,k} \ \mbox{ and } \
bf_{j,k} = s^*sf_{j,k} +f_{j,1}ss^*f_{1,k} f_{k,k}.
\]
We show term by term that both expressions agree, which will imply the claim.
For the first terms in both right hand sides, we have
\begin{align*}
f_{j,k}s^*s = f_{j,k} \left (1 -\sum_{i=1}^n f_{i,i} \right) &= f_{j,k} - f_{j,k}f_{k,k} =  f_{j,k} - f_{j,j}f_{j,k} \\
&= \left (1 -\sum_{i=1}^n f_{i,i} \right ) f_{j,k} =  s^*sf_{j,k}.
\end{align*}
For the second terms in both right hand sides, we have
\begin{align*}
f_{j,j}f_{j,1}ss^*f_{1,k} &= f_{j,1} (f_{k,k}s) s^*f_{1,k} = f_{j,1}ss^*f_{1,k} \\
&= f_{j,1}s(f_{1,1}s)^*f_{1,k} \\
&= f_{j,1}ss^*f_{1,k} f_{k,k}.
\end{align*}

Likewise, using the fact that $s f_{i,1} s = 0$ for all $i=1,\ldots,n$, we get
\[
sb = ss^*s + s\sum_{i=1}^n f_{i,1}ss^*f_{1,i}  = ss^*s
\]
and
\[
bs = s^*s^2 + \sum_{i=1}^n f_{i,1}ss^*f_{1,i}s = s^*sf_{1,1}s + f_{1,1}ss^*f_{1,1}s = 0 + ss^*s = sb.
\]
This completes the proof of the claim.

It follows that $b$ endows $D$ with a $C([0,1])$-algebra structure, and the quotient map $\kappa$ is a $C([0,1])$-algebra homomorphism.
For $t \in [0,1]$, denote by $D(t)$ and by $I_{n,n+1}(t)$ the corresponding fibers, and by
$\kappa_t \colon D(t) \to I_{n,n+1}(t)$ the corresponding surjection. It suffices now to show that $\kappa_t$ is injective for all $t$. Indeed, if
$a\in \ker(\kappa)$ is nonzero, then there exists $t\in [0,1]$ 
such that $a_t$ is nonzero, and clearly $a_t\in \ker(\kappa_t)$. 
We thus fix from now on $t\in [0,1]$. 

The fiber $I_{n,n+1}(t)$ is isomorphic to $M_n$ if $t = 0$, to $M_{n+1}$ if $t = 1$, and to $M_n \otimes M_{n+1}$ otherwise.
It therefore suffices to show that the fibers $D(t)$ are also isomorphic to those corresponding matrix algebras.
Denote by $f_{j,k}(t)$, $s(t)$ and $b(t)$ the images of the elements $f_{j,k}$, $s$ and $b$ in the fiber $D(t)$.

\underline{Case I - $t=0$:} Here $b(0) = 0$. In particular, as $b(0) \geq s^*(0) s(0)$, it follows that $s(0) = 0$.
Therefore, $\{f_{j,k}(0)\}_{j,k=1}^n$ generates $D(0)$, and these are precisely the matrix units of $M_n$. Thus $D(0) \cong M_n$.

\underline{Case II - $t = 1$:} For $k,l=1,\ldots,n$ we set $E_{j,k} = f_{j,1}(1)s(1)s^*(1)f_{1,k}(1)$. We also
set
\[E_{1,n+1} = s(1), \ E_{n+1,1} = s^*(1), \ \mbox{ and } \ E_{j,k} = E_{j,1}E_{1,k}\]
for all remaining cases with $k,l=1,\ldots,n+1$. Using that $b(1)=1$,
one now checks that $\{E_{j,k}\}_{k,l =1}^{n+1}$ are matrix units which generate $D(1)$, which is therefore isomorphic to $M_{n+1}$.

\underline{Case III - $t \in (0,1)$:} Here $b(t) = t1$. In order to lighten the notation, we write the generators
as in \autoref{rem:generators-Mn}. For $j,k,l=1,\ldots,n$, set
\[c_{j,k} = \frac{1}{t-t^2} s(t)f_{1,j}(t)s^*(t)f_{1,k}(t)\ \mbox{ and } \ d_{l} = \frac{1}{\sqrt{t}(1-t)} s(t)f_{1,l}(t).\]
We verify that these elements satisfy the conditions from \autoref{rem:generators-Mn}. To that end, note first that $s^2=0$, and that
\begin{align*}
t s(t)f_{1,1}(t)^2 s^*(t)
&= s(t)f_{1,1}(t) \cdot t1 \cdot f_{1,1}(t) s^*(t) \\
&=	s(t)f_{1,1}(t) \left ( s^*(t)s(t) + \sum_{j=1}^n f_{j,1}(t)s(t)s^*(t)f_{1,j}(t) \right ) f_{1,1}(t)s^*(t) \\
&= 	s(t)f_{1,1}(t)s^*(t)s(t)f_{1,1}(t)s^*(t) \\
& \ \ \ \ + s(t)f_{1,1}(t)^2s(t)s^*(t)f_{1,1}(t)^2s^*(t) \\
&= s(t)f_{1,1}(t)s^*(t)s(t)f_{1,1}(t)s^*(t) +s(t)^2s^*(t)f_{1,1}(t)^2s^*(t) \\
&= s(t)f_{1,1}(t)s^*(t)s(t)f_{1,1}(t)s^*(t) + 0 \\
&= s(t)f_{1,1}(t) \left ( 1-\sum_{j=1}^nf_{jj}(t)  \right ) f_{1,1}(t)s^*(t)  \\
&= s(t)\left (f_{1,1}(t)^2 - f_{1,1}(t)^3 \right ) s^*(t).
\end{align*}
Likewise, we see that
$$
t s(t)f_{1,1}(t)s^*(t) = s(t) \left (f_{1,1}(t) - f_{1,1}(t)^2 \right ) s^*(t).
$$
Thus
$$
s(t)f_{1,1}^2(t)s^*(t) = (1-t)s(t)f_{1,1}(t)s^*(t)
$$
and therefore
\begin{align*}
s(t)\left (f_{1,1}(t)^2 - f_{1,1}(t)^3 \right ) s^*(t) &= t(1-t)s(t)f_{1,1}(t)s^*(t) \\
&= t(1-t)s(t)f_{1,1}(t)s^*(t)f_{1,1}(t).
\end{align*}
We can now verify that the elements $c_{j,k}$ and $d_{l}$ from above satisfy the conditions from \autoref{rem:generators-Mn}.
For any quadruple $j,k,l,m=1,\ldots,n$ we have:
\begin{align*}
c_{j,k}c_{l,m}^* &= \left ( \frac{1}{t-t^2} \right )^2 s(t)f_{1,j}(t)s^*(t)f_{1,k}(t) \cdot f_{m,1}(t)s(t)f_{l,1}(t) s^*(t) \\
&= \delta_{k,m} \left ( \frac{1}{t-t^2} \right )^2 s(t)f_{1,j}(t)s^*(t) \left (f_{1,1}(t)^2 s(t) \right )f_{l,1}(t) s^*(t) \\
&= \delta_{k,m} \left ( \frac{1}{t-t^2} \right )^2 s(t)f_{1,j}(t)s^*(t) s(t)f_{l,1}(t) s^*(t) \\
&= \delta_{k,m} \left ( \frac{1}{t-t^2} \right )^2 s(t)f_{1,j}(t)\left ( 1-\sum_{i=1}^nf_{i,i}(t)  \right )f_{l,1}(t) s^*(t) \\
&= \delta_{k,m} \delta_{j,l}  \left ( \frac{1}{t-t^2} \right )^2 s(t)\left ( f_{1,1}(t)^2 -  f_{1,1}(t)^3 \right ) s^*(t) \\
&=  \delta_{k,m} \delta_{j,l} \left ( \frac{1}{t-t^2} \right )^2 \cdot  t(1-t)s(t)f_{1,1}(t)s^*(t)f_{1,1}(t) =  \delta_{k,m} \delta_{j,l} c_{1,1}
\end{align*}
For $j,k=1,\ldots,n$, we have:
\begin{align*}
d_jd_k^* &= \left (\frac{1}{\sqrt{t}(1-t)} \right )^2 s(t)f_{1,j}(t) \cdot  f_{k1}(t) s(t)^* \\
&= \delta_{j,k} \frac{1}{t(1-t)^2} s(t)f_{1,1}(t)^2  s(t)^*  \\
&= \delta_{j,k}  \frac{1}{t(1-t)^2} \cdot (1-t)s(t)f_{1,1}(t)s^*(t) = \delta_{j,k} c_{1,1}
\end{align*}
Similarly, for $j,k,l=1,\ldots,n$ we have:
\begin{align*}
c_{j,k}d_l^* &=   \frac{1}{t^{3/2}(1-t)^2}  s(t)f_{1,j}(t)s^*(t)f_{1,k}(t) \cdot  f_{l,1}(t) s(t)^* \\
&= \delta_{kl}  \frac{1}{t^{3/2}(1-t)^2}  s(t)f_{1,j}(t)s^*(t)f_{1,1}(t)^2 s^*(t) \\
&=  \delta_{kl}  \frac{1}{t^{3/2}(1-t)^2}  s(t)f_{1,j}(t)(s^*(t))^2 = 0
\end{align*}
and likewise $d_lc_{j,k}^*=0$. By \autoref{rem:generators-Mn}, it follows that
$\{c_{j,k},d_l\colon j,k,l=1,\ldots,n\}$ generates $M_{n(n+1)}$.

\textbf{Claim 3:} \emph{
The set $\{c_{j,k},d_l\colon j,k,l=1,\ldots,n\}$ also generates $D(t)$.} It suffices to show that this family generates $s(t)$ and $\{f_{1,j}(t) \colon j=1,2,\ldots,n\}$. Indeed, 
\begin{align*}
\sqrt{t} \sum_{j=1}^n c_{j,1}^*d_j &= \frac{1}{t(1-t)^2} \sum_{j=1}^n sf_{j,1}(t)s(t)^*s(t)f_{1,j}(t) \\
&= \frac{1}{t(1-t)^2} \sum_{j=1}^n sf_{j,1}(t) \left ( 1- \sum_{m=1}^n f_{m,m}(t) \right ) f_{1,j}(t) \\
&= \frac{1}{t(1-t)^2} \sum_{j=1}^n s \left ( f_{j,j}(t)^2 -  f_{j,j}(t)^3 \right )  \\
&= \frac{1}{t(1-t)^2} \sum_{j=1}^n s \left ( (1-s(t)^*s(t))^2 - (1-s(t)^*s(t))^3 \right ) \\
&= \frac{1}{t(1-t)^2} \left ( (1-t)^2 - (1-t)^3 \right ) s(t) \\
&= s(t)
\end{align*}
Next, we show that
\[
f_{1,j}(t) = \sum_{k=1}^n c_{k,1}^*c_{k,j} + (1-t)d_1^*d_j
\]
To see this, recall that we saw that $sb = ss^*s = bs$, so $s(t)s(t)^*s(t) = ts(t)$. So,
\begin{align*}
\sum_{k=1}^n c_{k,1}^*c_{k,j} &= 
\frac{1}{t^2(1-t)^2} \sum_{k=1}^n s(t) f_{k,1}(t) s(t)^*s(t)f_{1,k}(t) s(t)^* f_{1,j}(t) \\
&= \frac{1}{t^2(1-t)^2} \sum_{k=1}^n s(t) f_{k,1}(t) \left ( 1- \sum_{m=1}^n f_{m,m}(t) \right ) f_{1,k}(t) s(t)^* f_{1,j}(t) \\
&= \frac{1}{t^2(1-t)^2} \sum_{k=1}^n s(t) ( f_{k,k}(t)^2 - f_{k,k}(t)^3 ) s(t)^* f_{1,j}(t) \\
&= \frac{1}{t^2(1-t)^2}s(t) ( (1-s(t)^*s(t))^2 - (1-s(t)^*s(t))^3 ) s(t)^* f_{1,j}(t) \\
&= \frac{1}{t^2(1-t)^2} \cdot t(1-t)^2 s(t) s(t)^* f_{1,j}(t) \\
&= \frac{1}{t}  s(t) s(t)^* f_{1,j}(t)
\end{align*}
and, using that $s^*s$ commutes with $f_{1,1}$ and that $f_{m,m} f_{1,j} = 0$ for $m \neq 1$,
\begin{align*} 
(1-t)d_1^*d_j &=  \frac{1}{t(1-t)} f_{1,1}(t)s^*(t)s(t)f_{1,j}(t) \\
&= \frac{1}{t(1-t)}s^*(t)s(t)  f_{1,1}(t) f_{1,j}(t) \\
&= \frac{1}{t(1-t)}s^*(t)s(t) \left ( \sum_{m=1}^n f_{m,m}(t) \right ) f_{1,j}(t) \\
&= \frac{1}{t(1-t)}s^*(t)s(t) \left (1 - s(t)^*s(t) \right ) f_{1,j}(t) \\
&= \frac{1}{t(1-t)} \cdot (1-t) s^*(t)s(t) f_{1,j}(t) \\
&= \frac{1}{t} s^*(t)s(t) f_{1,j}(t)
\end{align*}
Combining those observations, we get
\begin{align*}
\sum_{k=1}^n c_{k,1}^*c_{k,j} + (1-t)d_1^*d_j &= \frac{1}{t} \left ( s(t) s(t)^* +  s^*(t)s(t) \right ) f_{1,j}(t) \\
&=  \frac{1}{t} \left ( s(t) s(t)^* +  s^*(t)s(t) +  \sum_{m=2}^n f_{m,1}(t)s(t)s(t)^*f_{1,m}(t) \right ) f_{1,j}(t)\\
&= \frac{1}{t} b(t) f_{1,j}(t) \\
&=  f_{1,j}(t)
\end{align*}
This proves the claim.

It follows that the pair $(\xi,s)$ as in the statement gives rise to a unital homomorphism
$\pi\colon I_{n,n+1}\to B$; it remains to check that $\pi$ is equivariant.
To lighten the notation, we use the same letter $s$ to denote the given element in $B$ and the element $s$ in the universal \ca~ $I_{n,n+1}$.

Note that $\beta$ leaves $\pi(I_{n,n+1})$ invariant. Furthermore, using that $v_ge_{1,j} = e_{1,j}$ for $j=1,\ldots,n$ and for all $g\in G$, it follows that
\[
\beta_g(\pi(b)) = s^*s + ss^* + \sum_{j=2}^n \xi(v_ge_{j,1})ss^*\xi(e_{1,j}v_g^*)
\]
for all $g\in G$. Since $v_g$ is unitary and $\xi$ is linear, we deduce that $\beta_g(\pi(b)) = \pi(b)$ for all $g\in G$.
Thus, the restriction of $\beta$ to $\pi(I_{n,n+1})$ is an action via $C([0,1])$-automorphisms. It follows that it suffices to check equivariance on each fiber.

Define finite-dimensional Hilbert spaces $H_0, H_1$ and $H_2$ via
$H_0 = \textrm{span}\{\pi(c_{1,1})\}$,
\[H_1 = \textrm{span}\{\pi(c_{j,k})\colon j,k =1,\ldots,n\},\] 
and  
\[H_2 = \textrm{span}\{\pi(d_l)\colon l=1,\ldots,n\}.\]
Then $H_0$, $H_1$ and $H_2$ are invariant under $\beta$.
Set $E = \mathrm{span}\{e_{1,1},e_{1,2},\ldots,e_{1,n}\}$. Note that there are natural isomorphisms $H_1 \cong E \otimes E$ and $H_2 \cong E$, the first one given by
identifying $e_{1,j} \otimes e_{1,k}$ with $c_{j,k}$, and the second one given by identifying $e_{1,k}$ with $d_k$, for $j,k=1,\ldots,n$.
With these identifications, $\beta$ acts as $v_g \otimes v_g$ on $H_1$ while leaving $H_0$ fixed, and acts as $v_g$ on $H_2$.
Thus, the action induced by $\beta$ on the fiber corresponding to some $t\in (0,1)$ is conjugate to $\Ad (v \otimes (v \oplus 1_{\C}))$. The end-cases $t=0$ and $t=1$ are verified
similarly, thus concluding the proof.
\end{proof}

We will apply \autoref{thm:univprop} in the proof of \autoref{thm:dimRokatmost2}, at the end of this section,
to representations $v$ of the form $\lambda^{\otimes k}\colon G\to \U(\ell^2(G)^{\otimes k})$, for $k\in\N$.
Indeed, if $e\in \B(\ell^2(G))$ denotes the projection onto the constant functions, then $e^{\otimes k}$ is a $G$-invariant rank-one projection,
thus showing that $\lambda^{\otimes k}$ satisfies the assumption in \autoref{thm:univprop}.

\begin{prop}\label{prop:GactIkdim2}
Let $\ep>0$ and let $G$ be a finite group. Then there exist $k\in\N$ and positive contractions $f_g^{(j)}\in I_G^{(k)}$, for
$g\in G$ and $j=0,1,2$, satisfying
\be
\item[(a)] $\|\mu^{(k)}_{g}(f^{(j)}_{h})-f^{(j)}_{gh}\|<\ep$ for all $j=0,1,2$, and for all $g,h\in g$;
\item[(b)] $\|f^{(j)}_{g}f^{(j)}_{h}\|<\ep$ for all $j=0,1,2$, and for all $g,h\in G$ with $g\neq h$;
\item[(c)] $\left\|1-\sum\limits_{j=0}^2\sum\limits_{g\in G}f^{(j)}_{g}\right\|<\ep$.
\ee
\end{prop}
\begin{proof}
Since the action $\bigotimes\limits_{n\in\N}\Ad(\lambda^{\otimes n}\oplus 1)\colon G\to \bigotimes\limits_{n\in\N}\B(\ell^2(G)^{\otimes n}\oplus \C)$ has
Rokhlin dimension 1 by \autoref{prop:DimRok1}, we can find $k\in\N$ and positive contractions
$\widetilde{f}_g^{(j)}\in \B(\ell^2(G)^{\otimes k}\oplus \C)$, for $g\in G$ and $j=0,1$, satisfying
\be\item[(i)] $\|\Ad(\lambda^{\otimes k}_g\oplus 1)(\widetilde{f}^{(j)}_{h})-\widetilde{f}^{(j)}_{gh}\|<\ep$ for all $j=0,1$, and for all $g,h\in g$;
\item[(ii)] $\|\widetilde{f}^{(j)}_{g}\widetilde{f}^{(j)}_{h}\|<\ep$ for all $j=0,1$, and for all $g,h\in G$ with $g\neq h$;
\item[(iii)] $\left\|1-\sum\limits_{g\in G}(\widetilde{f}^{(0)}_g+ \widetilde{f}^{(1)}_{g})\right\|<\ep$.
\ee

For $g\in G$, denote by $p_g\in \B(\ell^2(G))$ the projection onto the span of $\delta_g\in \ell^2(G)$. Then
$\Ad(\lambda_g)(p_h)=p_{gh}$ for all $g,h\in G$, and $\sum\limits_{g\in G}p_g=1$. Regard $p_g$ as an element in
$\B(\ell^2(G)^{\otimes k})\cong \B(\ell^2(G))^{\otimes k}$ via the first factor embedding.

Let $h_0\in C([0,1])$ denote the inclusion of $[0,1]$ into $\C$, and set $h_1=1-h_0\in C([0,1])$. For $g\in G$
and $j=0,1,2$, set
\[f_g^{(j)}=\begin{cases} h_0\widetilde{f}_g^{(j)} &\mbox{if } j=0,1; \\
h_1p_g & \mbox{if } j=2.\end{cases}\]

One checks that conditions (a), (b) and (c) in the statement are satisfied, completing the proof.
\end{proof}

We will need an equivariant version of Sato's property (SI).

\begin{df}\label{df:SI}
Let $G$ be a countable amenable group, let $A$ be a unital separable \ca~with non-empty trace space, and let $\alpha\colon G\to\Aut(A)$ be an action. Let $\omega$ be an ultrafilter.
We say that $(A,\alpha)$ has the \emph{equivariant property (SI)} if for any $\omega$-central sequences $(x_n)_{n\in\N}$ and $(y_n)_{n\in\N}$
of positive contractions in $A$ satisfying
\[\lim_{n\to\omega} \max_{\tau\in T(A)} \tau(x_n)=0, \ \ \inf_{m\in\N}\lim_{n\to\omega} \min_{\tau\in T(A)} \tau(y^m_n)>0, \ \ \mbox{ and } \]
\[\lim_{n\to \omega} \|\alpha_g(x_n)-x_n\|=\lim_{n\to\omega}\|\alpha_g(y_n)-y_n\|=0\] 
for all $g\in G$,
then there exists an $\omega$-central sequence $(s_n)_{n\in\N}$ of contractions in $A$ such that, for all $g\in G$, we have
\[\lim_{n\to\omega}\|s_n^*s_n-x_n\|=0, \ \ \lim_{n\to\omega}\|y_ns_n-s_n\|=0, \ \ \mbox{ and } \ \ \lim_{n\to \omega} \|\alpha_g(s_n)-s_n\|=0.\]
\end{df}

It is easy to check that the definition above is independent of the ultrafilter $\omega$.
When $G$ acts trivially on $A$, the equivariant property (SI) reduces to Sato's property (SI); see Definition~3.3 in~\cite{Sat_rohlin_2010}.
The following result shows that property (SI) implies the equivariant (SI) whenever the action has the weak tracial Rokhlin property, or
whenever $A$ is nuclear and $\partial_eT(A)$ is compact and finite-dimensional.

\begin{prop}\label{prop:EqSI}
Let $G$ be a countable amenable group, let $A$ be a separable, simple, unital \ca\ with
nonempty trace space, and let $\alpha\colon G\to\Aut(A)$ be an action. Assume that
$A$ has property (SI), and assume that at least one of the following conditions holds:
\be\item $\alpha$ has the weak tracial Rokhlin property.
\item $A$ is nuclear and $T(A)$ is a Bauer simplex with $\dim_{\mathrm{cov}}(\partial_eT(A))<\I$.
\ee
Then $(A,\alpha)$ has the equivariant property (SI).
\end{prop}
\begin{proof}
(1). Let $(x_n)_{n\in\N}$ and $(y_n)_{n\in\N}$ be sequences of positive contractions in $A$
as in \autoref{df:SI}.
Let $K\subseteq G$ be a finite set and let $\ep>0$.

\textbf{Claim:} \emph{there exist a $(K,\epsilon)$-invariant finite subset $S\subseteq G$ and an $\omega$-central sequence $(z_n)_{n\in\N}$
of positive contractions in $A$ satisfying the following conditions for all $g,h\in S$ with $g\neq h$:
\[z_n\leq y_n, \ \ \inf_{m\in\N}\lim_{n\to\omega}\min_{\tau\in T(A)} \tau(z_n^m)>0, \ \ \mbox{ and } \ \
 \lim_{n\to\omega} \alpha_g(z_n)\alpha_h(z_n)=0.
\]}

Find a paving system
$S_1,\ldots,S_N$ for $G$ such that each $S_\ell$ is $(K,\ep)$-invariant; see \autoref{thm:OWPavSyst}.
Let $f_{\ell, g}\in A_\omega\cap A'$, for $\ell=1,\ldots,N$ and $g\in S_\ell$,
be positive contractions as in the definition of the weak tracial \Rp\ for $\alpha$.
For $\ell=1,\ldots,N$, set $\kappa_\ell=\tau(f_{\ell,g})>0$
for some $g\in S_\ell$ and some $\tau\in T(A)$ (see condition (e) in \autoref{df:wtRp}).
Observe that
\[\tau(1)=\sum_{\ell=1}^N \sum_{g\in S_\ell}\tau(f_{\ell,g})= \sum_{\ell=1}^N\kappa_\ell|S_\ell|.\]

Set
\[\delta=\inf\limits_{m\in\N}\lim\limits_{n\to\omega} \min\limits_{\tau\in T(A)} \tau(y^m_n)>0 \ \ \mbox{ and } \ \ \kappa=\min_{\ell=1,\ldots,N}\kappa_\ell >0.\]
In order to prove the claim, it suffices to show that given $m\in\N$ and a finite subset $F\subseteq A$, there is a sequence $(z_n)_{n\in\N}$ of
positive contractions in $A$ satisfying
\bi
\item $\lim_{n\to\omega}\|z_na-az_n\|<\frac{1}{m}$ for all $a\in F$;
\item $z_n\leq y_n$ for all $n\in\N$;
\item $\lim_{n\to\omega} \|\alpha_g(z_n)\alpha_h(z_n)\|<\frac{1}{m}$ for all $g,h\in S$ with $g\neq h$; and
\item $\lim_{n\to\omega}\min_{\tau\in T(A)} \tau(z_n^m)>\kappa\delta-\frac{1}{m}$.
\ei

Fix $m\in\N$ and a finite subset $F\subseteq A$.
Let $y\in A_\omega\cap A'$ be the equivalence class of $(y_n)_{n\in\N}$.
Then
\[\tau(y^m)=\sum_{\ell=1}^N \sum_{g\in S_\ell}\tau(y^mf_{\ell,g}).\]
It follows that there exist $\ell_0=1,\ldots,N$ and $g_0\in S_{\ell_0}$ such that
\[\tau(y^m f_{\ell_0,g_0})\geq \tau(y^m)\kappa_{\ell_0}\geq \tau(y^m)\kappa.\]

Set $S=S_{\ell_0}$. Without loss of generality, we
can assume that the unit $e$ of $G$ belongs to $S$.
Set $f=f_{\ell_0,e}$, and let $(f_k)_{k\in\N}$ be a sequence of positive contractions
in $A$ representing $f$.

Choose $k_m\in\N$ large enough so that the following conditions are satisfied for all $g,h\in S$ with $g\neq h$ and for all $a\in F$:
\[\|\alpha_g(f_{k_m}^{1/m})\alpha_h(f_{k_m}^{1/m})\|<\frac{1}{m}, \ \ \|f_{k_m}^{1/m}a-af_{k_m}^{1/m}\|<\frac{1}{m}, \ \mbox{ and } \ \min_{\tau\in T(A)}\tau(f_{k_m})>\kappa-\frac{1}{m\delta}.\]
For $n\in\N$, set
\[z_n=y_n^{1/2}f_{k_m}^{1/m}y_n^{1/2} \in A^\omega.\]
The first three items above are routinely verified, so we only check the last one:
\[\lim_{n\to\omega}\min_{\tau\in T(A)} \tau(z_n^m)=\lim_{n\to\omega}\min_{\tau\in T(A)} \tau(y_n^mf_{k_m})\geq \lim_{n\to\omega}\min_{\tau\in T(A)} \tau(y_n^m)(\kappa-\frac{1}{m\delta})>\delta\kappa-\frac{1}{m},\]
as desired. This proves the claim.

Use property (SI) for $A$ with $(x_n)_{n\in\N}$ and $(z_n)_{n\in\N}$ to find an $\omega$-central sequence $(r_n)_{n\in\N}$
of positive contractions in $A$ such that
\[\lim_{n\to\omega}\|r_n^*r_n-x_n\|=0 \ \ \mbox{ and } \ \ \lim_{n\to\omega}\|z_nr_n-r_n\|=0.\]
For $n\in\N$, set $s_n=\frac{1}{\sqrt{|S|}} \sum\limits_{g\in S}\alpha_g(r_n)$. We will show that $(s_n)_{n\in\N}$ satisfies the
conditions in \autoref{df:SI}. We use that $\lim\limits_{n\to\omega}\|\alpha_g(r_n)-r_n\|=0$ at the second step to get
\begin{align*}
\lim_{n\to\omega}\|y_ns_n-s_n\|= \lim_{n\to\omega} \left\|\sum_{g\in S}\alpha_g(y_nr_n-r_n)\right\| =0.
\end{align*}
On the other hand,
\begin{align*}
\lim_{n\to\omega}\|x_n-s_n^*s_n\|&= \lim_{n\to\omega} \left\|x_n-\frac{1}{|S|}\sum_{g,h\in S} \alpha_g(r_n^*)\alpha_h(r_n)\right\|\\
&= \lim_{n\to\omega} \left\|x_n-\frac{1}{|S|}\sum_{g,h\in S} \alpha_g(r_n^*z_n)\alpha_h(z_nr_n)\right\|\\
&=\lim_{n\to\omega} \left\|\frac{1}{|S|}\sum_{g\in S} \alpha_g(x_n-r_n^*z_nz_nr_n)\right\|\\
&=\lim_{n\to\omega} \left\|\frac{1}{|S|}\sum_{g\in S} \alpha_g(x_n-r_n^*r_n)\right\|=0.
\end{align*}

Finally, since $S$ is $(K,\epsilon)$-invariant, it is easy to verify that $\lim\limits_{n\to\omega}\|\alpha_k(s_n)-s_n\|<\ep$
for all $k\in K$. Since $K\subseteq G$ and $\ep>0$ are arbitrary, an application of Kirchberg's $\ep$-test shows that
$(A,\alpha)$ has the equivariant property (SI).

(2). This follows by combining Proposition~4.4 and Proposition~5.1 in~\cite{Sat_actions_2016}.
\end{proof}

For a fixed group $G$ and a fixed normal subgroup $H$ of finite index, 
Ocneanu's model action $\mu_{G/H}$ (see \autoref{eg:FiniteGModelStrMcDuff} and \autoref{eg:muG}) 
can be realized as the weak extension of $\bigotimes\limits_{n\in\N}\Ad(\lambda_{G/H})$. With $\pi_H\colon G\to G/H$ denoting the canonical
quotient map, the composition $\mu_{G/H}\circ\pi_H$ is an action of 
$G$ on $\R$, which we will abbreviate to $\mu^G_H$. This action is easily
seen to be McDuff. Moreover, since $\mu^G_H\otimes \mu_G$ is an 
outer action of $G$ on $\R$, it is 
cocycle conjugate to $\mu_G$. In particular, any $\mu_G$-McDuff action
is automatically $\mu_H^G$-McDuff.

We are now ready to prove \autoref{thm:dimRokatmost2}.

\begin{proof}[Proof of \autoref{thm:dimRokatmost2}.]
We adopt the notation from \autoref{eg:EqW*bundle}. In particular, we denote by $(\M,\gamma)$ the equivariant $W^*$-bundle obtained
from $(A,\alpha)$.
To lighten the notation, we write $\M_\R$ for the tensor product
$W^*$-bundle $\M\overline{\otimes}\R$, and we write
$\gamma_\R=\gamma\otimes\id_\R$, which is an action of $G$ on
$\M_\R$. Similarly, we write $A_{\mathcal{Z}}$ for $A\otimes\mathcal{Z}$
and $\alpha_{\mathcal{Z}}=\alpha\otimes\id_{\mathcal{Z}}$. Note that
$(\M_\R,\gamma_\R)$ is then the equivariant $W^*$-bundle naturally 
obtained from $(A_\mathcal{Z},\alpha_{\mathcal{Z}})$.

Let $H\leq G$ be a normal subgroup of finite index, and adopt
the notation introduced before this proof.
By \autoref{prop:FiberW*Rp}, the fiber-action $\gamma^{\lambda}\colon G\to\Aut(\mathcal{M}_\lambda)$ is $\mu_G$-McDuff for
every $\lambda\in\partial_eT(A)/G$.
In particular, $\gamma^{\lambda}$ is $\mu^G_H$-McDuff.
It thus follows from \autoref{thm:MoRabsorbsMcDuff} 
that $\gamma_\R=\gamma\otimes \id_{\mathcal{R}}$ is cocycle conjugate to $\gamma\otimes \mu^G_H$.
Hence by \autoref{thm:W*BdleAbsR}, there is a unital, equivariant homomorphism
\[\psi\colon (\R,\mu^G_H)\to \left(\M_\R^\omega\cap\M_\R',\gamma_\R^\omega\right).\]

By Lemma~3.10 in \cite{BBSTWW_covering_2015} (see Theorem~3.1 
in~\cite{MatSat_decomposition_2014} for the case of one trace), 
there is a natural surjective equivariant homomorphism
\[\pi\colon (A_{\mathcal{Z}})_\omega\cap A_{\mathcal{Z}}'\to \mathcal{M}^\omega\cap \mathcal{M}'.\]
Let $\ep>0$, and let $k\in\N$ be as in the conclusion of \autoref{prop:GactIkdim2}.
Denote by 
\[\varphi\colon \left(\B(\ell^2(G/H)^{\otimes k}),\Ad(\lambda_{G/H}^{\otimes k})\circ \pi_H\right) \to \left(\M_\R^\omega\cap \M_\R',\gamma_\R^\omega\right)\] 
the restriction of $\psi$ to $\B(\ell^2(G/H)^{\otimes k})\subseteq \R$.
Since $J_A$ is an equivariant $\sigma$-ideal, there exists
a completely positive contractive equivariant order zero map 
\[\rho\colon \left(\B(\ell^2(G/H)^{\otimes k}),\Ad(\lambda_{G/H}^{\otimes k})\right) \to ((A_{\mathcal{Z}})_\omega\cap A_{\mathcal{Z}}',\alpha_\omega)\]
making the following diagram commute:
\begin{align*}
\xymatrix{ && (A_{\mathcal{Z}})_\omega\cap A_{\mathcal{Z}}'\ar[d]^-{\pi}\\
\B(\ell^2(G/H)^{\otimes k})\ar[urr]^-{\rho}\ar[rr]_-{\varphi} && \mathcal{M}_\R^\omega\cap \mathcal{M}_\R'.}
\end{align*}

We denote by $e\in \B(\ell^2(G/H))$ the projection onto the constant functions, and regard $e^{\otimes k}$ as a projection
in $\B(\ell^2(G/H)^{\otimes k})$. One can easily verify that $\tau_\omega(\rho^m(e^{\otimes k}))=1/[G:H]^k$ for all $\tau\in T(A)=T(A_{\mathcal{Z}})$ and for all $m\in\N$.

Since $\alpha_{\mathcal{Z}}=\alpha\otimes\id_{\mathcal{Z}}$ has the weak tracial Rokhlin property by
\autoref{thm:wtRp}, part~(a) of \autoref{prop:EqSI} implies that there exists a contraction
$s\in ((A_{\mathcal{Z}})_\omega\cap A_{\mathcal{Z}}')^{\alpha_\omega}$ satisfying $s^*s=1-\rho(1)$
and $\rho(e^{\otimes k})s=s$. By \autoref{thm:univprop}, there exists
a unital equivariant homomorphism 
\[\theta\colon (I_{G/H}^{(k)},\mu^{(k)})\to ((A_{\mathcal{Z}})_\omega\cap A_{\mathcal{Z}}',(\alpha_\mathcal{Z})_\omega).\]
If $f_{\overline{g}}^{(j)}\in I_{G/H}^{(k)}$, for ${\overline{g}}\in G/H$ and $j=0,1,2$, are positive contractions as in the conclusion of
\autoref{prop:GactIkdim2}, the positive contractions $\theta(f_{\overline{g}}^{(j)}) \in (A_{\mathcal{Z}})_\omega\cap A_{\mathcal{Z}}'$ satisfy the conditions
of \autoref{df:Rdim} up to $\ep$. Since $\ep>0$ is arbitrary, the result follows using a reindexation argument
(or from countable saturation of the equivariant ultrapower of $A_{\mathcal{Z}}$; see Subsection~2.2.4 in~\cite{GarLup_applications_2018}).
\end{proof}

\section{Equivariant \texorpdfstring{$\mathcal{Z}$}{Z}-stability of amenable actions}

In this section, we prove Theorem~\ref{thmintro:Zabs}, whose statement
we reproduce below.

\begin{thm}\label{thm:EqJiangSuAbs}
Let $G$ be a countable amenable group, let $A$ be a separable, simple, \uca\ with property (SI), 
and let $\alpha\colon G\to\Aut(A)$ be an action.
Suppose that $A$ is stably finite, that $T(A)$ is a nonempty Bauer simplex, that $\dim \partial_eT(A)<\I$,
that $\overline{A}^\tau$ is McDuff for all $\tau\in\partial_eT(A)$,
and that the induced action of $G$ on $\partial_eT(A)$ has finite orbits and Hausdorff orbit space.
Then $\alpha$ is strongly cocycle conjugate to
$\alpha\otimes\id_{\mathcal{Z}}$.
\end{thm}

Observe that we do not assume any form of outerness in the theorem above. In particular,
it reveals new information in the case of strongly
self-absorbing actions; see \autoref{cor:SSAActZstable}. It also contains and extends Sato's
recent result from \cite{Sat_actions_2016}
regarding $\mathcal{Z}$-stability of the crossed product, see \autoref{cor:CPisZstable}.
It should be pointed out that our methods, which
rely on the study of equivariantly McDuff $W^*$-bundles carried out in Section~3, are quite different from Sato's.

This result is a crucial ingredient in upcoming work of the first named author, Phillips, and Wang \cite{GarPhiWan_absorption_2017},
where it is shown that an action of an amenable group $G$ is strongly outer if and only if it absorbs a canonical \emph{model}
action of $G$ on $\mathcal{Z}$.

\autoref{thm:EqJiangSuAbs} can be regarded as an equivariant analog of the results of Kirchberg
and R{\o}rdam in~\cite{KirRor_central_2014} for the nonequivariant setting.
Our proof follows, broadly speaking, a similar strategy.
Namely, we first show that the $W^*$-bundle $(M,\gamma)$ obtained from $(A,\alpha)$ absorbs $(\mathcal{R},\id_{\mathcal{R}})$ equivariantly.
This gives (and, in fact, is equivalent to having) a unital homomorphism $\rho\colon M_k\to (\mathcal{M}^\omega\cap \mathcal{M}')^{\gamma^\omega}$ for
some $k\geq 1$; see \autoref{thm:W*BdleAbsR}. Property (SI), in combination with the weak tracial Rokhlin property for $\alpha$ (established
in \autoref{thm:wtRp}), yields an equivariant form of property (SI) which is used to extended $\rho$ to a unital
homomorphism from the dimension drop algebra $I_{k,k+1}$ into
$(A_\omega\cap A')^{\alpha_\omega}$. It is then easy to conclude from this that $\alpha$ absorbs $\id_{\mathcal{Z}}$ tensorially.

We begin with a proposition that may be interesting in its own right.

\begin{prop}\label{prop:FixedPtAlgsSurjective}
Let $G$ be a countable discrete amenable group, let $A$ be a separable \ca, and let $\alpha\colon G\to\Aut(A)$ be an action.
If $\omega$ is any free ultrafilter over $\N$, then there is a conditional expectation $E_{\omega}\colon A_\omega\to (A_\omega)^{\alpha_{\omega}}$
satisfying $E_\omega(A_\omega\cap A')=(A_\omega\cap A')^{\alpha_{\omega}}$.

As a consequence, if $\beta\colon G\to\Aut(B)$ is an action of $G$ on a \ca\ $B$, and $\pi\colon A_\omega\cap A'\to B$ is an equivariant surjective
homomorphism, then $\pi$ restricts to a surjective homomorphism $(A_\omega\cap A')^{\alpha_\omega}\to B^\beta$.
\end{prop}
\begin{proof}
Using amenability of $G$, choose a F\o lner sequence $(F_n)_{n\in\N}$. Let $P \colon \ell^{\infty}(A) \to \ell^{\infty}(A)$ be the unital completely positive map given by
\[P(a)_n = \frac{1}{|F_n|}\sum_{g\in F_n}\alpha_g(a_n)\]
for all $a\in \ell^\I(A)$ and all $n\in\N$.

By a slight abuse of notation, we denote also by $\alpha$ the induced action on $\ell^{\infty}(A)$.
Note that for any $g \in G$ and any $a\in \ell^{\infty}(A)$, the difference
$\alpha_g(P(a)) - P(a)$ belongs to $c_\omega(A)$.  Furthermore, $c_{\omega}(A)$ is invariant under $P$.

Denote by $q \colon \ell^{\infty}(A) \to A_{\omega}$ the quotient map.
Let $(z_k)_{k\in\N}$ be an enumeration of a dense subset of the unit ball of $A$.
Let $L\colon A_{\omega} \to \ell^{\infty}(A)$ be a (possibly non-linear) lifting for $q$ with $\|L(a)\| = \|a\|$ for all $a \in A_{\omega}$,
satisfying
\begin{enumerate}
    \item If $a \in (A_\omega)^{\alpha_{\omega}}$ then for all $m\in\N$ and for all $g \in F_m$ we have $\|\alpha_g(L(a)_m) - L(a)_m\| < 1/m$.
    \item If $a \in A_{\omega} \cap A'$ then for all $m\in\N$, for all $k=1,\ldots,m$, and for all $g \in F_m$ we have
    $\|[L(a)_m,\alpha_{g^{-1}}ß(z_k)]\| < 1/m$.
\end{enumerate}
(A lift $L$ satisfying the above conditions is easy to construct, defining it element by element, since we do not
even require $L$ to be linear or continuous.)

We now set $E_{\omega} = q \circ P \circ L$. Since for any $x,y \in A_{\omega}$ and any $\lambda_1,\lambda_2 \in \C$, the differences
\[L(\lambda_1x + \lambda_2y) - \lambda_1L(x) - \lambda_2L(y), L(xy) - L(x)L(y), \mbox{ and } L(x^*)-L(x)^*\]
belong to $c_{\omega}(A)$, it follows that $E_{\omega}$ is unital and that the image
of $E_{\omega}$ is (pointwise) fixed by the group action.

We claim that $E_{\omega}$ is completely positive. To this end, we can decompose $q \circ P$ differently as follows. Let
$C = \prod_{m=1}^{\infty} C(F_m,A)$ (that is, the \ca ~of bounded sequences whose $m$'th term is in $C(F_m,A)$).
For $m \in \N$, let $\theta_m \colon A \to C(F_m,A)$ be given by $\theta_m(a)(g) = \alpha_g(a)$ for all $a\in A$ and all
$g\in F_m$. Now let $\theta \colon \ell^{\infty}(A) \to C$ be given by $\theta(a)_m = \theta_m(a_m)$ for all
$a\in \ell^\I(A)$ and all $m\in\N$.
Denote by $\prod_{\omega}C(F_m,A)$ the associated ultraproduct\footnote{Recall that if $(A_n)_{n\in\N}$ is a sequence of \ca s and $\omega$ is a
free ultrafilter, then their ultraproduct $\prod_\omega A_n$ is defined as the quotient of $\prod_{n=1}^{\infty}A_n$ by the ideal
$J=\{(a_n)_{n\in\N}\in \prod_{n=1}^{\infty}A_n\colon \lim_{\omega} \|a_n\|=0\}$.}, and let
$r \colon C \to \prod_{\omega}C(F_m,A)$ be the quotient map.
Notice that $\theta$ is a homomorphism, and that $\theta(c_{\omega}(A)) \subseteq \ker(r)$.
Thus,
\[r \circ \theta \circ L \colon A_{\omega} \to \prod_{\omega}C(F_m,A)\]
is a unital homomorphism. Now, let $T \colon C \to \ell^{\infty}(A)$ be given by
$T(f)_m = \frac{1}{|F_m|}\sum_{g\in F_m}f_m(g)$ for all $f\in C$ and all $m\in\N$.
Then $T$ is a unital completely positive map, and $P = T \circ \theta$.
Furthermore, as $T(\ker(r)) \subseteq c_{\omega}(A)$, the map $T$ induces a unital
completely positive map
\[\tilde{T} \colon  \prod_{\omega}C(F_m,A) \to A_{\omega}.\]
Note that
\[\tilde{T} \circ r \circ \theta = q \circ T \circ \theta = q \circ P.\]
It follows
that $q \circ P \circ L$ is a composition of a completely positive map and a homomorphism,
thus completely positive as claimed.

Fix $a \in (A_\omega)^{\alpha_{\omega}}$. Then $\|P(L(a))_m - L(a)_m\|<1/m$ for all $m\in\N$, and
thus $P(L(a)) - L(a)$ belongs to $c_0(A) \subseteq c_{\omega}(A)$. We conclude that $E_{\omega}(a) = a$.

Fix $a \in A_{\omega} \cap A'$. In order to show that $E_{\omega}(a) \in A_{\omega} \cap A'$,
it suffices to show that $[P(L(a)),z_k] \in c_{\omega}(A)$ for all $k \in \N$. To see that this is the
case, let $m \in \N$. Then $\|[P(L(a))_m,z_k]\|<1/m$ whenever $k\leq m$, so
$[P(L(a)),z_k] \in  c_0(A) \subseteq c_{\omega}(A)$.

Now let $\beta\colon G\to\Aut(B)$ be an action of $G$ on a \ca\ $B$, and let $\pi\colon A_\omega\cap A'\to B$ be an equivariant surjective
homomorphism. Given $b\in B^\beta$ choose $a\in A_\omega\cap A'$ satisfying $\pi(a)=b$.
Equivariance of $\pi$ implies that $\pi(E_{\omega}(a))=\pi(a)=b$, as desired.
\end{proof}

The following lemma is well known in the non-equivariant setting. In our context, it follows from the fact that equivariant
ultrapowers are (countably) saturated; see Subsection~2.2.4 in~\cite{GarLup_applications_2018}.

\begin{lma}\label{lma:MapIntoSeqAlg}
Let $G$ be a discrete group, and let $(A,\alpha)=\varinjlim (A_n,\alpha^{(n)})$ be an equivariant direct limit
with unital equivariant connecting maps $A_n\to A_{n+1}$. Let $B$ be any \uca\ and let $\omega$ be a free ultrafilter
over $\N$. Then there is a unital equivariant homomorphism $A\to B_\omega\cap B'$ if and only if for every $n\in\N$
there exists a unital equivariant homomorphism $A_n\to B_\omega\cap B'$.
\end{lma}

Our last ingredient will be the verification that amenable group actions on finite sums of the hyperfinite II$_1$-factor
are equivariantly McDuff.

\begin{prop}\label{prop:FinDirSumRMcDuff}
Let $M$ be a finite direct sum of copies of $\mathcal{R}$, let $\alpha\colon G\to\Aut(M)$ be an action of
an amenable group $G$ on $M$, and let $\omega$ be a free ultrafilter on $\N$.
Then there exists a unital homomorphism $M_2\to (M^\omega\cap M')^{\alpha^\omega}$.
\end{prop}

We are now ready to prove equivariant Jiang-Su absorption.

\begin{proof}[Proof of \autoref{thm:EqJiangSuAbs}.]
We adopt the notation from \autoref{eg:EqW*bundle}. In particular, we denote by $(\M,\gamma)$ the equivariant $W^*$-bundle obtained
from $(A,\alpha)$. By \autoref{prop:FinDirSumRMcDuff} and
\autoref{thm:W*BdleAbsR}, the fiber-action $\gamma^{\lambda}\colon G\to\Aut(\mathcal{M}_\lambda)$ is $\id_{\mathcal{R}}$-McDuff, for
every $\lambda\in\partial_eT(A)/G$.
By \autoref{cor:DimKGamenable}, it follows that $\gamma$ itself is cocycle conjugate to $\gamma\otimes \id_{\mathcal{R}}$.
By \autoref{thm:W*BdleAbsR}, this is equivalent to the existence, for some $k\geq 2$ and some free ultrafilter $\omega$, of a unital homomorphism
\[\varphi\colon M_k\to (\mathcal{M}^\omega\cap \mathcal{M}')^{\gamma^\omega}.\]
By \autoref{prop:FixedPtAlgsSurjective}, there is a natural surjective homomorphism
$\pi\colon (A_\omega\cap A')^{\alpha_\omega}\to (\mathcal{M}^\omega\cap \mathcal{M}')^{\gamma^\omega}$. By projectivity of $C_0([0,1))\otimes M_k$,
there exists a completely positive contractive order zero map $\rho\colon M_k\to (A_\omega\cap A')^{\alpha_\omega}$ making the following diagram
commute:
\begin{align*}
\xymatrix{ && (A_\omega\cap A')^{\alpha_\omega}\ar[d]^-{\pi}\\
M_k\ar[urr]^-{\rho}\ar[rr]_-{\varphi} && (\mathcal{M}^\omega\cap \mathcal{M}')^{\gamma^\omega}.}
\end{align*}

It is then immediate to check that $\tau_\omega(\rho^m(e_{1,1}))=1/k$ for all $\tau\in T(A)$ and for all $m\in\N$.
For $j=1,\ldots,k$, set $c_j=\rho(e_{1,j})$, which is a contraction in
$(A_\omega\cap A')^{\alpha_\omega}$. It is clear that these elements
satisfy
\[c_1\geq 0 \ \ \mbox{ and } \ \ c_ic_j^*=\begin{cases}
  c_1^2 & \text{if}\ i=j \\
  0           & \text{else},
\end{cases}\]
for $1\leq i,j\leq k$. By part~(b) of \autoref{prop:EqSI}, there exists a contraction
$s\in (A_\omega\cap A')^{\alpha_\omega}$ satisfying $s^*s=1-\sum\limits_{j=1}^k c_j^*c_j$
and $c_1s=s$. By Proposition~5.1 in~\cite{RorWin_algebra_2010}, there exists
a unital homomorphism $I_{k,k+1}\to (A_\omega\cap A')^{\alpha_\omega}$.
By \autoref{lma:MapIntoSeqAlg}, there exists a unital homomorphism $\mathcal{Z}\to (A_\omega\cap A')^{\alpha_\omega}$.
Finally, using virtually the same argument as in
\autoref{thm:W*BdleAbsR}, the existence of such a unital homomorphism is equivalent to
$\alpha$ and $\alpha\otimes\id_{\mathcal{Z}}$ being strongly cocycle
conjugate; see, for example, Theorem~4.8 in~\cite{MatSat_stability_2012} and Theorem~2.6 in~\cite{Sza_strongly_2018}.
\end{proof}

It should be pointed out that amenability of $G$ is a necessary assumption in \autoref{thm:EqJiangSuAbs}, as
shown in \cite{GarLup_actions_2018}. In fact, the combination of these results implies that amenability of $G$
can be \emph{characterized} in terms of $\mathcal{Z}$-absorption.

\begin{cor}
Let $G$ be a discrete group. Then $G$ is amenable if and only if every action of $G$ on $\mathcal{Z}$
absorbs $\id_{\mathcal{Z}}$ tensorially.
\end{cor}
\begin{proof} This is a combination of \autoref{thm:EqJiangSuAbs} and Theorem~4.3 in~\cite{GarLup_actions_2018}.
More specifically, the latter implies that when $G$ is not amenable, then the Bernoulli shift of $G$ on
$\bigotimes\limits_{g\in G}\mathcal{Z}\cong\mathcal{Z}$ does not absorb $\id_{\mathcal{Z}}$. \end{proof}

\autoref{thm:EqJiangSuAbs} allows us to recover and extend the main result of the recent work~\cite{Sat_actions_2016}.

\begin{cor}\label{cor:CPisZstable}
Let $G$ be a countable amenable group, let $A$ be a separable, simple, \uca\ with property (SI),
and let $\alpha\colon G\to\Aut(A)$ be an action.
Suppose that $T(A)$ is a nonempty Bauer simplex, that $\dim \partial_eT(A)<\I$,
that $\overline{A}^\tau$ is McDuff for all $\tau\in\partial_eT(A)$,
and that the induced action of $G$ on $\partial_eT(A)$ has finite orbits and Hausdorff orbit space.

Then $A$ and $A\rtimes_\alpha G$ are $\mathcal{Z}$-stable, and so is $A^\alpha$ when $G$ is finite.
\end{cor}

As a further application, we obtain an equivariant analog of Winter's theorem from \cite{Win_strongly_2011}:
strongly self-absorbing actions of discrete amenable groups are equivariantly $\mathcal{Z}$-stable. This
answers a question from \cite{Sza_stronglyII_2016} concerning unitary regularity.

\begin{thm}\label{cor:SSAActZstable}
Let $\gamma\colon G\to\Aut(\mathcal{D})$ be a strongly self-absorbing action
of a discrete amenable group $G$ on a tracial strongly self-absorbing \ca\ $\mathcal{D}$.
Then $\gamma$ is cocycle conjugate to $\gamma\otimes\id_{\mathcal{Z}}$.
In particular, $\gamma$ is unitarily regular.
\end{thm}
\begin{proof} Since $\mathcal{D}$ is $\mathcal{Z}$-stable by the main result of
\cite{Win_strongly_2011}, the result follows from \autoref{thm:EqJiangSuAbs}.
\end{proof}

As shown in Example~5.4 of~\cite{Sza_stronglyII_2016}, equivariant $\mathcal{Z}$-stability fails for
actions of locally compact amenable groups on $\mathcal{Z}$-stable \ca s -- it even fails for compact
group actions on UHF-algebras.
On the other hand, the actions constructed in \cite{GarLup_actions_2018} show that equivariant
$\mathcal{Z}$-stability is not automatic for actions of discrete nonamenable groups, even if they
act on $\mathcal{Z}$ itself.


\end{document}